\documentclass[10pt]{amsart}    
\usepackage{latexsym}
\usepackage{amssymb}
\usepackage{pstricks,pst-coil}

\numberwithin{equation}{section} 

\font\tengothic=eufm10 scaled\magstep 1
\font\sevengothic=eufm7 scaled\magstep 1
\newfam\gothicfam
          \textfont\gothicfam=\tengothic
          \scriptfont\gothicfam=\sevengothic

\newcommand{\Z}{\mathbb{Z}}
\newcommand{\R}{\mathbb{R}}

\newcommand{\Q}{\mathbb{Q}}

\newcommand {\PP}{\mathbb{P}}
\newcommand {\sphere} {\mathbb{S}}

\newcommand{\cC}{\mathcal{C}}

\newcommand{\cT}{\mathcal{T}}
\newcommand{\cF}{\mathcal{F}}

\newcommand{\xb}{{\bf x}}
\newcommand{\yb}{{\bf y}}

\DeclareMathOperator{\Ann}{Ann}

\DeclareMathOperator{\Hilb}{Hilb}

\DeclareMathOperator{\lcm}{lcm}

\DeclareMathOperator{\pd}{pd}

\DeclareMathOperator{\reg}{reg}

\DeclareMathOperator{\supp}{supp}

\DeclareMathOperator{\Tor}{Tor}

\DeclareMathOperator{\sgn}{sgn}
\DeclareMathOperator{\spe}{sp}

\DeclareMathOperator{\bip}{bip}
\DeclareMathOperator{\nonbip}{nonbip}
\DeclareMathOperator{\mindeg}{mindeg}
\DeclareMathOperator{\Susp}{\Sigma}
\DeclareMathOperator{\rect}{rect}
\DeclareMathOperator{\depol}{depol}

\newcommand{\hocolim}{{\sf hocolim}}

\DeclareMathOperator{\pnt}{\raise 0.5mm \hbox{\large\bf.}}

\newcommand \Dbip[2] {{D_{{#1},{#2}}^{\bip}}}
\newcommand \Dnonbip[1] {{D_{#1}^{\nonbip}}}
\newcommand \Gbip[2] {{G_{{#1},{#2}}^{\bip}}}
\newcommand \Gnonbip[1] {{G_{#1}^{\nonbip}}}

\newtheorem{theorem}{Theorem}[section]
\newtheorem{lemma}[theorem]{Lemma}
\newtheorem{proposition}[theorem]{Proposition}
\newtheorem{corollary}[theorem]{Corollary}
\newtheorem{conjecture}[theorem]{Conjecture}

\newtheorem{prop-def}[theorem]{Proposition and Definition}

\theoremstyle{definition}
\newtheorem{definition}[theorem]{Definition} 
\newtheorem{remark}[theorem]{Remark}
\newtheorem{example}[theorem]{Example}

\newtheorem{question}[theorem]{Question}
\newtheorem{rem-def}[theorem]{Definition and Remark}

\title[Betti numbers of monomial ideals and shifted skew shapes]
{Betti numbers of monomial ideals and shifted skew shapes}

\author[Uwe Nagel]{Uwe Nagel}
\address{Department of Mathematics,
University of Kentucky, 715 Patterson Office Tower,
Lexington, KY 40506-0027, USA }

\email{uwenagel@ms.uky.edu}

\author[Victor Reiner]{Victor Reiner}
\address{
School of Mathematics, University of Minnesota, Minneapolis, MN 55455, USA}
\email{reiner@math.umn.edu}

\thanks{Second author partially supported by NSF grant DMS-0601010.}
\thanks{First author partially supported by NSA grant H98230-07-1-0065 and by the Institute for Mathematics \& its Applications at the University of Minnesota. }

\begin{document}
\begin{abstract}
We present two new problems on lower bounds for resolution Betti
numbers of monomial ideals generated in a fixed degree.   The first
concerns any such ideal and bounds the total Betti numbers, while
the second concerns ideals that are quadratic and bihomogeneous with
respect to two variable sets, but gives a more finely graded lower
bound.

These problems are solved for certain classes of ideals that
generalize (in two different directions) the edge ideals of
threshold graphs and Ferrers graphs.  In the process,  we produce
particularly simple cellular linear resolutions for strongly stable
and squarefree strongly stable ideals generated in a fixed degree,
and combinatorial interpretations for the Betti numbers of other
classes of ideals, all of which are independent of the coefficient
field.
\end{abstract}


\maketitle

\section{Introduction and the main problems}
\label{sec-introduction}

The paper concerns minimal free resolutions of an ideal $I$ in a
polynomial ring $S=k[x_1,\ldots,x_n]$ which is generated by
monomials of a fixed degree.  Many of its results were motivated by
two new problems,  Question~\ref{nonbipartite-conjecture} and
Conjecture~\ref{bipartite-conjecture} below, which we formulate
here.

Given a squarefree monomial ideal $I$ generated in degree $d$, it has a uniquely defined minimal generating set
of monomials, indexed by a collection $K$ of $d$-subsets of $\PP:=\{1,2,\ldots\}$ in the sense that
$$
I=(x_{i_1} \cdots x_{i_d}: \{i_1,\ldots,i_d\} \in K).
$$
Define the {\it colexicgraphic order} on the $d$-subsets $\binom{\PP}{d}$  by saying that
$$
\begin{aligned}
S= \{i_1 < \cdots < i_d\}\\
S'=\{ i'_1 < \cdots < i'_d \}
\end{aligned}
$$
have
$S <_{colex} S'$ if $i_k < i'_k$ for some $k \in \{1,\ldots,d\}$ and $i_j = i'_j$ for $j= k+1,\ldots,d$.
 For example, the colex order on
$\binom{\PP}{3}$ begins
$$
\{1,2,3\} <_{colex} \{1,2,4\} <_{colex} \{1,3,4\} <_{colex} \{2,3,4\} <_{colex} \{1,2,5\}
<_{colex} \{1,3,5\}  <_{colex} \{2,3,5\}  \cdots
$$
Call $K \subset \binom{\PP}{d}$ a {\it colexsegment} if it forms an
initial segment of the colexicographic ordering, and call $J$ a {\it
colexsegment-generated} ideal if $J=(x_{i_1} \cdots x_{i_d}:
\{i_1,\ldots,i_d\} \in K)$ for a colexsegment $K$.  To state our
first problem, recall that $\beta_i(I)=\dim_k \Tor^S(I,k)$ is the
number of $i^{th}$ syzygies in a minimal free $S$-resolution of $I$.
Furthermore, say that a monomial ideal $I$ generated in degree $d$
{\em obeys the colex lower bound} if, for all integers $i$,\;
$\beta_i(I) \geq \beta_i(J)$, where $J$ is the unique
colexsegment-generated (squarefree) monomial ideal having the same
number  of minimal generators as $I$, all of degree $d$. We ask:

\begin{question}
\label{nonbipartite-conjecture} Let $I$ be any monomial ideal
generated in degree $d$.   When does it obey the colex lower bound?
\end{question}

\noindent
We should remark that the standard technique of {\it polarization} \cite[\S3.2 Method 1]{MillerSturmfels}
immediately reduces this question to the case where $I$ is itself generated by
squarefree monomials, generated in a fixed degree~$d$.

The second problem concerns the situation where $I$ is {\it quadratic}, and furthermore,
generated by quadratic monomials $x_i y_j$ which are {\it bihomogeneous} with respect to
two sets of variables within the polynomial algebra $S=k[x_1,\ldots,x_m,y_1,\ldots,y_n]$.  In this case, $I$ is
the {\it edge ideal}
$$
I=(x_i y_j: \{x_i,y_j\} \text{ an edge of }G)
$$
for some {\it bipartite graph} $G$ on partitioned vertex set
$X \sqcup Y$ with $X=\{x_1,\ldots,x_m\}, Y=\{y_1,\ldots,y_n\}$.
Rather than considering only the ungraded Betti numbers $\beta_i$, here we take advantage of
the $\Z^m$-grading available on the $x$ variables, but
{\it ignoring the grading on the $y$ variables}.  That is, we set
$$
\begin{aligned}
\deg(x_i)&:=e_i \text{ for  }i=1,\ldots,m,\text{ but }\\
\deg(y_j)&:=0 \text{ for }j=1,\ldots,n.
\end{aligned}
$$
For each subset $X' \subseteq X$,
define the Betti number $\beta_{i,X',\bullet}(I)$ to be the $\Z^m$-graded Betti number for this
grading, or the following sum of the usual $\Z^{m+n}$-graded Betti numbers $\beta_{i,X' \sqcup Y'}(I)$ :
$$
\beta_{i,X',\bullet}(I):=\sum_{Y' \subseteq Y} \beta_{i,X' \sqcup Y'}(I).
$$
If the vertex $x_i \in X$ has {\it degree (valence)} $\deg_G(x_i)$ in $G$, then
the relevant ideal $J$ with which we will compare $I$ is
\begin{equation}
\label{associated-Ferrers-ideal}
J:=(x_i y_j: i=1,\ldots,m, \text{ and }j=1,2,\ldots,\deg_G(x_i)).
\end{equation}
The bipartite graphs corresponding to such ideals $J$ are known as {\it Ferrers graphs};
see \cite{CN1} and Example~\ref{Ferrers-definition} below.

\begin{conjecture}
\label{bipartite-conjecture}
Consider the edge ideal
$$
I=(x_i y_j: \{x_i,y_j\} \in G) \,\, \subset \,\, S=k[x_1,\ldots,x_m,y_1,\ldots,y_n]
$$
for some {\it bipartite graph} $G$ on $X \sqcup Y$ as above, and
let $J$ be the Ferrers graph edge ideal associated to $I$ as in \eqref{associated-Ferrers-ideal}.

Then $\beta_{i,X',\bullet}(I)  \geq \beta_{i,X',\bullet}(J)$ for all $i$ and all subsets $X' \subseteq X$.
\end{conjecture}

We remark that the  lower bounds on the Betti numbers in both of the problems can be made quite explicit.
In Question~\ref{nonbipartite-conjecture}, if
the monomial ideal $I$ has exactly $g$ minimal generating monomials, express
$g=\binom{\mu}{d} + \epsilon$ uniquely for some integers $\mu, \epsilon$ with
$\mu \geq d-1$ and $0 \leq \epsilon < \binom{\mu}{d-1}$.  Then the lower bound can be rewritten
(using Corollary~\ref{hypergraph-betti-corollary} below) as
$$
\beta_i(I) \geq \beta_i(J) = \sum_{j=d}^{\mu} \binom{j-1}{i, \,\, d-1, \,\, j-d-i} + \epsilon \binom{\mu+1-d}{i}.
$$
In Conjecture~\ref{bipartite-conjecture}, if for any subset of vertices $X' \subset X$, one denotes
by $\mindeg(X')$ the minimum degree of a vertex $x_i \in X'$ in the bipartite graph $G$,
then the lower bound can be rewritten
(using Proposition~\ref{Ferrers-betti-numbers} below) as
$$
\beta_{i,X',\bullet}(I) \geq \beta_{i,X',\bullet}(J) =
\begin{cases}
\binom{ \mindeg(X') }{i-|X'|+2} &\text{ if } |X'|<i+2 \\
0                               &\text{ otherwise}.\\
\end{cases}
$$

This is certainly not the first paper about lower bounds on the Betti number. For example, there are lower bounds shown by Evans and Griffith, Charalambous, Santoni, Brun, and R\"omer
establishing and strengthening the Buchsbaum-Eisenbud conjecture
(often referred to as {\it Horrocks's problem})
for monomial ideals (see \cite{BR}, \cite{R} and the references therein). The Buchsbaum-Eisenbud conjecture states that the $i$-th total Betti number of a homogeneous  ideal $I$ in a polynomial ring is at least $\binom{c}{i}$, where $c$ is the codimension of $I$. Observe that, for the ideals under consideration in this paper, we ask for much stronger lower bounds.

Another thread in the literature investigates the Betti numbers of ideals with fixed Hilbert function. Among these ideals, the {\it lex-segment ideal} has the maximal Betti numbers according to Bigatti Hulett, and Pardue (\cite{Bigatti} \cite{Hulett}, \cite{Pardue}). However, in general there is no common lower bound for these ideals (see, e.g., \cite{DMMR} and the references therein). In comparison, the novelty of our approach is that instead of the Hilbert function we fix the number of minimal generators of the ideals under consideration.

The remainder of the paper is structured as follows.

Part I introduces a new family of graphs and their edge ideals,
parametrized by well-known combinatorial objects called {\it shifted skew shapes};
each such shape will give rise to both bipartite and nonbipartite graphs,
generalizing two previously studied classes of graphs that have been recently examined
from the point of view of resolutions of their edge ideals -- the Ferrers graphs
\cite{CN1} and the threshold graphs \cite{CN2}.  It turns out that these new families of
edge ideals are extremely well-behaved from the viewpoint of their minimal free resolutions --
the first main result (Corollary~\ref{betti-number-corollary}) gives a combinatorial interpretation
for their most finely graded Betti numbers which is independent
of the coefficient field $k$.  This comes from showing that the relevant simplicial complexes
for these graph ideals, whose homology compute these Betti numbers by a well-known formula of Hochster,
always have the homotopy type of wedges of equidimensional spheres (Theorem~\ref{homotopy-type-theorem}).
This is in marked contrast to the
situation for arbitrary edge ideals of graphs, where the relevant simplicial complexes are well-known to have
the homeomorphism type of any simplicial complex (Proposition~\ref{clique-complex-homeomorphism-types}),
and for arbitrary bipartite graph ideals, where we note (Proposition~\ref{bipartite-complex-homotopy-types})
that the simplicial complexes can have the homotopy type of an arbitrary suspension.
We also show (Theorem~\ref{specialization-theorem}) that the resolutions for the
nonbipartite edge ideals within this class can be obtained by specialization from
the resolutions of the bipartite ones, as was shown in \cite{CN2} for Ferrers and threshold graphs.
We further interpret the Castelnuovo-Mumford regularity (Theorem~\ref{regularity-theorem}) of these
ideals, and indicate how to compute their Krull dimension and projective dimension.

Part II investigates a different generalization of the
Ferrers graph's and threshold graph's edge ideals, this time to {\it nonquadratic}
squarefree monomial ideals including the
special case of the {\it squarefree strongly stable ideals} studied by Aramova, Herzog and Hibi \cite{AHH}
which are generated in a fixed degree.  We provide a simple cellular resolution
for these ideals and some related ideals (Theorem~\ref{cellular-resolutions-theorem}),
related by polarization/specialization again as in \cite{CN2}.  We
also describe an analogously simple cellular resolution for strongly stable ideals
generated in a fixed degree, recovering a recent result of Sinefakopoulos \cite{Sinefakopoulos}.

Part III uses the previous parts to address Question~\ref{nonbipartite-conjecture} and Conjecture~\ref{bipartite-conjecture},
which are verified for all of the ideals whose Betti numbers were computed
in Parts I and II. However, we exhibit an example that shows the existence of monomial ideals that do not obey the colex lower bound (Remark~\ref{rem-counter-ex}). Moreover, a more precise version of Conjecture~\ref{bipartite-conjecture} is
formulated (Conjecture~\ref{bipartite-graph-bound-conjecture}), incorporating both an
upper and a lower bound on the Betti numbers for bipartite graph edge ideals, as well as
a characterization
of the case of equality in both bounds.  Furthermore, the upper bound, as well as the
characterizations for the cases of equality in both the upper and the lower bound are proven,
leaving only the lower bound itself unproven.

The Epilogue contains some questions suggested by the above results. In the Appendix some needed technical tools from combinatorial topology and commutative
algebra are provided.

\tableofcontents

\section{PART I. Shifted skew diagrams and graph ideals}
\label{shifted-skew-section}

\subsection{Shifted diagrams and skew diagrams}

  We begin with some terminology for diagrams in the shifted plane that
are perhaps not so standard in commutative algebra, but fairly
standard in the combinatorial theory of projective representations of
the symmetric group and Schur's $P$ and $Q$-functions \cite[Exercise I.9]{Macdonald}.

\begin{definition}
The {\it shifted plane} is the set of lattice points
$$
\{(i,j) \in \Z \times \Z: 1 \leq i < j \}
$$
drawn in the plane so that the first coordinate increases from the top row to
the bottom, and the second coordinate increases from left to right:
$$
\begin{matrix}
\cdot  & (1,2)  & (1,3)  & (1,4)  & \cdots \\
\cdot  & \cdot  & (2,3)  & (2,4)  & \cdots \\
\cdot  & \cdot  & \cdot  & (3,4)  & \cdots \\
\vdots & \vdots & \vdots & \vdots & \vdots
\end{matrix}
$$

Given a number partition $\lambda = (\lambda_1, \lambda_2, \cdots,\lambda_\ell)$
into {\it distinct parts}, that is, $\lambda_1 > \lambda_2 > \cdots > \lambda_\ell >0$,
the {\it shifted Ferrers diagram} for $\lambda$ is the set of cells/boxes in the
shifted plane having $\lambda_i$ cells left-justified in row $i$ for each $i$.
For example, $\lambda=(12,11,7,6,4,2,1)$ has this diagram:
$$
\begin{matrix}
\cdot  & \times & \times & \times & \times & \times & \times & \times & \times & \times & \times & \times & \times \\
\cdot  & \cdot & \times & \times & \times & \times & \times & \times & \times & \times & \times & \times & \times \\
\cdot  & \cdot & \cdot & \times & \times & \times & \times & \times & \times & \times & & &  \\
\cdot  & \cdot & \cdot & \cdot  & \times & \times & \times & \times & \times & \times & & &  \\
\cdot  & \cdot & \cdot & \cdot  & \cdot  & \times & \times & \times & \times & & & &  \\
\cdot  & \cdot & \cdot & \cdot  & \cdot  & \cdot  & \times & \times & & & & &  \\
\cdot  & \cdot & \cdot & \cdot  & \cdot  & \cdot  & \cdot  & \times & & & & &  \\
\cdot  & \cdot & \cdot & \cdot  & \cdot  & \cdot  & \cdot  & \cdot  & & & & &  \\
\end{matrix}
$$

Given another number partition $\mu$ with distinct parts for which $\mu_i < \lambda_i$,
one can form the {\it shifted skew diagram} $D=\lambda/\mu$ by removing the diagram for
$\mu$ from the diagram for $\lambda$.  For example, if $\mu=(11,9,6,3)$ and $\lambda=(12,11,7,6,4,2,1)$
as before, then $D=\lambda/\mu$ has the following shifted skew diagram, with row and column indices
labelled to emphasize its location within the shifted plane:
$$
\begin{matrix}
 &1&2&3&4&5&6&7&8&9&10&11&12&13\\
1&\cdot  &  &  &  &  &  & &  &  &  &  &  & \times \\
2&\cdot  & \cdot &  &  &  &  &  &  &  &  &  & \times & \times \\
3&\cdot  & \cdot & \cdot &  &  &  &  &  &  & \times & & &  \\
4&\cdot  & \cdot & \cdot & \cdot  &  &  &  & \times & \times & \times & & &  \\
5&\cdot  & \cdot & \cdot & \cdot  & \cdot  & \times & \times & \times & \times & & & &  \\
6&\cdot  & \cdot & \cdot & \cdot  & \cdot  & \cdot  & \times & \times & & & & &  \\
7&\cdot  & \cdot & \cdot & \cdot  & \cdot  & \cdot  & \cdot  & \times & & & & &  \\
8&\cdot  & \cdot & \cdot & \cdot  & \cdot  & \cdot  & \cdot  & \cdot  & & & & &  \\
\end{matrix}
$$
\end{definition}
In a shifted skew diagram $D$, cells in locations of the form $(i,i+1)$ will
be called {\it staircase cells}.  For example, the diagram above has three staircase
cells, namely $(5,6),(6,7),(7,8)$.

Given a shifted skew diagram $D$, and any pair $X,Y$ of ordered subsets of
positive integers
$$
\begin{aligned}
X &=\{x_1 < x_2 < \cdots < x_m\} \\
Y &=\{y_1 < y_2 < \cdots < y_n\},
\end{aligned}
$$
one can form a diagram $\Dbip{X}{Y}$ with rows indexed by $X$ and
columns indexed by $Y$, by restricting the diagram $D$ to these rows and columns.
For example if $D=\lambda/\mu$ is the shifted skew diagram shown above, and
if
$$
\begin{array}{lll}
X & =\{x_1,x_2,x_3,x_4\}&=\{2,4,5,7\}\\
Y &=\{y_1,y_2,y_3,y_4,y_5,y_6,y_7,y_8\}&=\{4,6,7,8,9,10,11,12\}
\end{array}
$$
then $\Dbip{X}{Y}$ is this diagram:
\begin{equation}
\label{bipartite-diagram-example}
\begin{matrix}
 &y_1&y_2&y_3&y_4&y_5&y_6&y_7&y_8\\
 &4&6&7&8&9&10&11&12\\
x_1=2&  &  &  &  &  &  &  &\times \\
x_2=4&  &  &  &\times&\times&\times& & \\
x_3=5&  &\times&\times&\times&\times&  &  &  \\
x_4=7&  &  &  &\times&  &  &  & \\
\end{matrix}
\end{equation}

Such diagrams $\Dbip{X}{Y}$ should no longer be considered as drawn in the
shifted plane, but rather inside the $m \times n$ rectangle with row and
column indices given by $X$ and $Y$.

On the other hand, given a shifted skew diagram $D$, and an ordered subset
$X$, one can also form the diagram $\Dnonbip{X} (=\Dbip{X}{X})$, which one should
think of as drawn in a shifted plane whose rows and columns are indexed by $X$.
For example, if $D=\lambda/\mu$ as above and $X=\{x_1,x_2,x_3,x_4,x_5,x_6\}=\{2,4,5,7,8,10\}$,
then $\Dnonbip{X}$ is this diagram:
\begin{equation}
\label{nonbipartite-diagram-example}
\begin{matrix}
     &x_1&x_2&x_3&x_4&x_5&x_6\\
     &2&4&5&7&8&10\\
x_1=2&\cdot  &  &  &      &      &  \\
x_2=4&\cdot  &\cdot  &  &      &\times&\times\\
x_3=5&\cdot  &\cdot  &\cdot  &\times&\times&  &  &  \\
x_4=7&\cdot  &\cdot  &\cdot  &\cdot      &\times&  &  & \\
x_5=8&\cdot  &\cdot  &\cdot  &\cdot      &\cdot      &  &  & \\
x_6=10&\cdot  &\cdot  &\cdot  &\cdot     &\cdot      &\cdot  &  & \\
\end{matrix}
\end{equation}
For such diagrams $\Dnonbip{X}$, call the cells in locations of the
form $(x_i,x_{i+1})$ its {\it staircase cells}.  For example, in $\Dnonbip{X}$ shown
above there are two staircase cells, in positions $(x_3,x_4), (x_4,x_5)$.

\subsection{Graphs and graph ideals}

\begin{definition}
A {\it (simple) graph} $G$ on vertex set $V$ is a collection $E(G) \subset \binom{V}{2}:=\{\{u,v\}: u,v \in V\}$
called its {\it edges}.  Having fixed a field $k$ to use as coefficients, any graph $G$ gives rise to a square-free quadratic monomial ideal called
its {\it edge ideal} $I(G)$ inside the polynomial ring\footnote{
We hope that using the names of vertices as polynomial variables, a very convenient
abuse of notation,  causes no confusion.} $k[V]:=k[v]_{v \in V}$, generated by the
monomials $uv$ as one runs through all edges $\{u,v\}$ in $E(G)$.
\end{definition}

Note that since $I(G)$ is a monomial ideal, it is homogeneous with respect to the $\Z^{|V|}$-grading on $k[V]$ in which the degree of the variable $v$ is the standard basis vector $e_v \in \R^{|V|}$.

There is a situation in which a different grading also appears.

\begin{definition}
Say that a simple graph $G$ is {\it bipartite} with respect to the partition
$V = V_1 \sqcup V_2$ of its vertex set $V$ if every edge in $E(G)$ has the form
$\{v_1,v_2\}$ with $v_i \in V_i$ for $i=1,2$.

Equivalently, $G$ is bipartite with respect to $V=V_1 \sqcup V_2$ if and only if
$I(G)$ is homogeneous with respect to the $\Z^2$-grading in $k[V]$ in which
the variables labelled by vertices in $V_1$ all have degree $(1,0)$,
which the variables labelled by vertices in $V_2$ all have degree $(0,1)$.
\end{definition}

Given any shifted skew diagram $D$, the two kinds of subdiagrams
$\Dbip{X}{Y}, \Dnonbip{X}$ give rise to two kinds of graphs, and hence
to two kinds of edge ideals:
\begin{enumerate}
\item[$\bullet$] For a pair of ordered sets $X=\{x_1,\ldots,x_m\},Y=\{y_1,\ldots,y_n\}$, one has the
the bipartite $\Gbip{X}{Y}(D)$ graph on vertex set $X \sqcup Y$, with an edge $\{x_i,y_j\}$ for every cell
$(x_i,y_j)$ in the diagram $\Dbip{X}{Y}$.  Its edge ideal
$I(\Gbip{X}{Y}(D))$ is inside the polynomial algebra $k[x_1,\ldots,x_m,y_1,\ldots,y_n]$.
\item[$\bullet$] For a single ordered set $X=\{x_1,\ldots,x_m\}$, one has the
the nonbipartite $\Gnonbip{X}(D)$ graph on vertex set $X$, with an edge $\{x_i,x_j\}$ for every cell
$(x_i,x_j)$ in the diagram $\Dnonbip{X}$.  Its edge ideal
$I(\Gnonbip{X}(D))$ is inside the polynomial algebra $k[x_1,\ldots,x_m]$.
\end{enumerate}

\begin{example} \rm \ \\
\label{first-graphs-example}
If $\Dbip{X}{Y}$ and $\Dnonbip{X}$ are the diagrams shown in \eqref{bipartite-diagram-example}, \eqref{nonbipartite-diagram-example},
respectively, then
$$
\begin{aligned}
I(\Gbip{X}{Y}(D)) & = (x_1 y_8, x_2 y_4, x_2 y_5, x_2 y_6, x_3 y_2, x_3 y_3, x_3 y_4, x_3 y_5, x_4 y_4) \\
 & \subset k[x_1,x_2,x_3,x_4,y_1,y_2,y_3,y_4,y_5,y_6,y_7,y_8]\\
I(\Gnonbip{X}(D)) & = (x_2 x_5, x_2 x_6, x_3 x_4, x_3 x_5, x_4 x_5) \\
 & \subset k[x_1,x_2,x_3,x_4,x_5,x_6]
\end{aligned}
$$
\end{example}


We review now some well-studied classes of graphs that were our motivating special cases.

\begin{example}(Ferrers bipartite graphs) \rm \ \\
\label{Ferrers-definition}
Say that $\Dbip{X}{Y}$ is {\it Ferrers} if whenever $i < i'$,
the columns occupied by the cells in the row $x_{i'}$ form a subset of those
occupied by the cells in row $x_i$.  The graph $\Gbip{X}{Y}(D)$ is then completely determined up
to isomorphism by the partition $\lambda=(\lambda_1 \geq \cdots \geq \lambda_m)$
where $\lambda_i$ is the number of cells in the row $x_i$.  Call such a Ferrers graph $G_\lambda$.
An explicit cellular minimal free resolution of $I(G_\lambda)$ for the Ferrers graphs $G_\lambda$
was given in \cite{CN1}, thereby determining its Betti numbers --  see also
Example~\ref{Ferrers-example} below.
\end{example}

\begin{example}(threshold graphs) \rm \ \\
If $D$ is the shifted Ferrers diagram for a strict partition $\lambda=(\lambda_1 > \cdots > \lambda_m)$,
then the graph $\Gnonbip{[n]}(D)$ is called a {\it threshold} graph.  Such graphs have numerous
equivalent characterizations -- see \cite{MahadevPeled}.

An explicit cellular minimal free resolution of $I(\Gnonbip{[n]}(D))$
in this case was derived in \cite{CN2} by specialization from the resolution
of an associated Ferrers graph from \cite{CN1}.
\end{example}

\subsection{Betti numbers and simplicial complexes}

Edge ideals $I(G)$ of graphs are exactly the squarefree quadratic monomial ideals.
More generally, any squarefree monomial ideal $I$ in a polynomial algebra $k[V]$
has some special properties with regard to its minimal free resolution(s) as $k[V]$-module.
Since $I$ is a monomial ideal, the resolution can be chosen
$\Z^{|V|}$-homogenous.  Because it is generated by squarefree monomials, the free summands in each resolvent
will have basis elements occurring in degrees which are also squarefree, corresponding
to subsets $V' \subset V$.  The {\it finely graded Betti number} $\beta_{i,V'}(I)$
is defined to be the number of such basis elements in the $i^{th}$ syzygy module occurring in the
resolution, or equivalently,
$$
\beta_{i,V'}(I) = \dim_k \Tor_i^{k[V]}(I,k)_{V'}
$$
where here $M_{V'}$ denotes the $V'$-graded component of a $\Z^{|V|}$-graded vector space.
The {\it standard graded} and {\it ungraded Betti numbers} are the coarser data defined by
$$
\begin{array}{llll}
\beta_{i,j}(I) &= \dim_k \Tor_i^{k[V]}(I,k)_j &=\sum_{V' \subseteq V: |V'|=j} \beta_{i,V'}(I) &\\
\beta_{i}(I)   &= \dim_k \Tor_i^{k[V]}(I,k)   &=\sum_{V' \subseteq V} \beta_{i,V'}(I) &= \sum_j \beta_{ij}(I).
\end{array}
$$

A famous formula of Hochster relates these resolution Betti numbers to simplicial
homology.  An {\it abstract simplicial complex $\Delta$} on vertex set $V$ is a collection
of subsets $F$ of $V$ (called {\it faces} of $\Delta$) which is closed under inclusion:
if $G \subset F$ and $F \in \Delta$ then $G \in \Delta$.  Maximal faces of $\Delta$ under inclusion
are called {\it facets} of $\Delta$.

There is a straightforward bijection
(the {\it Stanley-Reisner} correspondence) between simplicial complexes $\Delta$ on vertex set $V$ and
squarefree monomial ideals $I_\Delta$ inside $k[V]$:  let $I_\Delta$ be generated by all squarefree monomials
$x_{v_1} \cdots x_{v_s}$ for which $\{v_1,\ldots,v_s\} \not\in \Delta$.   Hochster's formula
for $\beta_{i,V'}(I_\Delta)$ expresses it in terms of the {\it (reduced) simplicial homology}
of the {\it vertex-induced subcomplex}
$$
\Delta_{V'} := \{F \in \Delta: F \subset V'\}.
$$

\begin{proposition}(Hochster's formula \cite[Corollary 5.12]{MillerSturmfels})
\label{Hochster-formula}
For a squarefree monomial ideal $I_\Delta \subset k[V]$ and any $V' \subset V$,
one has a $k$-vector space isomorphism
$$
\Tor_i^{k[V]}(I,k)_{V'}  \cong \tilde{H}_{|V'|-i-2}(\Delta_{V'})
$$
and hence
$$
\beta_{i,V'}(I_\Delta) = \dim_k \tilde{H}_{|V'|-i-2}(\Delta_{V'}).
$$
\end{proposition}

If $I_\Delta = I(G)$ for a graph $G$ on vertex set $V$, then we will write $\Delta = \Delta(G)$;
the name for such simplicial complexes $\Delta$ is that they are {\it flag} or {\it clique} complexes.
{\bf Warning:} this does {\it not} mean that $\Delta$ is the $1$-dimensional simplicial complex generated
by the edges of $G$.  In fact, there is a somewhat more direct relationship between the edges of $G$ and
the canonical Alexander dual of $\Delta(G)$.

\begin{definition}
\label{Alexander-dual}
Given a simplicial complex $\Delta$ on vertex set $V$, its {\it canonical Alexander dual} $\Delta^\vee$
is the simplical complex on vertex set $V$ defined by
$$
\Delta^\vee:=\{ V \setminus G: G \not\in \Delta\}.
$$
\end{definition}

\noindent
Note that the operation $\Delta \mapsto \Delta^\vee$ is involutive: $(\Delta^\vee)^\vee = \Delta$.
It is an easy exercise in the definitions to check that, for a graph $G$ on vertex set $V$,
the facets of the Alexander dual $\Delta(G)^\vee$ are exactly the
complementary sets $V \setminus \{u,v\}$ to the edges $\{u,v\}$ in $E(G)$.

Lastly, note that a shifted skew diagrams $D$ will give rise to two simplicial complexes
$$
\Delta(\Gbip{X}{Y}(D)), \Delta(\Gnonbip{X}(D))
$$
which control the Betti numbers of the edge ideals $I(\Gbip{X}{Y}(D)), I(\Gnonbip{X}(D)).$
More precisely, each vertex-induced subcomplex which appears in Proposition~\ref{Hochster-formula}
for calculating the graded Betti numbers is another simplicial complex of the same form:
$$
\begin{aligned}
\Delta(\Gbip{X}{Y}(D))_{X' \sqcup Y'} &= \Delta(\Gbip{X'}{Y'}(D))\\
\Delta(\Gnonbip{X}(D))_{X'} &= \Delta(\Gnonbip{X'}(D))
\end{aligned}
$$

Thus our next goal will be to study the homotopy type of $\Delta(\Gbip{X}{Y}(D))$ and $\Delta(\Gnonbip{X}(D))$.

\subsection{Rectangular decomposition}
\label{sec-rectangular}

The idea in this section is to produce what we call the
{\it rectangular decomposition} for any diagram
$\Dbip{X}{Y}$ (or $\Dnonbip{X}$).  As an informal illustration,
here is the rectangular decomposition of the following diagram $\Dbip{X}{Y}$
into pieces of various types, explained below the diagram:
\begin{equation}
\label{rectangular-decomp-example}
\begin{matrix}
 &y_1&y_2&y_3&y_4&y_5&y_6&y_7&y_8&y_9&y_{10}&y_{11}&y_{12}&y_{13}&y_{14}&y_{15}&y_{16}\\
x_1& & & & & & & & & & & & & & & &\cdot\\
x_2& & & & & & & & & & & & & & & &\cdot\\
x_3& & & & & & & & & & & & &r_1&r_1&{\mathbf r_1}& \\
x_4& & & & & & & & & & &e&e&r_1&r_1&r_1& \\
x_5& & & & & & & & & &r_2&r_2&{\mathbf r_2}&e&e& & \\
x_6& & & & & & & & &e&r_2&r_2&r_2&e& & & \\
x_7& & & & & & & &e&e&r_2&r_2&r_2&e& & & \\
x_8& & & & & & &{\mathbf r_3}& & & & & & & & & \\
x_9& & & & & &e&r_3& & & & & & & & & \\
x_{10}& & & & &e&e&r_3& & & & & & & & & \\
x_{11}& &p&p&p&p&{\mathbf p}& & & & & & & & & & \\
x_{12}&e&p&p&p&p&p& & & & & & & & & & \\
x_{13}& & & &p&p&p& & & & & & & & & & \\
x_{14}& & & &e&e& & & & & & & & & & & \\
\end{matrix}
\end{equation}
There are
\begin{enumerate}
\item[$\bullet$] some {\it full rectangles}, of which there are
three in the diagram at right, whose cells have been labelled $r_1$ or $r_2$ or $r_3$,
\item[$\bullet$] some {\it empty rectangles}, of which there are
two in the diagram at right, one indicated by dots occupying rows
$\{x_1,x_2\}$ and column $\{y_{16}\}$, the other occupying columns $\{y_8,y_9\}$ but
having zero width (lying {\it ``between''} rows $x_7$ and $x_8$),
\item[$\bullet$] at most one {\it pedestal}, whose cells
are labelled ``$p$'' in the diagram at right, and
\item[$\bullet$] some {\it excess cells}, labelled by ``$e$''.
\end{enumerate}

\noindent
Informally, the idea behind the rectangular decomposition is that in analyzing the
homotopy type of the associated simplicial complexes in Section~\ref{sec-homotopy},
one finds that
\begin{enumerate}
\item[$\bullet$]
removing excess cells does not change the homotopy type,
\item[$\bullet$]
once the excess cells are removed, the complex decomposes into a simplicial join of
the complexes corresponding to each full/empty rectangle and the pedestal (if present),
\item[$\bullet$]
complexes associated to full rectangles are zero-dimensional spheres,
\item[$\bullet$]
complexes assocated to empty rectangles are simplices, hence contractible, and
\item[$\bullet$]
the complex assocated to a pedestal is contractible in the case
of $\Dbip{X}{Y}$, or an $s$-fold wedge of zero-spheres in the case of
$\Dnonbip{X}$ where $s$ is the number of (non-excess) staircase cells.
\end{enumerate}

\noindent
Consequently, the homotopy type for $\Delta(\Gbip{X}{Y}(D)), \Delta(\Gnonbip{X}(D))$
will always be either contractible or a wedge
of equidimensional spheres, easily predicted from the above
decomposition.

Here is the formal algorithm that produces the rectangular decomposition.

\begin{definition}
\label{rectangle-decomp-defn}
Define the {\it rectangular decomposition} of $\Dbip{X}{Y}$ recursively for any
shifted skew diagram $D$ and ordered sets $X=\{x_1 < \cdots <x_m\}$,
$Y=\{y_1< \cdots <y_n\}$ with $X \sqcup Y \neq \varnothing$,
allowing either $X$ or $Y$ to be empty.  The algorithm will in general
go through several iterations, terminating either
when $X \sqcup Y$ becomes empty, or when one
encounters a {\it pedestal} in Subcase 2b below.

Say that $\Dbip{X}{Y}$ {\it has a top cell}
if it contains in cell in position $(x_1,y_n)$;  in particular this requires
{\it both} $X, Y$ to be nonempty.

Initialize the set of {\it excess cells} as the empty set;  cells will be identified
as excess cells during iterations of the algorithm.

In each iteration, there are several cases.

\vskip.1in
\noindent
{\sf Case 1.}
$\Dbip{X}{Y}$ has no top cell.

Then there exist some
\begin{equation}
\label{initial-final-segments}
\begin{aligned}
\text{ initial segment }X'&=\{x_1,x_2,\cdots,x_{m'}\} \subset X, \text{ and }\\
\text{ final segment }Y'&=\{y_{n'},y_{n'+1},\cdots,y_n\} \subset Y
\end{aligned}
\end{equation}
such that
both $\Dbip{X}{Y'}$ and $\Dbip{X'}{Y}$ contain
no cells.
In this case, pick the segments $X', Y'$ {\it maximal} with this property,
and call $\Dbip{X'}{Y'}$ the first {\it empty rectangle} in the
rectangular decomposition.  Note that $X' \cup Y' \neq \varnothing$,
but it is possible that either $X'$ or $Y'$ might be empty, in which
case one has an empty rectangle with zero length or zero width (!).

Now remove the rows and columns $X', Y'$, that is,
replace $\Dbip{X}{Y}$ by $\Dbip{X \setminus X'}{Y \setminus Y'}$,
and continue the rectangular decomposition.

\vskip.1in
\noindent
{\sf Case 2.}
$\Dbip{X}{Y}$ has a top cell, namely $(x_1,y_n)$.

Define indices $m',n'$ uniquely by saying $m'$ (resp. $n'$) is maximal (resp. minimal)
for which $(x_{m'},y_n)$ (resp. $(x_1,y_{n'})$) is a cell of $\Dbip{X}{Y}$.

If $\Dbip{X}{Y}$ has a cell in position $(x_{m'},y_{n'})$, then this will
be called its {\it neck cell}.

Again define initial, final segments $X', Y'$ by
$$
\begin{aligned}
X'&=\{x_1,x_2,\cdots,x_{m'}\} \subset X, \text{ and }\\
Y'&=\{y_{n'},y_{n'+1},\cdots,y_n\} \subset Y.
\end{aligned}
$$

\vskip.1in
\noindent
{\sf Subcase 2a.}
$\Dbip{X}{Y}$ has both a top cell and a neck cell (possibly the {\it same} cell!)

In this case, $\Dbip{X'}{Y'}$ is a full rectangle in the sense that every possible
position $(x_i,y_j)$ with $i \in X', j \in Y'$ actually contains a cell of $D$.
In fact, our choice of $m', n'$ makes $X',Y'$ maximal with respect to this property.
Call $\Dbip{X'}{Y'}$ the first {\it full rectangle} in the
rectangular decomposition.

Then add to the set of excess cells all cells of $\Dbip{X \setminus X'}{Y'}$ (i.e., those
lying below the full rectangle in the same columns) and all cells of $\Dbip{X'}{Y \setminus Y'}$
(i.e., those lying left of the full rectangle in the same rows).

Lastly, remove the rows and columns $X', Y'$ from $X, Y$, that is,
replace $\Dbip{X}{Y}$ by $\Dbip{X \setminus X'}{Y \setminus Y'}$,
and continue the rectangular decomposition.

\vskip.1in
\noindent
{\sf Subcase 2b.}
$\Dbip{X}{Y}$ has a top cell but no neck.

Now call $\Dbip{X'}{Y'}$ the {\it pedestal} in the
rectangular decomposition.  Note that not every diagram will have such a
pedestal.

As in Subase 2a, add all cells of $\Dbip{X \setminus X'}{Y'}$ and $\Dbip{X'}{Y \setminus Y'}$
to the set of excess cells.  But now the algorithm also terminates.
\end{definition}

\begin{example}
The diagram $\Dbip{X}{Y}$ in \eqref{rectangular-decomp-example}
whose nonempty cells are labelled $e,r_1,r_2,r_3,p$ passes through
six iterations of the algorithm:
\begin{enumerate}
\item[\bf{1st}]
Case 1-- add the empty rectangle $\Dbip{\{x_1,x_2\}}{\{y_{16}\}}$ to the decomposition.
\item[\bf{2nd}]
Subcase 2a-- add the full rectangle $\Dbip{\{x_3,x_4\}}{\{y_{13},y_{14},y_{15}\}}$ (with cells labelled $r_1$,
top cell $(x_3,y_{15})$ in boldface, neck cell $(x_4,y_{13})$)
to the decomposition, and identify two excess cells to its left as well
as four excess cells below it.
\item[\bf{3rd}]
Subcase 2a-- add the
full rectangle $\Dbip{\{x_5,x_6,x_7\}}{\{y_{10},y_{11},y_{12}\}}$ (with cells labelled with $r_2$,
top cell $(x_5,y_{12})$ in boldface, neck cell $(x_7,y_{10})$)
to the decomposition, and identify three excess cells to its left.
\item[\bf{4th}]
Case 1-- add the empty rectangle $\Dbip{\varnothing}{\{y_8,y_9\}}$ to the decomposition.
Note that this empty rectangle has zero width, i.e. it occupies the empty set $X'=\varnothing$
of rows (``between'' rows and $x_7$ and $x_8$).
\item[\bf{5th}]
Subcase 2a-- add the
full rectangle $\Dbip{\{x_8,x_9,x_{10}\}}{\{y_7\}}$ (with cells labelled with $r_3$,
top cell $(x_8,y_7)$ in boldface, neck cell $(x_{10},y_7)$)
to the decomposition, and identify three excess cells to its left.
\item[\bf{6th}]
Subcase 2b-- add the pedestal $\Dbip{\{x_{11},x_{12},x_{13}\}}{\{y_2,y_3,y_4,y_5,y_6\}}$ (with cells labelled with $p$,
top cell $(x_{11},y_6,)$ in boldface, {\it no} neck cell)
to the decomposition, and identify one excess cell to its left, two excess cells below it.
\end{enumerate}
\end{example}

\begin{definition} \rm \ \\
The same algorithm produces the {\it rectangular decomposition} of $\Dnonbip{X}$, viewing it as $\Dbip{X}{Y}$
with $Y=X$.  The only difference is that if a pedestal occurs in the rectangular decomposition (Subcase 2b)
for $\Dnonbip{X}$, one can view the pedestal itself as
a diagram in the shifted plane, and hence certain of its
cells are distinguished as staircase cells.  The number of these staircase cells becomes
important in the next section when one analyzes the homotopy type
of $\Delta(\Gnonbip{X}(D))$.
\end{definition}

Before closing this section, we note a simple criterion for when $\Dbip{X}{Y}$ has a pedestal,
used later as an aid to show that certain diagrams $\Dbip{X}{Y}$ have $\Delta(\Gbip{X}{Y}(D))$ contractible.

\begin{proposition}
\label{pedestal-proposition}
For any shifted skew diagram $D$ and ordered subsets $X,Y$, the diagram $\Dbip{X}{Y}$ has a pedestal
if and only if it contains two cells $c=(i,j), c'=(i',j')$ with $i<i'$ and $j < j'$ but does not
contain the cell $(i',j)$ in the southwest corner of the rectangle that they define.
\end{proposition}
\begin{proof}
Assume $\Dbip{X}{Y}$ has pedestal $\Dbip{X'}{Y'}$, with top cell $(x_1,y_n)$ and $m':=\max X'$ and $n':=\min Y'$.
Then $c=(x_1,y_{n'}), c'=(x_{m'},y_n)$ satisfy the conditions in the proposition, because $(i',j)=(x_{m'},y_{n'})$
is the location of the missing neck cell that would have made the pedestal into a
full rectangle.

On the other hand, it is easily seen that when $\Dbip{X}{Y}$ has no pedestal it looks like a usual
{\it skew Ferrers diagram} \cite[\S I.1]{Macdonald}.
Such diagrams have the property that when they contain two cells $c,c'$ forming
the northwest and southeast corners of a rectangle, the entire rectangle is in the diagram, including
its southwest corner cell.
\end{proof}

\subsection{Homotopy type and Betti numbers}
\label{sec-homotopy}

The goal of this section is Theorem~\ref{homotopy-type-theorem}, describing the homotopy type of
$\Delta(\Gbip{X}{Y}(D))$ (resp. $\Delta(\Gnonbip{X}(D))$)
in terms of the rectangular decomposition of
$\Dbip{X}{Y}$ (resp. $\Dnonbip{X}$).

The key point is that one can remove excess cells from the diagrams without
changing the homotopy type of the associated simplicial complexes.

\begin{lemma}
\label{excess-cell-lemma}
Assume one has two nested diagrams $D_1 \subset D_2$ with
both $D_i$ of the form $\Dbip{X}{Y}$ (respectively, $\Dnonbip{X}$).
Let $\Delta_1 \subset \Delta_2$ be their associated simplicial
complexes of the form $\Delta(\Gbip{X}{Y}(D))$ (respectively, $\Delta(\Gnonbip{X}(D))$).
Furthermore assume that $D_1$ is obtained from $D_2$ by removing one
excess cell of $D_2$.

Then $\Delta_1, \Delta_2$ are homotopy equivalent.
\end{lemma}
\begin{proof}
By Lemma~\ref{topology-lemma} in the Appendix, it suffices to show that the Alexander dual
$\Delta_2^\vee$ is obtained from $\Delta_1^\vee$ by adding a new facet $F$
with the property that the subcomplex $2^F \cap \Delta_1^\vee$
has a cone vertex.

We give the argument for the case of $\Gbip{X}{Y}(D))$;  the only change necessary
for the case of $\Gnonbip{X}(D)$ is to replace each occurrence of a vertex $y_j$
with the corresponding vertex $x_j$ having the same subscript $j$.

Let $e=(x_i,y_j)$ be the unique cell in $D_2 \setminus D_1$.  Since $e$ is
an excess cell, it must have been identified as excess during an iteration
of the rectangular decomposition algorithm that fell into Subcase 2a or 2b.
Then $e$ is located either below or to the left of a full rectangle
or pedestal created during that iteration;  call this rectangle or pedestal $R$ in either
case.  Let $(x_{m'},y_{n'})$ be the top cell for the rectangle or pedestal $R$.
This implies $i > m'$ and $j < n'$.

Note that the extra facet $F$ of $\Delta_2^\vee$ not in $\Delta_1^\vee$ corresponding to $e$
has vertices $X \sqcup Y \setminus \{x_i,y_j\}$.
If $e$ is located below (resp. to the left of) $R$, we will show that the vertex $v:=y_{n'}$ (resp. $v:=x_{m'}$) forms a
cone vertex for the intersection subcomplex $2^F \cap \Delta_1^\vee$.  This means showing for all facets $F'$ of $\Delta_1^\vee$
there exists a facet $F''$ of $\Delta_1^\vee$ containing $v$ with the further property that $F \cap F' \subset F''$.
If $F'$ corresponds to the cell $(x_{i'},y_{j'})$ of $D_1$, then this means one must find a cell $(x_{i''},y_{j''})$
of $D_1$ with $y_{j''} \neq y_{n'}$ (resp. $x_{i''} \neq x_{m'}$) and the further property that
$$
\{x_i,y_j\} \cup \{x_{i'},y_{j'}\} \supset \{x_{i''},y_{j''}\}.
$$

If $y_{j'} \neq y_{n'}$ (resp. $x_{i'} \neq x_{m'}$) then this is easy;  let $(x_{i''},y_{j''}):=(x_{i'},y_{j'})$.
In other words, if $v \not\in F'$ then one can simply take $F'':=F'$.

If $y_{j'} = y_{n'}$ (resp. $x_{i'} = x_{m'}$) then let $(x_{i''},y_{j''}):=(x_{i'},y_{j})$
(resp. let $(x_{i''},y_{j''}):=(x_{i},y_{j'})$.  There always
exists a a cell located at $(x_{i''},y_{j''})$ in $D_1$ because
this position is different from $e$ and $D_2$ has a cell located in positions $e$ and $(x_{i'}, y_{j'})$.
Hence $F''=X \sqcup Y \setminus \{ x_{i''},y_{j''} \}$ is a facet of $\Delta_1^\vee$.
\end{proof}

\begin{definition}
\label{rectangularity-definition}
Call a diagram of the form $\Dbip{X}{Y}$ {\it spherical} if in its rectangular decomposition
it has only full rectangles and possibly some excess cells, but no empty rectangles nor pedestal.

Given a diagram of the form $E=\Dbip{X}{Y}$ or $E=\Dnonbip{X}$, define its {\it rectangularity}
$\rect(E)$ to be the number of full rectangles and/or pedestals (if present)
in its rectangular decomposition.
\end{definition}

For example, Diagram (\ref{rectangular-decomp-example}) has three full rectangles and one pedestal, thus its rectangularity is four. It is not spherical.

The following result justifies the name spherical in Definition \ref{rectangularity-definition}. 

\begin{theorem}
\label{homotopy-type-theorem}
Let $D$ be any shifted skew diagram $D$.

For any ordered subsets $X, Y$, the homotopy type of $\Delta(\Gbip{X}{Y}(D))$ is
\begin{enumerate}
\item[$\bullet$] an $(\rect(\Dbip{X}{Y})-1)$-dimensional
sphere if $\Dbip{X}{Y}$ is spherical, and
\item[$\bullet$] contractible otherwise.
\end{enumerate}

For any ordered subset $X$,
the homotopy type of $\Delta(\Gnonbip{X}(D))$ is
\begin{enumerate}
\item[$\bullet$] contractible if there are any empty rectangles in the rectangular
decomposition, and
\item[$\bullet$] an $s$-fold wedge of $(\rect(\Dnonbip{X})-1)$-dimensional
spheres if $s$ denotes the number of non-excess staircase cells otherwise.
\end{enumerate}
\end{theorem}

\begin{proof}
Lemma~\ref{excess-cell-lemma} reduces the proof to the case where the diagrams have no excess cells.

When there are no excess cells, the diagrams are {\it disjoint unions} of their
various empty or full rectangles and pedestal, where
here the disjoint union of diagrams means diagrams that share no row or column indices.  In this case, it is easily seen that
the relevant graphs $\Gbip{X}{Y}(D)$ and $\Gnonbip{X}(D)$ are also {\it disjoint
unions} of the graphs corresponding to these pieces (full/empty rectangle
or pedestal).
Consequently the complexes $\Delta(\Gbip{X}{Y}(D))$ and
$\Delta(\Gnonbip{X}(D))$ are {\it simplicial joins} \cite[\S 62]{Munkres} of the complexes
corresponding to these pieces.

Thus it remains to analyze the homotopy types of the two kinds of complexes
when there is only one piece (empty rectangle, full rectangle, or
pedestal) in the rectangular decomposition.

For an empty rectangle, either complex is contractible because
it is the full simplex $2^V$ on its vertex set $V=X \sqcup Y$ or
$V=X$.

For a full rectangle, either complex is homotopy equivalent to a zero
sphere because it is the disjoint union of two full simplices, one on the
vertices indexing its rows, the other on the vertices indexing its columns.

For a pedestal, one analyzes $\Dbip{X}{Y}$ and $\Dnonbip{X}$
separately.

In the case of a pedestal in the shifted plane of the form $\Dnonbip{X}$, say with $s$ (non-excess)
staircase cells in positions
$$
(x_{i},x_{i+1}), (x_{i+1},x_{i+2}),\ldots,(x_{i+s-2},x_{i+s-1}),(x_{i+s-1},x_{i+s}),
$$
one can check directly that $\Delta(\Gnonbip{X}(D))$ is the disjoint union of the
$s+1$ full simplices on the vertex sets
$$
\{x_1,x_2,\ldots,x_i\}, \{x_{i+1}\}, \{x_{i+2}\}, \ldots,
\{x_{i+s-1}\},\{x_{i+s},x_{i+s+1},\ldots,x_n\},
$$
where $(1, n)$ is the position of the top cell of the pedestal.
Note that such a disjoint union of $s+1$ simplices is homotopy equivalent
to $s+1$ isolated vertices, that is, an $s$-fold wedge of $0$-spheres.

In the case of a pedestal of the form $\Dbip{X}{Y}$, one notes that
$\Gbip{X}{Y}(D)$ is not changed up to isomorphism if one relabels
the ordered set $Y=\{y_1< \cdots <y_n\}$ of column indices in
{\it backwards order}, i.e. replace $y_j$ with $y_{n+1-j}$.  This has no effect
on $\Gbip{X}{Y}(D)$ up to graph isomorphism, nor on $\Delta(\Gbip{X}{Y}(D))$
up to simplicial isomorphism.  However, now the diagram $\Dbip{X}{Y}$ is no longer a pedestal, but
rather has a rectangular decomposition in two iterations:  the
first creates a full rectangle and labels all the remaining cells
as excess cells, while the second iteration creates an empty rectangle of zero width.
An example is shown here
$$
\begin{matrix}
   &y_1&y_2&y_3&y_4&y_5&y_6\\
x_1&p&p&p&p&p&{\mathbf p}\\
x_2&p&p&p&p&p&p\\
x_3& & &p&p&p&p\\
x_4& & & &p&p&p\\
\end{matrix}
\qquad \rightsquigarrow \qquad
\begin{matrix}
   &y_1&y_2&y_3&y_4&y_5&y_6\\
x_1&r_1&r_1&r_1&r_1&r_1&{\mathbf r_1}\\
x_2&r_1&r_1&r_1&r_1&r_1&r_1\\
x_3&e  &e  &e  &e  &  \\
x_4&e  &e  &e  &   &  \\
\end{matrix}
$$
in which the rectangular decomposition for the diagram on the right creates
the full rectangle $\Dbip{\{x_1,x_2\}}{Y}$ and removes $5$ excess cells in the
first iteration, then creates the empty rectangle $\Dbip{\{x_3,x_4\}}{\varnothing}$ in the second iteration.
Thus pedestals of the form $\Dbip{X}{Y}$ have $\Delta(\Gbip{X}{Y}(D))$ contractible.

The homotopy type analysis of these base cases then completes the proof, bearing in mind the
following homotopy-theoretic properties\footnote{These properties are reasonably well-known.
They may be deduced, for example,
from the analogous but perhaps better-known properties \cite[\S III.2]{Whitehead} of the associative
{\it smash product} (or {\it reduced join}) operation $X \wedge Y$, using
the fact that the join $X * Y$ of $X$ and $Y$ is homotopy equivalent to the suspension of
$X \wedge Y$, or equivalently, $\sphere^1 \wedge X \wedge Y$ \cite[\S X.8.III]{Whitehead}.}
of the join operation:
\begin{enumerate}
\item[$\bullet$] A join with a contractible complex yields a contractible complex.
\item[$\bullet$] The join of a $d_1$-dimensional sphere (up to homotopy equivalence) and a
a $d_2$-dimensional sphere (up to homotopy equvialence) gives a $(d_1+d_2+1)$-dimensional
sphere (up to homotopy equivalence).
\item[$\bullet$] Forming joins commutes (up to homotopy equivalence)
with taking wedges.
\end{enumerate}
\end{proof}

Hochster's formula (Proposition \ref{Hochster-formula}) combined with Theorem~\ref{homotopy-type-theorem}
immediately yields the following.

\begin{corollary}
\label{betti-number-corollary}
For any shifted skew diagram $D$ and any ordered subsets $X,Y$, the ideals
$I(\Gbip{X}{Y}(D))$ and $I(\Gnonbip{X}(D))$ have multigraded Betti numbers independent of the coefficient
field $k$:
$$
\begin{aligned}
\beta_{i,X' \sqcup Y'}(I(\Gbip{X}{Y}(D))) &=
\begin{cases}
1 &\text{ if }\Dbip{X'}{Y'}\text{ is spherical with }\rect(\Dbip{X'}{Y'})= |X' \cup Y'|-i-1 \\
0 &\text{ otherwise.}
\end{cases} \\
\beta_{i,X'}(I(\Gnonbip{X}(D)))&=
\begin{cases}
s &\text{ if }\Dnonbip{X'}\text{ has no empty rectangles, has }\rect(\Dnonbip{X'})=|X'|-i-1\\
  &\quad\text{ and has }s\text{ non-excess staircase cells}\\
0 &\text{ otherwise.}
\end{cases}
\end{aligned}
$$
\end{corollary}

\subsection{Case study: Ferrers diagrams and rook theory}
\label{Ferrers-example}

We analyze here in detail the example of Ferrers diagrams, recovering results from \cite{CN1},
and noting a curious connection to rook theory.

Recall from Example~\ref{Ferrers-definition} that for a partition $\lambda=(\lambda_1 \geq \cdots \geq \lambda_m)$,
the Ferrers graph $G_\lambda$ corresponds to a diagram $\Dbip{X}{Y}$ having $\lambda_i$ cells in row $i$,
namely $\{(x_i,y_j): 1 \leq i \leq m, \,\, 1 \leq j \leq \lambda_i\}$.

\begin{definition} \rm \ \\
Say that the cell $(x_i,y_j)$ in the Ferrers diagram for $\lambda$
lies on the $k^{th}$ antidiagonal if $k=i+j$, and let $\alpha_k(\lambda)$ for
$k=2,3,\ldots$ denote the number of cells on the $k^{th}$ antidiagonal.
\end{definition}

For example, if $\lambda=(4,4,2)$ then
$
(\alpha_2(\lambda),\alpha_3(\lambda),\alpha_4(\lambda),\alpha_5(\lambda),\alpha_6(\lambda))=(1,2,3,3,1)
$
with the diagram corresponding to $G_\lambda$ shown below, having cells labelled according to the
antidiagonal on which they lie
$$
\begin{matrix}
   & y_1&y_2&y_3&y_4\\
x_1& 2 & 3 & 4 & 5 \\
x_2& 3 & 4 & 5 & 6 \\
x_3&4 & 5 &   &
\end{matrix}
$$

Given $X' \subseteq X, \,\, Y' \subseteq Y$ say that $X' \times Y' \subseteq \lambda$ if
$X'$ and $Y'$ are non-empty and
the full rectangle $X' \times Y'$ is covered by cells in the diagram $\Dbip{X}{Y}$ corresponding to
$G_\lambda$.

\begin{proposition}
\label{Ferrers-betti-numbers}
For any partition $\lambda=(\lambda_1 \geq \cdots \geq \lambda_m > 0)$, consider the Ferrers (bipartite)
graph $G_\lambda$ on vertex set $X \sqcup Y$ where $X=\{x_1,\ldots,x_m\}$ and
$Y=\{y_1,\ldots,y_{\lambda_1}\}$.  Then for all $i \geq 0$ one has
\begin{equation}
\label{Ferrers-finest-Betti-formula}
\begin{aligned}
\beta_{i,X' \sqcup Y'}(I(G_\lambda))
  &= \begin{cases}
        1 & \text{ if }|X'|+|Y'|=i+2\text{ and }X' \times Y'\subseteq\lambda \\
        0 & \text{ otherwise}
     \end{cases} \\
  &\text{ for all }X' \subseteq X, \,\, Y' \subseteq Y.
\end{aligned}
\end{equation}

\begin{equation}
\label{Ferrers-next-finest-Betti-formula}
\begin{aligned}
\beta_{i,X',\bullet}(I(G_\lambda))
  &:= \sum_{Y' \subseteq Y} \beta_{i,X' \sqcup Y'}(I(G_\lambda))\\
  &=\begin{cases}
     \binom{\lambda_m}{i-|X'|+2} &\text{ if }|X'| < i+2\\
     0 &\text{ otherwise.}
    \end{cases}
\end{aligned}
\end{equation}

\begin{equation}
\label{Ferrers-coarsest-Betti-formula}
\begin{aligned}
\beta_i(I(G_\lambda))
  &=|\{ (X',Y'): |X'|+|Y'|=i+2\text{ and }X' \times Y'\subseteq\lambda \}|\\
  &=\sum_{m'=1}^m \sum_{n'=1}^{\lambda_{m'}} \binom{m'+n'-2}{i} \\
  &=\sum_{k \geq 2} \alpha_k(\lambda) \binom{k-2}{i}\\
  &=\binom{\lambda_1}{i+1} + \binom{\lambda_2+1}{i+1} + \cdots + \binom{\lambda_m+m-1}{i+1} - \binom{m}{i+2}.
\end{aligned}
\end{equation}
\end{proposition}
\begin{proof}
A Ferrers diagram $\Dbip{X}{Y}$ is easily seen to be spherical if and only if
it is a full rectangle $X \times Y$, which will always have $\rect(\Dbip{X}{Y})=1$.
Thus Corollary~\ref{betti-number-corollary} immediately gives \eqref{Ferrers-finest-Betti-formula},
which then immediately implies \eqref{Ferrers-next-finest-Betti-formula},
as well as the first formula in \eqref{Ferrers-coarsest-Betti-formula}.

The second formula in \eqref{Ferrers-coarsest-Betti-formula} follows
from the first formula by classifying the spherical subdiagrams
$X' \times Y'$ inside $\lambda$ having $|X'|+|Y'|=i+2$ according to their
southeasternmost cell $(x_{m'},y_{n'})$ so that
$$
\begin{aligned}
m'&=\max X'\\
n'&=\max Y'.
\end{aligned}
$$
One can check that there are exactly $\binom{m'+n'-2}{i}$ such rectangular subdiagrams.
The third formula in \eqref{Ferrers-coarsest-Betti-formula} then
comes from grouping the second formula according to the value $k=m'+n'$.

The last formula in \eqref{Ferrers-coarsest-Betti-formula} (which is equivalent to one stated
in \cite[Theorem 2.1]{CN1}) comes from summing the inner summation in the second formula
of \eqref{Ferrers-coarsest-Betti-formula}.  One has
$$
\sum_{n'=1}^{\lambda_{m'}} \binom{m'+n'-2}{i}
 = \binom{\lambda_{m'}+m'-1}{i+1}-\binom{m'-1}{i+1}
$$
and then one uses the fact that $\sum_{m'=1}^{m} \binom{m'-1}{i+1}=\binom{m}{i+2}$.
\end{proof}

\noindent
We remark that the formulae in Proposition~\ref{Ferrers-betti-numbers} will
also apply to {\it row-nested graphs} which appear later (Section~\ref{bipartite-conjecture-section})
as these are exactly the bipartite graphs isomorphic to Ferrers graphs.

These formulae also allow one to compare the Betti numbers of different Ferrers graphs,
and lead to a curious corollary relating to the combinatorial theory of {\it rook placements}.
Given a diagram $D \subset \Z \times \Z$,
call an $r$-element subset of $D$ a {\it (non-attacking) rook placement} on $D$ if
no two of the $r$ squares share any row or column.  Say that two diagrams $D, D'$ in
the plane $\Z \times \Z$ are {\it rook-equivalent} if they have the same number of
$r$-element rook placements for all $r$.  In particular, taking $r=1$, this means $D, D'$
must have the same number of cells, but in general, it is a somewhat subtle equivalence relation.
However, when one restricts the equivalence relation to Ferrers diagrams, rook-equivalence
has a nice characterization, due originally to Foata and Sch\"utzenberger,
elegantly reformulated by Goldman, Joichi, and White, and
reformulated further in the following fashion by Ding \cite{Ding}.

\begin{proposition}
\label{Ding-prop}
Given two partitions $\lambda, \mu$, their associated Ferrers diagrams are rook
equivalent if and only if $\alpha_k(\lambda)=\alpha_k(\mu)$ for all $k$.
\end{proposition}

\begin{corollary}
\label{curious-rook-corollary}
For two partitions $\lambda, \mu$, the
Ferrers graph ideals $I(G_\lambda), I(G_\mu)$ have the same (ungraded) Betti numbers $\beta_i$ for all $i$
if and only if $\alpha_k(\lambda)=\alpha_k(\mu)$ for all $k$, that is,
if and only if $\lambda, \mu$ are rook equivalent.
\end{corollary}

\begin{proof}
The formula $\beta_i(I(G_\lambda) =\sum_{k \geq 2} \alpha_k(\lambda) \binom{k-2}{i}$
in Proposition~\ref{Ferrers-betti-numbers} gives a linear relation
between the vectors $(\beta_i(I(G_\lambda)))_{i\geq 2}$ and $(\alpha_k(\lambda))_{k \geq 2}$, governed by
an invertible matrix of coefficients.  This yields the first equivalence. The
second follows from Proposition~\ref{Ding-prop}.
\end{proof}

\subsection{Specialization from bipartite to nonbipartite graphs}
\label{sec-specialization}

The goal of this section is Theorem~\ref{specialization-theorem}, showing that
$I(\Gbip{X}{X}(D))$ is a well-behaved polarization of $I(\Gnonbip{X}(D))$,
generalizing results from \cite{CN2}.  This turns out to be very useful
later when proving results about various invariants of these ideals (e.g.,
Castelnuovo-Mumford regularity, Krull dimension, projective dimension, agreement
with conjectural resolution bounds);  it is generally much easier to prove things
directly for $I(\Gbip{X}{Y}(D))$ and then apply Theorem~\ref{specialization-theorem} to
deduce the corresponding result for $I(\Gnonbip{X}(D))$.

Given a shifted skew diagram $D$ with rows and columns indexed by
$[n]:= \{1,2,\ldots,n\}$, we have seen how to associate with it two ideals
in two different polynomial rings over a field $k$:
$$
\begin{aligned}
I(\Gbip{[n]}{[n]}(D)) &\subset k[x_1,\ldots,x_n,y_1,\ldots,y_n]:=k[\xb,\yb] \\
I(\Gnonbip{[n]}(D)) &\subset k[x_1,\ldots,x_n]:=k[\xb]
\end{aligned}
$$
For both ideals we have seen how to compute multigraded Betti numbers, which we now wish
to compare via a certain {\it specialization} of the $\Z^{2n}$-grading on
$k[x_1,\ldots,x_n,y_1,\ldots,y_n]$ to a $\Z^n$-grading. Consider the map
$$
\begin{matrix}
 \{x_1,\ldots,x_n,y_1,\ldots,y_n\} & \overset{sp}{\rightarrow} & \{x_1,\ldots,x_n\} \\
     x_i            & \mapsto                   & x_i \\
     y_j            & \mapsto                   & x_j
\end{matrix}
$$
and the associated map of the gradings $\Z^{2n} \overset{sp}{\rightarrow} \Z^n$
that sends the standard basis vectors $e_i, e_{n+i} \mapsto e_i$ for $i=1,2,\ldots,n$.
Using this to define a $\Z^n$-grading on $k[x_1,\ldots,x_n,y_1,\ldots,y_n]$, one has for any
multidegree $\alpha \in \Z^n$ a specialized Betti number $\beta^{sp}_{i,\alpha}(I(\Gbip{[n]}{[n]}(D))$.

\begin{theorem}
\label{specialization-theorem}
For $D$ a shifted skew diagram with rows and columns indexed by $[n]$, one has
\begin{equation}
\label{first-specialization-equation}
\beta_{i,\alpha}(I(\Gnonbip{[n]}(D)) =
\beta^{sp}_{i,\alpha}(I(\Gbip{[n]}{[n]}(D))
\end{equation}
for all $\alpha \in \Z^n$.

Equivalently,
\begin{enumerate}
\item[(i)] for all $X,Y \subseteq [n]$ one has
$\beta_{i,X \sqcup Y}(I(\Gbip{[n]}{[n]}(D))=0$ unless $X \cap sp(Y) = \varnothing$, and
\item[(ii)] for all $Z \subseteq [n]$, one has
$$
\beta_{i,Z}(I(\Gnonbip{[n]}(D)) = \sum_{\substack{X,Y \subseteq [n]: \\ X \sqcup sp(Y)=Z}}
\beta_{i,X \sqcup Y}(I(\Gbip{[n]}{[n]}(D)).
$$
\end{enumerate}
\end{theorem}

\begin{proof}
We leave the discussion of the equivalence of the stated conditions to the reader,
except for pointing out that (i) is a consequence of \eqref{first-specialization-equation}
because the squarefree monomial ideal $I(\Gbip{[n]}{[n]}(D))$
can have non-trivial Betti numbers only in the squarefree multidegrees $\delta \in \{0, 1\}^{2n}$.
%

To prove (i), if $X \cap sp(Y) \neq \varnothing$, say if an index $j$ lies in both $X$ and in $Y$, we will show that
$\Delta(\Gbip{X}{Y}(D))$ is contractible and
hence $\beta_{i,X \sqcup Y}(I(\Gbip{[n]}{[n]}(D))=0$.
Contractibility comes from the fact that either $\Dbip{X}{Y}$ has
\begin{enumerate}
\item[$\bullet$] no cells in row $j$, so $\Delta(\Gbip{X}{Y}(D))$ has a cone vertex, or
\item[$\bullet$] no cells in column $j$, so $\Delta(\Gbip{X}{Y}(D))$ has a cone vertex, or
\item[$\bullet$] some cell $c$ in row $j$, {\it and} some cell $c'$ in column $j$.
But there is no cell of $\Dbip{X}{Y}$ in position $(j,j)$, which is the southwest corner of the rectangle
defined by $c$ and $c'$  (since $D$ itself has no such cell, as $(j,j)$ is not even a cell in the shifted plane).
Hence $\Dbip{X}{Y}$ contains a pedestal by Proposition~\ref{pedestal-proposition},
and $\Delta(\Gbip{X}{Y}(D))$ is contractible by Theorem~\ref{homotopy-type-theorem}.
\end{enumerate}

To prove (ii),  note that one may assume $Z=[n]$ without loss of generality.
Also note that the only non-zero summands on the right side of the equation in (ii)
are $X,Y \subset [n]$ with $X \sqcup sp(Y)=[n]$ for which $\Delta(\Gbip{X}{Y}(D))$ is not contractible.
Thus we wish to show
\begin{equation}
\label{specialization-equation}
\beta_{i,[n]}(I(\Gnonbip{[n]}(D))) =
\sum_{\substack{X,Y \subseteq [n]: \\ X \sqcup sp(Y)=[n] \\ \Delta(\Gbip{X}{Y}(D))\text{ not contractible}}}
\beta_{i,X \sqcup Y}(I(\Gbip{[n]}{[n]}(D))).
\end{equation}

Given each pair $X, Y$ appearing in the right side of \eqref{specialization-equation},
the proof is completed in three steps.
\begin{enumerate}
\item[{\sf Step 1.}] Show that $D$ has a top cell if and only if $\Dbip{X}{Y}$ does.
\item[{\sf Step 2.}] Show that if they both have a top cell, then the rectangular decomposition for $D$ begins with
a full rectangle (not a pedestal) if and only if the same is true for $\Dbip{X}{Y}$, and furthermore these
two full rectangles are {\it exactly the same}.
\item[{\sf Step 3.}]  One is reduced to the case where $D$ starts its
rectangular decomposition with a pedestal, which must be analyzed.
\end{enumerate}

\vskip .1in
\noindent
{\sf Step 1.}
Note that column $1$ and row $n$ are both empty in $D$.
Hence non-contractibility of $\Delta(\Gbip{X}{Y}(D))$ implies $1 \not\in Y$ and
$n \not\in X$.  But $X \sqcup Y=[n]$, so this forces $1 \in X, n \in Y$.
Thus $D$ contains a top cell, namely $(1,n)$ if and only if
$\Dbip{X}{Y}$ does.

\vskip .1in
\noindent
{\sf Step 2.}
Assume without loss of generality that both $D$ and $\Dbip{X}{Y}$ contain the top cell
$(1, n)$.
Assume that the first step in the rectangular decomposition for $D$ finds a full rectangle $\Dbip{X'}{Y'}$, say with
neck cell $(m',n')$.  The first step in the rectangular decomposition for $\Dbip{X}{Y}$ finds either a full rectangle
or pedestal $\Dbip{X''}{Y''}$.    We wish to carefully argue that these are the same, i.e. that $X'' = X'$ and $Y''=Y'$.

Start by noting that $1 \in X'', n \in Y''$.  One can characterize $X'$ as the largest initial segment of $[n]$ with
the property that $X' \times \{n\} \subset D$.  Similarly one has that $X''$ is the largest initial segment of $X$
with $X'' \times \{n\} \subset \Dbip{X}{Y}$.  But this implies that $X''=X' \cap X$.  Similarly one can argue that
$Y''= Y' \cap Y$.  Thus it remains to show that $X' \subset X$ and $Y' \subset Y$.

To argue this, we must first ``prepare'' $\Dbip{X}{Y}$ by possibly removing some of its excess cells.  Given any
cell $c=(i,j)$ in $\Dbip{X}{Y}$ that has both $i,j \in X'$, we claim that $c$ is an excess cell to the left of
the first rectangle $\Dbip{X''}{Y''}$.  To see this claim, we need to check that its row index $i$ lies in $X''$ and
that its column index $j$ is less than any element of $Y''$.  The first fact is true since $i \in X' \cap X = X''$.
The second follows because $j \in X'$ implies
$$
j \leq \max X' = m' < n' = \min Y' \leq \min Y'';
$$
the relation $m' < n'$ comes from the fact that $(m', n')$ is a cell of $D$ (so it lies in the shifted plane),
while the last inequality is a consequence of the fact that $Y'' = Y' \cap Y \subseteq Y'$.

Thus without loss of generality, $\Dbip{X}{Y}$ has no cells in $(i,j)$ with both $i,j \in X'$; they are all
excess cells which can be removed without affecting $\Delta(\Gbip{X}{Y}(D))$ up to homotopy.  This means
$\Dbip{X}{Y}$ has all of the columns indexed by $X'$ empty.  Non-contractibility of $\Delta(\Gbip{X}{Y}(D))$
then forces $X' \cap Y = \varnothing$.  Together with $X \sqcup Y = [n]$, this implies, $X' \subseteq X$,
and hence $X'' = X' \cap X = X'$, as desired.  A symmetric argument shows $Y''=Y'$, completing Step 2.

\vskip .1in
\noindent
{\sf Step 3.}
By Steps 1 and 2, one may assume without loss of generality that $D$ produces a pedestal in the first (and only)
step of its rectangular decomposition.  One must show why equation
\eqref{specialization-equation} holds in this case.

We claim non-contractibility of $\Delta(\Gbip{X}{Y}(D))$ has strong consequences for the form of $X$ and $Y$.
It forces any row $i$ in $X$ to contain at least one cell of $\Dbip{X}{Y}$; call this cell $c$.  Similarly, any column $j$ in $Y$ contains at least
one cell of $\Dbip{X}{Y}$; call this cell $c'$.  Non-contractiblity also forces $i<j$ for any such $i$ in $X$ and $j$ in $Y$:
if $i \geq j$, then the cell $(i,j)$ that would be the southwest corner of the rectangle defined by $c, c'$ is not
in $\Dbip{X}{Y}$ (since it is not in the shifted plane), and hence $\Dbip{X}{Y}$ has a pedestal by
Proposition~\ref{pedestal-proposition} and $\Delta(\Gbip{X}{Y}(D))$ is
contractible by Theorem~\ref{homotopy-type-theorem}.  In other words, $\max X < \min Y$, which combined with
$X \sqcup Y=[n]$ forces
$$
\begin{aligned}
X&=\{1,2,\ldots,j\}\\
Y&=\{j+1,j+2,\ldots,n\}
\end{aligned}
$$
for some $j=1,2,\ldots,n-1$.
One can also check that $\Dbip{\{1,2,\ldots,j\}}{\{j+1,j+2,\ldots,n\}}$
is a full rectangle if $(j,j+1)$ is a non-excess staircase cell
in the pedestal of $D$, and otherwise $\Delta(\Gbip{\{1,2,\ldots,j\}}{\{j+1,j+2,\ldots,n\}})$ is contractible.
Thus \eqref{specialization-equation} holds because both sides
\begin{enumerate}
\item[$\bullet$]
vanish for $i \neq n-2$, and
\item[$\bullet$]
are equal to the number of (non-excess) staircase cells in the pedestal of $D$ for $i = n-2$.
\end{enumerate}
\end{proof}

The following result includes consequences for the Castelnuovo-Mumford regularity and the projective dimension. We refer to Subsection \ref{sec-regularity} for the definitions.

\begin{corollary}
\label{specialization-corollary}
In the setting of Theorem~\ref{specialization-theorem},
if $X = Y$, then
one has
\begin{enumerate}
\item[(i)] $\beta_{ij}(I(\Gbip{X}{Y}(D))) =\beta_{ij}(I(\Gnonbip{X}(D)))$ for all $i,j$.
In particular, the two ideals share the same projective dimension and
Castelnuovo-Mumford regularity.
\item[(ii)] The linear forms $\theta_1,\ldots,\theta_n$ where
$\theta_i := x_i - y_i$ have images in the quotient $k[\xb,\yb]/I(\Gbip{X}{Y}(D))$ forming
a regular sequence.
\item[(iii)] A minimal free resolution for $I(\Gnonbip{X}(D))$ as $k[\xb]$-module can be
obtained from a minimal free resolution for $I(\Gbip{X}{Y}(D))$ as $k[\xb,\yb]$-module,
simply by modding out $(\theta):=(\theta_1,\ldots,\theta_n)$, that is, by tensoring
over $k[\xb,\yb]$ with $k[\xb,\yb]/(\theta)$.
\end{enumerate}
\end{corollary}
\begin{proof}
Assertion (i) follows from Theorem~\ref{specialization-theorem}
and Hochster's formula.
The remaining assertions are seen to be equivalent to it by
iterating Lemma~\ref{polarization-lemma} from Appendix B.
\end{proof}

\begin{example}
Such specializations do {\it not} work so well for an arbitrary bipartite graph $G$
and its edge ideal $I(G) \subset k[\xb,\yb]$.  In other words, it is {\it not} in general
true that the specialized ideal $I^{\nonbip} \subset k[\xb]$
for which $k[\xb,\yb]/(I(G) + (\theta)) = k[\xb]/I^{\nonbip}$ has
$\beta_{ij}(I^{\nonbip})=\beta_{ij}(I(G))$.

For example, let $G$ be the bipartite graph on vertex set $X \sqcup Y=\{x_1,\ldots,x_5,y_1,\ldots,y_5\}$
for which
$$
\begin{aligned}
I(G) &=( x_1 y_3,\,\, x_1 y_4,\,\, x_2 y_3,\,\, x_2 y_5,\,\, x_3 y_4,\,\, x_3 y_5 ), \text{ and } \\
I^{\nonbip}&=( x_1 x_3,\,\, x_1 x_4,\,\, x_2 x_3,\,\, x_2 x_5,\,\, x_3 x_4,\,\, x_3 x_5 ).
\end{aligned}
$$
This bipartite graph $G$ is a $6$-cycle, which one can check is {\it not}
of the form $\Gbip{X}{Y}(D)$ for any shifted skew-shape $D$.  However, one can still think of the
edges of $G$ as corresponding to the cells of a diagram in
the shifted plane, which would look like this:
$$
\begin{matrix}
 &1&2&3&4&5\\
1&\cdot  &       & \times& \times &        \\
2&\cdot  & \cdot & \times&        & \times \\
3&\cdot  & \cdot & \cdot & \times & \times \\
4&\cdot  & \cdot & \cdot & \cdot  &         \\
5&\cdot  & \cdot & \cdot & \cdot  & \cdot
\end{matrix}
$$

Here is the result of a {\tt Macaulay 2} calculation of their graded Betti numbers,
with $k=\Q$:
\begin{verbatim}

i1 : S=QQ[x1,x2,x3,x4,x5,y1,y2,y3,y4,y5];

i2 : IG=ideal(x1*y3,x1*y4,x2*y3,x2*y5,x3*y4,x3*y5);

o2 : Ideal of S

i3 : betti(resolution(IG))

            0 1 2 3 4
o3 = total: 1 6 9 6 2
         0: 1 . . . .
         1: . 6 6 . .
         2: . . 3 6 2

i4 : Snonbip=QQ[x1,x2,x3,x4,x5];

i5 : Inonbip=ideal(x1*x3,x1*x4,x2*x3,x2*x5,x3*x4,x3*x5);

o5 : Ideal of Snonbip

i6 : betti(resolution(Inonbip))

            0 1 2 3 4
o6 = total: 1 6 9 5 1
         0: 1 . . . .
         1: . 6 8 4 1
         2: . . 1 1 .

\end{verbatim}
\end{example}

\subsection{Castelnuovo-Mumford regularity}
\label{sec-regularity}

The next three subsections  discuss three natural invariants for
the ideals $I(\Dbip{X}{Y})$ and $I(\Dnonbip{X})$, namely their
\begin{enumerate}
\item[$\bullet$] {\it Castelnuovo-Mumford regularity},
\item[$\bullet$] {\it projective (or homological) dimension}, and
\item[$\bullet$] {\it Krull dimension} of the quotient rings
$k[\xb,\yb]/I(\Gbip{X}{Y}(D))$ and $k[\xb]/I(\Gnonbip{X}(D))$.
\end{enumerate}

Recall the definition of the {\it Castelnuovo-Mumford regularity} $\reg_S(M)$ for a
$\Z$-graded module $M$ over a regular $\Z$-graded $k$-algebra $S$:
$$
\reg_S(M)= \max \{ j-i: \beta_{ij}^S(M)(=\dim_k \Tor^S_i(M,k)_j) \neq 0\}.
$$
The goal of this section is Theorem~\ref{regularity-theorem}, which interprets
combinatorially the regularity for both classes of ideals $I(\Gbip{X}{Y}(D)), I(\Gnonbip{X}(D))$, in terms
of the quantity {\it rectangularity} defined in Definition~\ref{rectangularity-definition}
above.

\begin{theorem}
\label{regularity-theorem}
For any shifted skew diagram and ordered subsets $X,Y$, one has
$$
\begin{aligned}
\reg_{k[\xb,\yb]}(I(\Gbip{X}{Y}(D)))&=\rect(\Dbip{X}{Y})+1 \\
\reg_{k[\xb]}(I(\Gnonbip{X}(D)))&=\rect(\Dnonbip{X})+1
\end{aligned}
$$
\end{theorem}
\begin{proof}
Note that the assertion for $\Dnonbip{X}$ will follow after proving it
for $\Dbip{X}{Y}$, since
$$
\rect(\Dnonbip{X})=\rect(\Dbip{X}{X})
$$
by definition of
the rectangular decomposition, and
$$
\reg_{k[\xb]}(I(\Gnonbip{X}(D)))=\reg_{k[\xb,\yb]}(I(\Gbip{X}{X}(D)))
$$
by Theorem~\ref{specialization-theorem}.

To prove the assertion for $\Dbip{X}{Y}$, first note that
$$
\begin{aligned}
&\reg_{k[\xb,\yb]}(I(\Dbip{X}{Y})) \\
  &:= \max\{j-i: \beta^{k[\xb,\yb]}_{ij}(I(\Dbip{X}{Y})) \neq 0\}\\
  &= \max\{|X' \sqcup Y'|-i: X' \subseteq X, \,\, Y' \subseteq Y \text{ and }
           \beta^{k[\xb,\yb]}_{i, X' \sqcup Y'}(I(\Dbip{X}{Y})) \neq 0\}\\
  &= \max\{\rect(\Dbip{X'}{Y'})+1:
          X' \subseteq X, \,\, Y' \subseteq Y, \text{ and }\Dbip{X'}{Y'}\text{ is spherical } \}
\end{aligned}
$$
where the last equality comes from Corollary~\ref{betti-number-corollary}.

To show the inequality $\reg_{k[\xb,\yb]}(I(\Dbip{X}{Y})) \geq \rect(\Dbip{X}{Y})+1$,
note that if one chooses $X', Y'$ to be the rows and columns occupied
by the union of all the full rectangles along with the first few equal-sized (i.e. longest)
rows in the pedestal (if present), then the subdiagram $\Dbip{X'}{Y'}$
is spherical with $\rect(\Dbip{X'}{Y'})=\rect(\Dbip{X}{Y})$.

The reverse inequality follows from Lemma~\ref{rectangularity-monotonicity} below.
\end{proof}

\begin{lemma}
\label{rectangularity-monotonicity}
For any
non-empty
diagram of the form $\Dbip{X}{Y}$ and any subsets $X' \subseteq X, \,\, Y' \subseteq Y$,
one has
$$
\rect(\Dbip{X'}{Y'}) \leq \rect(\Dbip{X}{Y})
$$
\end{lemma}
\begin{proof}
Prove this by induction on $|X|+|Y|$.  The base case where $|X|+|Y|=1$ is trivial.
For the inductive step, it suffices to prove that when one removes a row or column from
$\Dbip{X}{Y}$, the rectangularity cannot go up.  Without loss of generality one is removing
a
non-empty
column $C$ from $E:=\Dbip{X}{Y}$, and we wish to show that
\begin{equation}
\label{desired-rectangularity-inequality}
\rect(E \setminus C) \leq \rect(E).
\end{equation}
By induction, one may assume that $E$ has a top cell, else one can remove an empty row or column from $E$,
leaving both $\rect(E), \rect(E \setminus C)$ unchanged.
Hence the first step in the rectangular decomposition for $E$ identifies either a full rectangle or pedestal.
If it is a pedestal, then $\rect(E)= \rect(E \setminus C)=1$.  Thus without loss of generality
one may assume that the first step identifies a full rectangle $R$;  let $E^-$ denote the remaining
diagram after one removes from $E$ the rows and columns occupied by this full rectangle $R$.

For most choices of the column $C$, one has that $E \setminus C$ shares the same top cell as $E$,
and begins its rectangular decomposition with the full rectangle $R \setminus C$
or $R$. In the second case,
one has $(E \setminus C)^- = E^- \setminus C^-$ for some column $C^-$.  Using
\begin{equation}
\label{rectangularity-recursion}
\begin{aligned}
\rect(E) &= \rect(E^-)+ 1\\
\rect(E \setminus C) &= \rect((E \setminus C)^-) + 1
\end{aligned}
\end{equation}
along with the inductive hypothesis applied to $E^-$, one obtains the desired inequality
\eqref{desired-rectangularity-inequality}.

In the first case, we have $E^- = (E \setminus C)^-$ and we argue by induction, unless
$C$ is the rightmost column $C_n$, {\it and} the
column $C_{n-1}$ second from the right occupies a different set of rows from those
occupied by $C_n$.

\vskip .1in
\noindent
{\sf Case 1.} The column $C_{n-1}$ starts in the same (top) row as the column $C_n$, but is longer and hence extends to
lower rows than $C_n$.  Here one finds that $E \setminus C_n$ begins its rectangular decomposition
with a full rectangle that occupies more rows than $R$.  Hence when this larger rectangle is removed from
$E \setminus C_n$, one finds that $(E \setminus C)^-$ is obtained from $E^-$ by removing
some {\it rows}, and so the inductive hypothesis applies to show
$\rect((E \setminus C)^-) \leq \rect(E^-)$.  Then \eqref{rectangularity-recursion}
gives the desired inequality \eqref{desired-rectangularity-inequality}.

\vskip .1in
\noindent
{\sf Case 2.}
The column $C_{n-1}$ does not start in the top row, unlike column $C_n$.
In this case $R=C_n$ is the entire first full rectangle in the decomposition for $E$.
Since column $C_{n-1}$ does not start in the top row, it must extend down to at least as many
rows as column $C_n$ does, or further.  This means that
$(E \setminus C)^-$ is obtained from
$E^-$ by removing some columns (at least the column $C_{n-1}$) and possibly also some rows.
Thus,
the inductive hypothesis again shows $\rect((E \setminus C)^-) \leq \rect(E^-)$,
and one again applies \eqref{rectangularity-recursion} to conclude the
desired inequality \eqref{desired-rectangularity-inequality}.
\end{proof}

\subsection{Krull dimension}

To interpret the Krull dimension of the quotient rings $k[\xb,\yb]/I(\Gbip{X}{Y}(D))$
and $k[\xb]/I(\Gnonbip{X}(D))$, Corollary~\ref{specialization-corollary} again says that one only
needs to do this for $k[\xb,\yb]/I(\Gbip{X}{Y}(D))$.

For {\it any} bipartite graph $G$ on vertex set $X \sqcup Y$ with edges $E(G)$ (not necessarily
of the form $\Gbip{X}{Y}(D)$), the Krull dimension for $k[\xb,\yb]/I(G)$ is the quantity $\alpha(G)$
equal to the maximum size of a {\it coclique} ({\it stable set}, {\it independent set}) of vertices.
This quantity $\alpha(G)$ is one of four graph invariants for a graph $G=(V,E)$
closely related by classical theorems of graph theory (see e.g. \cite[Chapter 3]{West}), which we review here:
$$
\begin{aligned}
\alpha(G)&:= \max\{ |C| : C \subset V \text{ is a coclique, i.e. }C\text{ contains no vertices that share an edge} \} \\
\tau(G)&:= \min\{ |F| : F \subset E \text{ is an edge cover, i.e. }F\text{ is incident to all of }V\} \\
\nu(G)&:=\max\{ |M| : M \subset E \text{ is a matching, i.e. }M\text{ contains no edges that share a vertex} \} \\
\rho
(G)&:= \min\{ |W| : W \subset V \text{ is a vertex cover, i.e. }W\text{ is incident to all of }E\}. \\
\end{aligned}
$$
Gallai's Theorem asserts that for any graph $G$ one has
$$
\alpha(G)+\tau(G)=|V|=\nu(G)+\rho(G)
$$
while K\"onig's Theorems assert that for a bipartite graph $G$ one has
$$
\alpha(G)=\rho(G) = |V|-\tau(G) = |V|-\nu(G).
$$
There are very efficient algorithms (e.g. the {\it augmenting path algorithm})
for computing $\alpha(G)$ by finding a maximum-cardinality matching in a bipartite
graph $G$.  Hence the Krull dimension $\rho(G)=\alpha(G)$ is easy to compute for $k[\xb,\yb]/I(G)$ of
any bipartite graph $G$.  We do not know of a
faster algorithm tailored to the specific class of bipartite graphs $\Gbip{X}{Y}(D)$ when
$D$ is a shifted skew diagram.

\subsection{Projective dimension}

Recall that the {\it projective (or homological) dimension} $\pd_S(M)$ for a finitely-generated
module $M$ over a polynomial algebra $S$ is the length of any minimal free $S$-resolution of $M$,
that is, the largest $i$ for which $\beta^S_{i}(M) \neq 0$.  Also recall that for
for any ideal $I$, since $\beta_{i}(I) = \beta_{i+1}(S/I)$, one has $\pd_S(I) = \pd_S(S/I)-1$.

In studying the {\it projective (or homological) dimension} of
the ideals $I(\Gbip{X}{Y}(D))$ and $I(\Gnonbip{X}(D))$, one is again reduced
to studying the former, as Theorem~\ref{specialization-theorem} implies
$$
\pd_{k[\xb]} I(\Gnonbip{X}(D)) = \pd_{k[\xb,\yb]} I(\Gbip{X}{X}).
$$

For the latter, one at least has the
following combinatorial interpretation.

\begin{proposition}
Given any shifted skew diagram $D$ and ordered subsets $X, Y$,
one has
$$
\pd_{k[\xb,\yb]} I(\Gbip{X}{Y}(D))
   = \max_{X',Y'}\{ |X'|+|Y'|-\rect(\Dbip{X'}{Y'})-1 \}
$$
where the maximum runs over all subsets $X' \subseteq X, \,\, Y' \subseteq Y$
for which $\Dbip{X'}{Y'}$ is spherical.
\end{proposition}
\begin{proof}
This is immediate from Corollary~\ref{betti-number-corollary},
since the $X', Y'$ with $\Dbip{X'}{Y'}$ spherical
are the ones which contribute to nonzero $\beta_{i}$,
namely with $i= |X'|+|Y'|-\rect(\Dbip{X'}{Y'})-1$.
\end{proof}

One might hope that this maximum can be computed quickly from the rectangular decomposition,
but this is not even true in the case where $\Dbip{X}{Y}$ looks like a single Ferrers diagram.
Here the rectangular decomposition is very simple, in that it has one full rectangle, followed possibly by one empty
rectangle.  However, the spherical subdiagrams $\Dbip{X'}{Y'}$
one must consider to compute the above maximum correspond to the corner cells of the Ferrers diagram;
cf. \cite[Corollary 2.2]{CN1}.

\begin{remark} \rm \ \\
Herzog and Hibi \cite[Corollary 3.5]{HH} have shown that, for each bipartite graph $G$, the ring $k[\xb,\yb]/I(G)$ is Cohen-Macaulay if and only if the projective variety defined by the edge ideal $I(G)$ is equidimensional and connected in codimension one. We suspect that the analogous conclusion is also true for a nonbipartite graph $\Gnonbip{X}(D)$.
\end{remark}

\section{PART II:  Skew hypergraph ideals}
\label{hypergraphs}

\subsection{Non-quadratic monomial ideals and hypergraphs}
\label{skew-hypergraphs-defn-section}

  Consider ideals $I$ in $k[\xb]:=k[x_1,\ldots,x_n]$
generated by squarefree monomial generators $x_{i_1} \cdots x_{i_d}$ of a fixed degree $d \geq 2$.
When the number of variables $n$ is allowed to vary, such ideals are parametrized by the collection
$$
K :=\{\{i_1,\ldots,i_d\}: x_{i_1} \cdots x_{i_d} \in I \} \subseteq \binom{\PP}{d}
$$
called a {\it $d$-uniform hypergraph}, where here $\PP:=\{1,2,\ldots\}$ denotes the positive
integers.  Our goal here is to introduce hypergraph
generalizations of the ideals $I(\Gnonbip{X}(D)), I(\Gbip{X}{Y}(D))$
coming from shifted skew diagrams,
as well as the Ferrers graph ideals $I(G_\lambda)$, in order to ask and answer questions about their resolutions.
For this it helps to consider certain orderings and pre-orderings
on the $d$-subsets $\binom{\PP}{d}$.

\begin{definition} \rm \ \\
Given two $d$-subsets
$$
\begin{aligned}
S= \{i_1 < \cdots < i_d\}\\
S'=\{ i'_1 < \cdots < i'_d \}
\end{aligned}
$$
say that $S \leq_{Gale} S'$ in the {\it Gale (or componentwise, or Bruhat) partial ordering}
on $\binom{\PP}{d}$ if $i_j \leq i'_j$ for all $j$.

Say that $S \leq_{max} S'$ in the {\it preordering by maxima} on $\binom{\PP}{d}$ if $i_d \leq i'_d$.

Say that $S \leq_{colex} S'$ in the {\it colexicographic (or squashed) linear ordering} on $\binom{\PP}{d}$
if
$S = S'$ or
the maximum element of the symmetric difference $S \Delta S':=(S \setminus S') \sqcup (S' \setminus S)$
lies in $S'$.

Note that
$$
S \leq_{Gale} S' \text{ implies } S \leq_{colex} S' \text{ implies } S \leq_{max} S'.
$$

For the sake of considering monomial ideals which are {\it not} necessarily squarefree,
define a {\it $d$-element multiset} of $\PP$ to be a sequence $(i_1,i_2,\ldots,i_d)$ with $i_j \in \PP$
and $i_1 \leq i_2 \leq \cdots \leq _d$.  Denote by $\binom{\PP+d-1}{d}$ the collection of
all such $d$-element multisets; clearly monomial ideals $I$ generated in degree $d$ are parametrized
by the collection\footnote{One might call this collection $M$ a {\it hypermultigraph}, but we will rather try to avoid choosing
some terminology for such an object!}
$$
M:=\{(i_1 \leq \ldots \leq i_d): x_{i_1} \cdots x_{i_d} \in I \} \subseteq \binom{\PP+d-1}{d}.
$$

Define the {\it Gale partial ordering} on $\binom{\PP+d-1}{d}$
by saying
$$
(i_1 \leq  \cdots \leq i_d) \quad \leq_{Gale} \quad (i'_1 \leq  \cdots \leq i'_d)
\qquad \text{ if }i_j \leq i'_j\text{ for }j=1,2,\ldots,d.
$$
Note that there is a simple {\it depolarization} bijection
$$
\begin{array}{llll}
\depol: & \binom{\PP}{d}        &\longrightarrow &\binom{\PP+d-1}{d} \\
        &\{i_1 < \cdots < i_d\} & \longmapsto    &(i_1,i_2-1,i_3-2, \ldots,i_d-(d-1))
\end{array}
$$
which is also an order-isomorphism between the Gale orders on these two sets.
\end{definition}

We omit the straightforward proof of the following easy properties of
the Gale orderings, which will be used in the proof of Theorem~\ref{cellular-resolutions-theorem}
below.

\begin{proposition}
\label{gale-properties}
The Gale orderings on $\binom{\PP}{d}$ and $\binom{\PP+d-1}{d}$ share the following properties.
\begin{enumerate}
\item[(i)] They are lattices with meet and join operations corresponding to
componentwise minimum and maximum, that is, if
$$
\begin{aligned}
v&=(i_1,\ldots,i_d) \\
v'&=(i'_1,\ldots,i'_d)
\end{aligned}
$$
then
$$
\begin{aligned}
v \wedge v'&=(\min\{i_1,i'_1\},\ldots,\min\{i_d,i'_d\}) \\
v \vee v'&=(\max\{i_1,i'_1\},\ldots,\max\{i_d,i'_d\}). \\
\end{aligned}
$$
\item[(ii)] They have the property that if $\xb^v, \xb^{v'}$ divide some monomial $\alpha$,
then $\xb^{v \wedge v'}$ also divides $\alpha$.
\end{enumerate}
\end{proposition}

\begin{definition} \rm \ \\
Generalizing threshold graphs, say that a $d$-uniform hypergraph $K \subseteq \binom{\PP}{d}$ is
{\it squarefree strongly stable} if it forms an order ideal in the Gale ordering on $\binom{\PP}{d}$.
Similarly, say that  collection $M \subseteq \binom{\PP+d-1}{d}$ is {\it strongly stable}
if it forms an order ideal in the Gale ordering on $\binom{\PP+d-1}{d}$.
The reason for the terminology\footnote{In an unfortunate clash of notation,
the squarefree strongly stable $d$-uniform hypergraphs $K$ are sometimes
called {\it shifted}, although they have nothing to do with the shifted plane occurring earlier
in this paper!  In yet another unfortunate clash of notation, the word {\it threshold} has been
used for a property of hypergraphs that is somewhat stronger than being squarefree strongly
stable; see \cite{KlivansReiner}.}
is that the associated squarefree monomial ideal
$I(K)$ (resp. monomial ideal $I(M)$)
generated by
$$
\{ x_{i_1} \cdots x_{i_k}: (i_1,\cdots , i_d) \in K \,\, (\text{resp. } M)\}
$$
is usually called a
{\it squarefree strongly stable (resp. strongly stable) ideal generated in degree $d$}.
\end{definition}

Eliahou and Kervaire\cite{EK} gave an explicit minimal free resolution for the more general class of
{\it stable} monomial ideals \cite{EK}, including those generated in different
degrees; Aramova, Herzog and Hibi \cite{AHH} gave an analogous resolution
for {\it squarefree  stable ideals}, again including those generated in different degrees.
In Theorem~\ref{cellular-resolutions-theorem} below, we
will recover an extremely simple cellular version of
these minimal free resolutions for both kinds of ideals,
when the ideals are generated in a single degree $d$.  In fact, we will show that
when $M=\depol(K)$, the two resolutions for $I(K)$ and $I(M)$ are in a precise sense,
the same.  The resolution for strongly stable ideals also reproves a recent result of
Sinefakopoulos \cite{Sinefakopoulos}, who produced such a cellular resolution by
a somewhat more complicated inductive process.  We have not checked whether his
cellular resolution is exactly the {\it same} as ours.

\begin{definition} \rm \ \\
Define a {\it skew squarefree strongly stable} $d$-uniform hypergraph be one of the form $K \setminus K'$
where $K', K$ are both squarefree strongly stable and $K' \subseteq K$; such hypergraphs have
been studied recently from the viewpoint of combinatorial Laplacians by Duval \cite{Duval}.

  Say that a $d$-uniform hypergraph is {\it $d$-partite} on a partitioned vertex set
$X^{(1)}\sqcup \cdots \sqcup X^{(d)}$ if each of its $d$-sets $\{i_1 < \ldots < i_d\}$ has $i_j \in X^{(j)}$
for all $j$.

  Given either a $d$-uniform hypergraph $K \subset \binom{\PP}{d}$,
or a finite collection $M \subset  \binom{\PP+d-1}{d}$,
we will associate to it a $d$-partite $d$-uniform hypergraph $F(K)$ or $F(M)$ on
$X^{(1)}\sqcup \cdots \sqcup X^{(d)}$ where $X^{(j)}:=\{1^{(j)},2^{(j)},\ldots\}$, namely
$$
\begin{aligned}
F(K):=\{ \{i^{(1)}_1 , i^{(2)}_2, \ldots, i^{(d)}_d\}: \{i_1 < \cdots < i_d\} \in K\} \\
F(M):=\{ \{i^{(1)}_1 , i^{(2)}_2, \ldots, i^{(d)}_d\}: (i_1 \leq \cdots \leq i_d) \in M\}. \\
\end{aligned}
$$
One also derives from these hypergraphs $F(K), F(M)$ certain ideals $I(F(K)), I(F(M))$ in a
polynomial algebra having $d$ different variable sets.
\end{definition}
\medskip

\begin{example} \rm \ \\
\label{hypergraph-ideals-example}
We illustrate this here for $d=3$, relabelling the partitioned vertex set
$X^{(1)} \sqcup X^{(2)} \sqcup X^{(3)}$ as
$$
\{a_1,a_2,\ldots\} \sqcup \{b_1,b_2,\ldots\} \sqcup \{c_1,c_2,\ldots\}
$$
to avoid superscripts:
$$
\begin{matrix}
K&=&\{123,&124,&134,&234,&125,&135\}\\
I(K)&=&(x_1 x_2 x_3,& x_1 x_2 x_4,& x_1 x_3 x_4,& x_2 x_3 x_4,& x_1 x_2 x_5,& x_1 x_3 x_5 )\\
I(F(K))&=&(a_1 b_2 c_3,&a_1 b_2 c_4,&a_1 b_3 c_4,&a_2 b_3 c_4,&a_1 b_2 c_5,&a_1 b_3 c_5)\\
\end{matrix}
$$
and letting $M:=\depol(K)$, one has
$$
\begin{matrix}
M&=&\{111,&112,&122,&222,&113,&123\}\\
I(M)&=&(x_1^3,&x_1^2 x_2,&x_1 x_2^2,&x_2^3,& x_1^2x_3,& x_1 x_2 x_3)\\
I(F(M))&=&(a_1 b_1 c_1,& a_1 b_1 c_2,& a_1 b_2 c_2,& a_2 b_2 c_2,& a_1 b_1 c_3,& a_1 b_2 c_3)\\
\end{matrix}
$$
\end{example}

In the next subsection, we will focus on the {\it non-skew}
special case where $K'$ is empty, generalizing threshold and Ferrers graph ideals,
by giving a simple cellular linear resolution for the ideals $I(K), I(M), I(F(K)), I(F(M))$
generalizing those from  \cite{CN1,CN2}, and which are in a precise sense,
all the same if $M=\depol(K)$.
In fact, the same methods will also apply to the following ideals, which are a slightly
different generalization of Ferrers graph ideals to hypergraphs.

\begin{definition}
Say that a $d$-partite $d$-uniform hypergraph $F$ on vertex set $X^{(1)}\sqcup \cdots \sqcup X^{(d)}$
is a {\it Ferrers hypergraph} if there is a linear ordering on each $X^{(j)}$ such that
whenever $(i_1,\ldots,i_d) \in F $ and $(i'_1,\ldots,i'_d)$ satisfies
$i'_j \leq i_j$ in $X^{(j)}$ for all $j$, one also has $(i'_1 ,\ldots,i'_d) \in F$.
In other words, $F$ is an order ideal in the componentwise partial ordering on
$X^{(1)} \times \cdots \times X^{(d)}$.

\end{definition}

The next proposition generalizes the fact that Ferrers graphs $G_\lambda$ are isomorphic to a subclass
of graphs of the form $\Gbip{X}{Y}(D)$ for shifted skew diagrams $D$.

\begin{proposition}
Every Ferrers $d$-uniform hypergraph $F$ is isomorphic to a $d$-partite $d$-uniform hypergraph
of the form $F(K \setminus K')$ with $K, K'$ squarefree strongly stable.
\end{proposition}
\begin{proof}
Let $F$ have partitioned vertex set $X^{(1)} \sqcup \cdots \sqcup X^{(d)}$,
and let $N:=\max_j\{ |X^{(j)}|\}$.  One can then regard the componentwise
ordering on $X^{(1)} \times \cdots \times X^{(d)}$ as a subposet of the componentwise
order $[N]^d$ for $[N]:=\{1,2,\ldots,N\}$, and $F \subseteq [N]^d$ as an order ideal.

Then the interval $[S_F,T_F]_{Gale}$
in the Gale ordering on $\binom{\PP}{d}$ between the sets
$$
\begin{aligned}
S_F&:=\{1 < N+1 < 2N+1 < \cdots < (d-1)N+1 \} \\
T_F&:=\{N < 2N < 3N < \cdots < dN\}
\end{aligned}
$$
has an obvious order-isomorphism
$$
\begin{array}{lll}
\phi: [S_F, T_F] & \longrightarrow &[N]^d\\
\{i_1 < \cdots < i_d\} & \longmapsto &(i_1,i_2-N,i_3-2N,\ldots,i_d-(d-1)N).\\
\end{array}
$$
The inverse image $\phi^{-1}(F)$ is an order ideal inside the interval $[S_F,T_F]_{Gale}$
which is order-isomorhpic to $F$.
Define the squarefree strongly stable hypergraphs
$$
\begin{aligned}
K&:=\{S \in \binom{\PP}{d}: \text{ there exist }S' \in \phi^{-1}(F)\text{ with }S' \geq S\}\\
K'&:=\{ S \in K: S \not\geq S_F\}.
\end{aligned}
$$
Then it is easily seen that  $F(K \setminus K')$ and $F$ are
isomorphic as $d$-partite $d$-uniform hypergraphs.
\end{proof}

\subsection{Cellular linear resolutions}
\label{cellular-linear-resolution-section}

  We give a quick review here of the theory of cellular resolutions \cite[Chapter 4]{MillerSturmfels}.
Then we use this to produce an extremely simple, linear\footnote{We are slightly abusing notation here.
Strictly speaking, what we get should be called a {\it $d$-linear resolution}:
all the minimal generators of the ideal have degree $d$, while all higher syzygy maps are
given by linear forms.}, minimal free resolution for
the  square-free strongly stable ideals $I(K)$ generated in fixed degree,
as well as their relatives $I(F(K)),I(M),I(F(M))$,
and for all ideals $I(F)$ with $F$ a Ferrers hypergraph.

\begin{definition}
Let $\cC$ be a {\it polyhedral cell complex}, that is, a finite collection $\cC=\{P_i\}$ of
convex polytopes $P_i$ (called {\it cells} or {\it faces} of $\cC$) in some Euclidean space,
with each face of $P_i$ also lying in $\cC$, and the intersection $P_i \cap P_j$ forming
a face of both $P_i$ and $P_j$.

Given a labelling of the vertices (= $0$-dimensional cells) of $\cC$ by monomials in a
polynomial ring $S=k[x_1,\ldots,x_N]$,
one obtains a labelling of each face $P$ by the least common multiple $m_{P}$ of the monomials
that label the vertices lying in $P$.  Letting $I$ be the monomial ideal generated by
all the monomial labels of all of the vertices, one obtains a $\Z^N$-graded complex of $S$-modules
$\cF(\cC)$ in which the $i^{th}$ term  $\cF_i(\cC)$ for $i \geq -1$ is the
free $S$-module with basis elements $e_P$ indexed by the $i$-dimensional faces $P$ of $\cC$,
decreed to have multidegree $m_P$.  The differential is defined $S$-linearly by
$$
d(e_P) := \sum_Q \sgn(P,Q) \frac{m_P}{m_Q} e_Q
$$
in which $Q$ runs through all the codimension $1$ faces of $P$, and $\sgn(P,Q) \in \{+1,-1\}$
denotes the {\it incidence function} produced from an orientation of the cells of $\cC$
used in the usual cellular chain complex that computes the homology of $\cC$.
\end{definition}

Note that $\cC$, if nonempty, always has exactly one face of dimension $-1$, namely the empty face $\varnothing$,
so that $\cF_{-1}(\cC) \cong S$ is a free $S$-module of rank $1$ with basis element $e_\varnothing$
of multidegree $0$.  Furthermore, note that the complex $\cF(\cC)$ has been arranged so that $S/I$ is the cokernel
of the map $\cF_{0}(\cC) \overset{d}{\rightarrow} \cF_{-1}(\cC)$.  In some cases,
$\cF(\cC)$ is a resolution of $S/I$ and lets us compute its Betti numbers -- the
basic proposition in the theory of cellular resolutions tells us that this is controlled
by the reduced homology with coefficients in $k$ of the subcomplexes defined
for each monomial multidegree $\alpha$ by
$$
\begin{aligned}
\cC_{\leq \alpha} &: = \{ P \in \cC: m_P \text{ divides } \alpha  \} \\
\cC_{< \alpha} &: = \{ P \in \cC: m_P \text{ divides } \alpha, \text{ but }m_P \neq \alpha \}.
\end{aligned}
$$

\begin{proposition}\cite[Proposition 4.5]{MillerSturmfels}
\label{miller-sturmfels-prop}
$\cF(\cC)$ is a resolution of $S/I$ if and only if,  for every multidegree $\alpha \in \Z^N$, the
subcomplex $\cC_{ \leq \alpha }$ is $k$-acyclic.
\end{proposition}

\subsection{The complex-of-boxes resolution}
\label{boxes-resolution-section}

We next describe the particular polyhedral complexes that will support our cellular resolutions.

\begin{definition}
\label{def-complex-of-boxes}
Given $F$, a $d$-partite $d$-uniform hypergraph on the partitioned
vertex set $X^{(1)} \sqcup \cdots \sqcup X^{(d)}$,
call a subset of $F$ which is a Cartesian product $X_1 \times \cdots \times X_d$
for some subsets $X_j \subseteq X^{(j)}$ a {\it box} inside $K$.
Define the {\it complex of boxes inside $F$}
to be the polyhedral subcomplex of the product of simplices $2^{X^{(1)}} \times \cdots \times 2^{X^{(d)}}$
having faces indexed by the boxes inside $K$.  Alternatively, the complex of boxes inside $F$ is
defined to be the {\it vertex-induced subcomplex} of the
Cartesian product of simplices $2^{X^{(1)}} \times \cdots \times 2^{X^{(d)}}$ on
the set of vertices indexed by the sets in $F$.  That is,
it consists of all polytopal cells in the Cartesian product whose vertices all lie in $F$.
\end{definition}

\begin{example}
\label{complex-of-boxes-example}
Let $K,M=\depol(K), F(K), F(M)$ be as in Example~\ref{hypergraph-ideals-example}.
Then the complex of boxes $\cC$ inside $F(K)$ or $F(M)$ are both isomorphic to a quadrangle and triangle
glued along an edge, with a pendant edge hanging from a nonadjacent vertex of the quadrangle.
The following diagrams illustrate these complexes of boxes, with
vertices labelled in boldface by the generators of the ideals $I(F(K)), I(K), I(F(M)), I(M)$,
and with higher-dimensional faces $P$ labelled in small script by the least common multiple
$m_P$. The complexes for $I(K)$ and $I(M)$ are obtained from the ones for $I(F(K))$ and $I(F(M))$, respectively, by specializing the labels as described in Theorem \ref{cellular-resolutions-theorem} below.

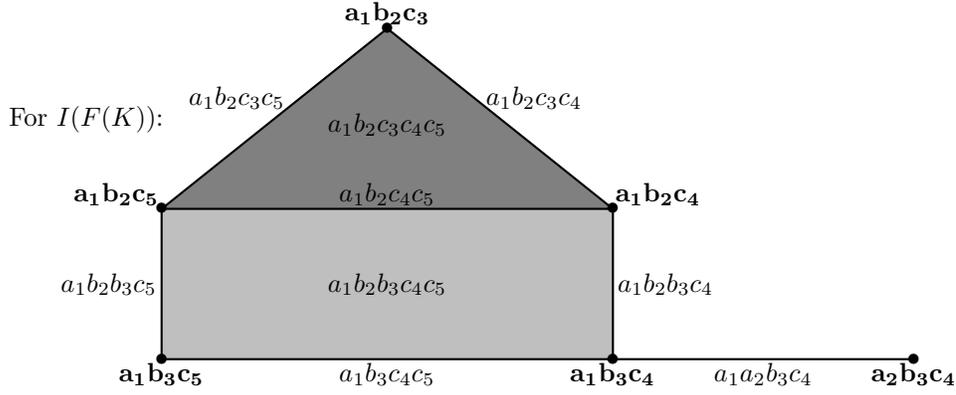
\begin{center}
\begin{pspicture}(6,-1)(4,5.0)
\psset{xunit=.5cm, yunit=.4cm}

\rput(-2,8){For $I(F(K))$:}

\pscustom[linestyle=none,fillstyle=solid,fillcolor=lightgray]{
\psline(0,0)(12,0) \psline(12,5)(0,5) }

\pscustom[linestyle=none,fillstyle=solid,fillcolor=gray]{
\psline(0,5)(12,5) \psline(6,11)(6,11) }

\rput(0,0){$\bullet$} \rput(12,0){$\bullet$} \rput(20,0){$\bullet$}
\rput(0,5){$\bullet$} \rput(12,5){$\bullet$}
\rput(6,11){$\bullet$}
\psline[linestyle=solid](0,0)(12,0)
\psline[linestyle=solid](0,5)(12,5)
\psline[linestyle=solid](12,0)(12,5)
\psline[linestyle=solid](0,0)(0,5)
\psline[linestyle=solid](0,5)(6,11)
\psline[linestyle=solid](12,5)(6,11)
\psline[linestyle=solid](12,0)(20,0)

\rput(0,-0.5){$\mathbf{a_1 b_3 c_5}$} \rput(12,-0.5){$\mathbf{a_1 b_3 c_4}$}
\rput(20,-0.5){$\mathbf{a_2 b_3 c_4}$}
\rput(-1.2,5.5){$\mathbf{a_1 b_2 c_5}$}
\rput(13.2,5.5){$\mathbf{a_1 b_2 c_4}$}
\rput(6,11.5){$\mathbf{a_1 b_2 c_3}$}
%
\rput(-1.4,2.5){${a_1 b_2 b_3 c_5}$}
\rput(13.4,2.5){${a_1 b_2 b_3 c_4}$}
\rput(6,-0.5){${a_1 b_3 c_4 c_5}$}
\rput(6,5.5){${a_1 b_2 c_4 c_5}$}
\rput(16,-0.5){${a_1 a_2 b_3 c_4 }$}
\rput(2.0,8.7){${a_1 b_2 c_3 c_5}$}
\rput(9.9,8.7){${a_1 b_2 c_3 c_4}$}
%
\rput(6,2.5){${a_1 b_2 b_3 c_4 c_5}$}
\rput(6,7.8){${a_1 b_2 c_3 c_4 c_5}$}

\end{pspicture}
\end{center}

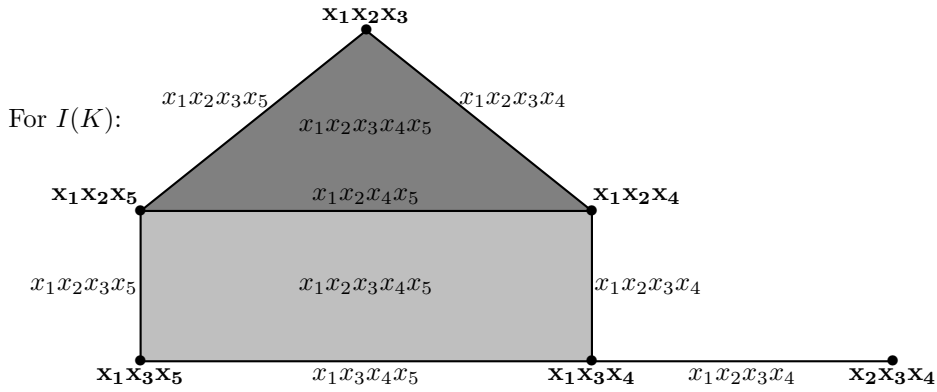
\begin{center}
\begin{pspicture}(6,-1)(4,4.5)
\psset{xunit=.5cm, yunit=.4cm}

\rput(-2,8){For $I(K)$:}

\pscustom[linestyle=none,fillstyle=solid,fillcolor=lightgray]{
\psline(0,0)(12,0) \psline(12,5)(0,5) }

\pscustom[linestyle=none,fillstyle=solid,fillcolor=gray]{
\psline(0,5)(12,5) \psline(6,11)(6,11) }

\rput(0,0){$\bullet$} \rput(12,0){$\bullet$} \rput(20,0){$\bullet$}
\rput(0,5){$\bullet$} \rput(12,5){$\bullet$}
\rput(6,11){$\bullet$}
\psline[linestyle=solid](0,0)(12,0)
\psline[linestyle=solid](0,5)(12,5)
\psline[linestyle=solid](12,0)(12,5)
\psline[linestyle=solid](0,0)(0,5)
\psline[linestyle=solid](0,5)(6,11)
\psline[linestyle=solid](12,5)(6,11)
\psline[linestyle=solid](12,0)(20,0)

\rput(0,-0.5){$\mathbf{x_1 x_3 x_5}$} \rput(12,-0.5){$\mathbf{x_1 x_3 x_4}$}
\rput(20,-0.5){$\mathbf{x_2 x_3 x_4}$}
\rput(-1.2,5.5){$\mathbf{x_1 x_2 x_5}$}
\rput(13.2,5.5){$\mathbf{x_1 x_2 x_4}$}
\rput(6,11.5){$\mathbf{x_1 x_2 x_3}$}
%
\rput(-1.5,2.5){${x_1 x_2 x_3 x_5}$}
\rput(13.5,2.5){${x_1 x_2 x_3 x_4}$}
\rput(6,-0.5){${x_1 x_3 x_4 x_5}$}
\rput(6,5.5){${x_1 x_2 x_4 x_5}$}
\rput(16,-0.5){${x_1 x_2 x_3 x_4 }$}
\rput(2.0,8.7){${x_1 x_2 x_3 x_5}$}
\rput(9.9,8.7){${x_1 x_2 x_3 x_4}$}
%
\rput(6,2.5){${x_1 x_2 x_3 x_4 x_5}$}
\rput(6,7.8){${x_1 x_2 x_3 x_4 x_5}$}

\end{pspicture}
\end{center}

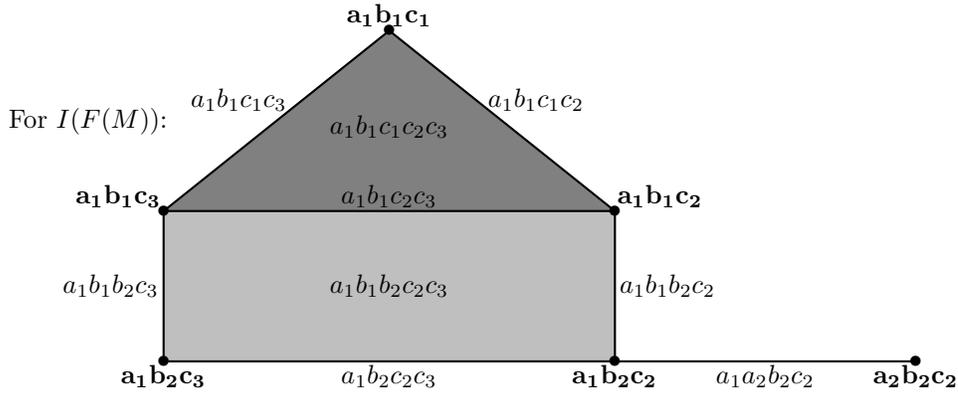
\begin{center}
\begin{pspicture}(6,-1)(4,4.5)
\psset{xunit=.5cm, yunit=.4cm}

\rput(-2,8){For $I(F(M))$:}

\pscustom[linestyle=none,fillstyle=solid,fillcolor=lightgray]{
\psline(0,0)(12,0) \psline(12,5)(0,5) }

\pscustom[linestyle=none,fillstyle=solid,fillcolor=gray]{
\psline(0,5)(12,5) \psline(6,11)(6,11) }

\rput(0,0){$\bullet$} \rput(12,0){$\bullet$} \rput(20,0){$\bullet$}
\rput(0,5){$\bullet$} \rput(12,5){$\bullet$}
\rput(6,11){$\bullet$}
\psline[linestyle=solid](0,0)(12,0)
\psline[linestyle=solid](0,5)(12,5)
\psline[linestyle=solid](12,0)(12,5)
\psline[linestyle=solid](0,0)(0,5)
\psline[linestyle=solid](0,5)(6,11)
\psline[linestyle=solid](12,5)(6,11)
\psline[linestyle=solid](12,0)(20,0)

\rput(0,-0.5){$\mathbf{a_1 b_2 c_3}$} \rput(12,-0.5){$\mathbf{a_1 b_2 c_2}$}
\rput(20,-0.5){$\mathbf{a_2 b_2 c_2}$}
\rput(-1.2,5.5){$\mathbf{a_1 b_1 c_3}$}
\rput(13.2,5.5){$\mathbf{a_1 b_1 c_2}$}
\rput(6,11.5){$\mathbf{a_1 b_1 c_1}$}
%
\rput(-1.4,2.5){${a_1 b_1 b_2 c_3}$}
\rput(13.4,2.5){${a_1 b_1 b_2 c_2}$}
\rput(6,-0.5){${a_1 b_2 c_2 c_3}$}
\rput(6,5.5){${a_1 b_1 c_2 c_3}$}
\rput(16,-0.5){${a_1 a_2 b_2 c_2 }$}
\rput(2.0,8.7){${a_1 b_1 c_1 c_3}$}
\rput(9.9,8.7){${a_1 b_1 c_1 c_2}$}
%
\rput(6,2.5){${a_1 b_1 b_2 c_2 c_3}$}
\rput(6,7.8){${a_1 b_1 c_1 c_2 c_3}$}

\end{pspicture}
\end{center}

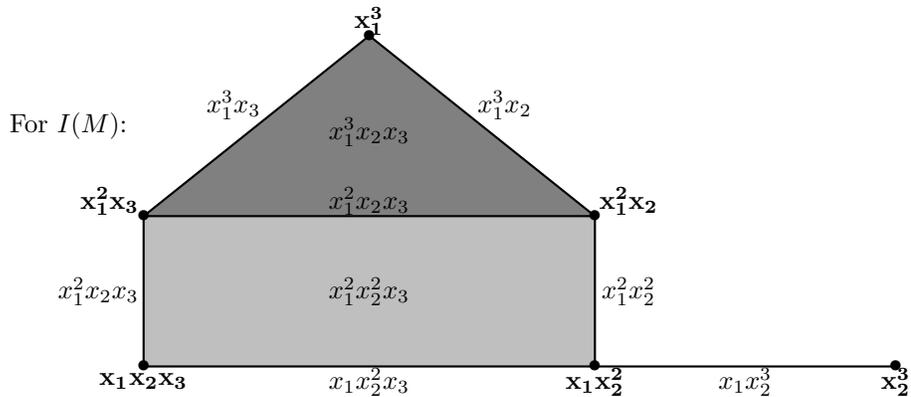
\begin{center}
\begin{pspicture}(6,-1)(4,4.5)
\psset{xunit=.5cm, yunit=.4cm}

\rput(-2,8){For $I(M)$:}

\pscustom[linestyle=none,fillstyle=solid,fillcolor=lightgray]{
\psline(0,0)(12,0) \psline(12,5)(0,5) }

\pscustom[linestyle=none,fillstyle=solid,fillcolor=gray]{
\psline(0,5)(12,5) \psline(6,11)(6,11) }

\rput(0,0){$\bullet$} \rput(12,0){$\bullet$} \rput(20,0){$\bullet$}
\rput(0,5){$\bullet$} \rput(12,5){$\bullet$}
\rput(6,11){$\bullet$}
\psline[linestyle=solid](0,0)(12,0)
\psline[linestyle=solid](0,5)(12,5)
\psline[linestyle=solid](12,0)(12,5)
\psline[linestyle=solid](0,0)(0,5)
\psline[linestyle=solid](0,5)(6,11)
\psline[linestyle=solid](12,5)(6,11)
\psline[linestyle=solid](12,0)(20,0)

\rput(0,-0.5){$\mathbf{x_1x_2x_3}$} \rput(12,-0.5){$\mathbf{x_1x_2^2}$}
\rput(20,-0.5){$\mathbf{x_2^3}$}
\rput(-0.9,5.5){$\mathbf{x_1^2x_3}$}
\rput(12.9,5.5){$\mathbf{x_1^2x_2}$}
\rput(6,11.5){$\mathbf{x_1^3}$}
%
\rput(-1.2,2.5){${x_1^2x_2x_3}$}
\rput(12.9,2.5){${x_1^2x_2^2}$}
\rput(6,-0.5){${x_1x_2^2x_3}$}
\rput(6,5.5){${x_1^2x_2x_3}$}
\rput(16,-0.5){${x_1 x_2^3}$}
\rput(2.4,8.7){${x_1^3 x_3}$}
\rput(9.6,8.7){${x_1^3 x_2}$}
%
\rput(6,2.5){${x_1^2 x_2^2 x_3}$}
\rput(6,7.8){${x_1^3 x_2 x_3}$}

\end{pspicture}
\end{center}

\end{example}

\begin{theorem}
\label{cellular-resolutions-theorem}
Let $K \subset \binom{\PP}{d}$ be squarefree strongly stable, and
let $M \subset \binom{\PP+d-1}{d}$ be strongly stable, with
$F(K), F(M)$ their associated $d$-partite $d$-uniform hypergraphs.
Let $F$ be any $d$-partite $d$-uniform Ferrers hypergraph.

\begin{enumerate}
\item[(i)]
For any of the ideals $I=I(F(K)), I(F(M)), I(F)$ inside $S:=k[\xb^{(1)},\ldots,\xb^{(d)}]$,
labelling a vertex $(i_1,\ldots,i_d)$ of the complex of boxes by the monomial $x^{(1)}_{i_1} \cdots x^{(d)}_{i_d}$
gives a minimal linear cellular $S$-resolution of $S/I$.

Hence $\beta_{\sum_j |X_j|-d,X_1 \sqcup \cdots \sqcup X_d}(I) = 1$
for every box $X_1 \times \cdots \times X_d$ inside $F$ or $F(K)$,
and all other Betti numbers vanish.

\item[(ii)]
Furthermore, the specialization map
$$
\begin{array}{ccc}
\spe: k[\xb^{(1)},\ldots,\xb^{(d)}] &\longrightarrow& k[\xb] \\
x^{(j)}_i & \longmapsto & x_i
\end{array}
$$
sends the resolution for $I(F(K))$ or $I(F(M))$ to a (minimal, linear, cellular) resolution for $I(K)$ or $I(M)$.
In other words, re-labelling a vertex $(i_1,\ldots,i_k)$ of the complex of boxes for $F(K)$ or $F(M)$
by $x_{i_1} \cdots x_{i_d}$ yields a $k[\xb]$-resolution of $I(K)$ or $I(M)$.

In particular, $I(F(K)), I(K)$ have the same $\Z$-graded Betti numbers,
and $I(F(M)), I(M)$ have the same $\Z$-graded Betti numbers.
%

\item[(iii)]
If $M=\depol(K)$, then
the bijection $\binom{\PP}{d} \overset{\depol}{\longrightarrow} \binom{\PP+d-1}{d}$
induces a cellular isomorphism of the complex of boxes for $I(F(K))$ and $I(F(M))$,
preserving the degree of the monomials $m_P$ labelling faces.

Consequently, $I=I(K),I(M),I(F(K)),I(F(M))$ all have the same
$\Z$-graded Betti numbers $\beta_{ij}(I)$ in this siutation.
\end{enumerate}
\end{theorem}

\begin{proof}
To simplify notation, assume $d=3$, and let $a_i, b_j, c_k$ be the three sets of variables;
it will be clear that the argument given works for general $d$.

We deal with the part of (i) asserting that the cellular complexes give
cellular resolutions last.  Assuming this, for the rest of assertion (i),
note that for each box
\begin{equation}
\label{typical-three-box}
P = \{a_{i_1},\ldots,a_{i_r}\} \times \{b_{j_1},\ldots,b_{j_s}\} \times \{c_{k_1},\ldots,c_{k_t}\}
\end{equation}
in one of the appropriate complexes of boxes inside $F, F(K), F(M)$, the least common multiple
monomial $m_P$ will have the appropriate degree for a $d$-linear cellular resolution,
namely $\deg m_P = \dim P + d.$  This is because one can easily check that for the labelling
with generators of $I(F), I(F(K)), I(F(M))$, one has
$$
m_P = (a_{i_1} \cdots a_{i_r})(b_{j_1} \cdots b_{j_s})(c_{k_1} \cdots c_{k_t}),
$$
and for the labelling with generators of $I(K), I(M)$, one has
$$
m_P = (x_{i_1} \cdots x_{i_r}) (x_{j_1} \cdots x_{j_s}) (x_{k_1} \cdots x_{k_t}).
$$
In any case,
$$
\begin{aligned}
\deg m_P &=r+s+t,\text{ while }\\
\dim P   &= (r-1)+(s-1)+(t-1),
\end{aligned}
$$
so $\deg m_P = \dim P + 3 (= \dim P + d)$.

The above descriptions of $m_P$ also show assertion (ii) of the theorem.
Assertion (iii) follows when $M=\depol(K)$ because the depolarization bijection on vertices
extends to a bijection sending the typical box $P$ inside $F(K)$ shown
in \eqref{typical-three-box} to the following box inside $F(M)$:
$$
\depol(P):=
   \{a_{i_1},\ldots,a_{i_r}\} \times \{b_{j_1-1},\ldots,b_{j_s-1}\} \times \{c_{k_1-2},\ldots,c_{k_t-2}\}.
$$

Lastly we deal with the first part of assertion (i), asserting that one has
various cellular resolutions.  By Proposition~\ref{miller-sturmfels-prop},
it suffices to show that for any of the ideals
$I(F), I(K), I(F(K)), I(M), I(F(M))$, if $\cC$ is the appropriate complex of boxes
labelled with the generators of this ideal, then
for any multidegree $\alpha$ in the appropriate polynomial ring,
the subcomplex $\cC_{\leq \alpha}$ is contractible.  In fact, we will do this
by induction on the number of vertices of $\cC_{\leq \alpha}$; in the base
case when $\cC_{\leq \alpha}$ has only one vertex, this is trivial.
In the inductive step, pick any vertex $v$ of $\cC_{\leq \alpha}$ whose corresponding
set or multiset is {\it Gale-maximal} among all the vertices of $\cC_{\leq \alpha}$.
We claim that
\begin{enumerate}
\item[(a)] there is a unique facet (maximal face) $P_{v,\alpha}$ of $\cC_{\leq \alpha}$
containing $v$, and
\item[(b)] if $v$ is not the {\it only} vertex of $\cC_{\leq \alpha}$, then
this facet $P_{v,\alpha}$ has strictly positive dimension.
\end{enumerate}
Assuming claims (a) and (b) for the moment, the argument is completed
as follows.  Lemma~\ref{unique-facet-topology-lemma} below implies
that $\cC_{\leq \alpha}$ is homotopy equivalent to the
subcomplex $\cC_{\leq \alpha} \setminus \{v\}$ obtained by
deleting all faces containing $v$.  Because $\cC$ and $\cC_{\leq \alpha}$ are
defined as vertex-induced subcomplexes, the deletion $\cC_{\leq \alpha} \setminus \{v\}$
is isomorphic to one of the subcomplexes $\cC'_{\leq \alpha'}$ which arises for an ideal
$\hat{I}$ in the same family as $I$, where one has removed the generator of $I$
corresponding to $v$.
Since $\cC'_{\leq \alpha'}$ has at least one fewer vertex than $\cC_{\leq \alpha}$, it
is contractible by induction.  Hence $\cC_{\leq \alpha}$ is also contractible.

\vskip.1in
\noindent
{\sf Proof of Claim (a):}
We exhibit explicitly the unique facet $P_{v,\alpha}$ of $\cC_{\leq \alpha}$
that contains a vertex $v$ corresponding to a Gale-maximal triple $(i_1,i_2,i_3)$,
for each kind of ideal $I=I(F),I(K),I(M),I(F(K)),I(F(M))$.  In each case it is not
hard to check that if $v$ were contained in another face $P \not\subseteq P_{v,\alpha}$ of $\cC_{\leq \alpha}$,
it would contradict the Gale-maximality of $(i_1,i_2,i_3)$.
As notation, when one has a totally ordered set such as $x_1,x_2,\ldots$, denote
the closed interval $\{x_i,x_{i+1},\ldots,x_{j-1},x_j\}$ by $[x_i,x_j]$.

\vskip.1in
\noindent
When $I=I(F)$ for a Ferrers hypergraph $F$, then $\alpha$ is a monomial in the variables
$a_i,b_j,c_k$, and one has
$$
P_{v,\alpha} = \left( [a_{1},a_{i_1}] \cap \supp \alpha \right)
                  \times  \left( [b_{1},b_{i_2}] \cap \supp \alpha  \right)
                     \times   \left( [c_{1},c_{i_3}] \cap \supp \alpha \right) .
$$
For example, if $v=(2,4,2)$ and $\alpha=a_1^5 a_2 a_3^4 b_1^2 b_2 b_4 b_5^9 c_2^3 c_4^2 c_5$,
then $P_{v,\alpha}=\{a_1,a_2\} \times \{b_1,b_2,b_4\} \times \{c_2\}.$

\vskip.1in
\noindent
When $I=I(F(K))$ for $K \subset \binom{\PP}{d}$ squarefree strongly stable,
then $\alpha$ is a monomial in the variables $a_i,b_j,c_k$, and one has
$$
P_{v,\alpha} = \left( [a_{1},a_{i_1}] \cap \supp \alpha \right)
                  \times  \left( [b_{i_1+1},b_{i_2}] \cap \supp \alpha  \right)
                     \times   \left( [c_{i_2+1},c_{i_3}] \cap \supp \alpha \right) .
$$
For example, if $v=(3,4,6)$ and $\alpha=a_1 a_3^3  b_1^2 b_3^4 b_4^5 b_7 c_2^7 c_4 c_5^3 c_6^2 c_7$,
then $P_{v,\alpha}=\{a_1,a_3\} \times \{b_4\} \times \{c_5,c_6\}.$

\vskip.1in
\noindent
When $I=I(F(M))$ for $M \subset \binom{\PP+d-1}{d}$ strongly stable,
then $\alpha$ is a monomial in the variables $a_i,b_j,c_k$, and one has
$$
P_{v,\alpha} = \left( [a_{1},a_{i_1}] \cap \supp \alpha \right)
                  \times  \left( [b_{i_1},b_{i_2}] \cap \supp \alpha  \right)
                     \times   \left( [c_{i_2},c_{i_3}] \cap \supp \alpha \right) .
$$
For example, if $v=(3,4,6)$ and $\alpha=a_1 a_3^3  b_1^2 b_3^4 b_4^5 b_7 c_2^7 c_4 c_5^3 c_6^2 c_7$,
then $P_{v,\alpha}=\{a_1,a_3\} \times \{b_3, b_4\} \times \{c_4,c_5,c_6\}.$

\vskip.1in
\noindent
When $I=I(K)$ for $K \subset \binom{\PP}{d}$ squarefree strongly stable,
then $\alpha$ is a monomial in the variables $x_i$, $P_{v,\alpha}$ is the specialization of the corresponding box $\widetilde{P}_{v,\alpha}$ in the complex resolving $I(F(K))$,
and one has
$$
\begin{aligned}
\widetilde{P}_{v,\alpha} &= \left( [a_{1},a_{i_1}] \cap \{a_j: x_j \in \supp \alpha\} \right) \\
            & \qquad \times  \left( [b_{i_1+1},b_{i_2}] \cap  \{b_j: x_j \in \supp \alpha\}  \right) \\
            & \qquad \qquad \times   \left( [c_{i_2+1},c_{i_3}] \cap  \{c_j: x_j \in \supp \alpha\} \right) .
\end{aligned}
$$
For example, if $v=(3,4,6)$ and $\alpha=x_1 x_3^4 x_4^5 x_5^3 x_6^2 x_7$,
then $\widetilde{P}_{v,\alpha}=\{a_1,a_3\} \times \{b_4\} \times \{c_5,c_6\}.$

\vskip.1in
\noindent
When $I=I(M)$ for $M \subset \binom{\PP+d-1}{d}$ strongly stable,
then $\alpha$ is a monomial in the variables $x_i$, but here one must be slightly
more careful because $I$ is not a squarefree monomial ideal.
This means that the {\it multiplicities} of the variables $x_i$ in $\alpha$ become relevant,
not just which variables $x_i$ occur in its support.  Define
$m^j_k(v)$ to be the multiplicity of the entry $j$ among the first $k$ coordinates
of $(i_1,i_2,i_3)$; this means that $m^j_0(v)=0$ for any value $j$.  Then define subsets
$$
\begin{aligned}
S_1(v,\alpha) &= \{a_j: x_j^{m^j_0(v)+1} \text{ divides }\alpha\}  \\
S_2(v,\alpha) &= \{b_j: x_j^{m^j_1(v)+1} \text{ divides }\alpha\} \\
S_3(v,\alpha) &= \{c_j: x_j^{m^j_2(v)+1} \text{ divides }\alpha\}
\end{aligned}
$$
One can then check that $P_{v,\alpha}$ is the specialization of the corresponding box $\widetilde{P}_{v,\alpha}$ in the complex resolving $I(F(M))$, where
$$
\widetilde{P}_{v,\alpha} = \left( [a_{1},a_{i_1}] \cap S_1(v,\alpha) \right)
                  \times  \left( [b_{i_1},b_{i_2}] \cap S_2(v,\alpha)  \right)
                     \times   \left( [c_{i_2},c_{i_3}] \cap S_3(v,\alpha) \right) .
$$
For example, if $v=(3,3,4)$ then $m_2^3(v)=2$.  This implies that
$$
\begin{aligned}
\text{ if }\alpha=x_1^2 x_3^3 x_4^7 x_5^2
 &\text{ then }   \widetilde{P}_{v,\alpha}=\{a_1,a_3\} \times \{b_3\} \times \{c_3,c_4\},\\
\text{ if }\alpha=x_1^2 x_3^2 x_4^7 x_5^2
 &\text{ then }   \widetilde{P}_{v,\alpha}=\{a_1,a_3\} \times \{b_3\} \times \{c_3\}.
\end{aligned}
$$

\vskip.1in
\noindent
{\sf Proof of Claim (b):}
If $v$ is not the only vertex of $\cC_{\leq \alpha}$, then because $v$ is Gale-maximal,
without loss of generality we may assume that there is another vertex $v'$ of $\cC_{\leq \alpha}$
which lies strictly below $v$ in the Gale ordering:  take any other vertex $w \neq v$
of $\cC_{\leq \alpha}$ and Proposition~\ref{gale-properties}(ii) implies that $v':=v \wedge w$ has
the desired property.

Now suppose for the sake of contradiction that the unique facet $P_{v,\alpha}$
of $\cC_{\leq \alpha}$ containing $v$ which was exhibited above is zero-dimensional.
This means that the box $P_{v,\alpha} = X_1 \times \cdots \times X_d$ has each
``side'' $X_m$ of the box of cardinality $|X_m|=1$.
Looking at the descriptions of $P_{v,\alpha}$ above
for each of the ideals $I(F),I(K),I(F(K)),I(M),I(F(M)$,
one can argue by induction on $m$ that
$$
\begin{aligned}
v&=(i_1,\ldots,i_d)\\
v'&=(i'_1,\ldots,i'_d)
\end{aligned}
$$
must be equal in their first $m$ coordinates
for $m=1,2,\ldots,d$, using the facts that $v' \leq_{Gale} v$, that $\xb^{v'}$ divides $\alpha$, and
that $|X_m|=1$.  Hence $v' = v$, a contradiction.
\end{proof}

We deduce from this the explicit graded Betti numbers of
$I(F), I(K), I(F(K)), I(M), I(F(M))$ in the above setting.  The answers
for $I(M), I(K)$ agree with the results of Eliahou and Kervaire \cite{EK} and
Aramova, Herzog, and Hibi \cite{AHH}.  The answers for $I(F)$ generalize
Corollary~\ref{Ferrers-betti-numbers}.

\begin{corollary}
\label{hypergraph-betti-corollary}
If $K$ is a squarefree strongly stable $d$-uniform hypergraph, and $M=\depol(K)$,
then all four ideals $I=I(K), I(F(K)), I(M), I(F(M))$
have $\beta_{ij}(I)=0$ unless $j=i+d$ and
$$
\begin{aligned}
\beta_i(I)=\beta_{i,i+d}(I)
 &= \sum_{S \in K} \binom{\max{S}-d}{i} \\
 &= \sum_{k \geq d} \mu_k(K) \binom{k-d}{i}
\end{aligned}
$$
where $\mu_k(K):=|\{S \in K: \max(S)=k\}|$.

If $F$ is a $d$-partite Ferrers $d$-uniform hypergraph then $\beta_{ij}(I(F))=0$ unless $j=i+d$ and
$$
\begin{aligned}
\beta_i(I(F))=\beta_{i,i+d}(I(F))
&= \sum_{(i_1,\ldots,i_d) \in F} \binom{\sum_j i_j-d}{i} \\
&= \sum_{k \geq d} \alpha_k(K) \binom{k-d}{i}
\end{aligned}
$$
where $\alpha_k(K):=|\{(i_1,\ldots,i_d) \in F: \sum_j i_j = k\}$.
\end{corollary}

\begin{proof}
Theorem~\ref{cellular-resolutions-theorem} tells us that all four
$I=I(K), I(F(K)), I(M), I(F(M))$ have
the same graded Betti numbers $\beta_{ij}(I)$,
which vanish unless $j=i+d$.  Furthermore, given a subset of positive integers $X$,
it tells us that the multigraded Betti number
$\beta_{\sum_j |X_j|-d,X}(I(K))$ is the number of
boxes $X_1 \times \cdots \times X_d$ inside $F(K)$ giving a decomposition
$X=X_1 \sqcup \cdots \sqcup X_d$.

Classify these boxes according to their set of maxima
$$
\begin{array}{ccccccccc}
S&:=\{ &\max X_1 & < & \cdots & < & \max X_d &\} &\\
 & =\{ & i_1     & < & \cdots & < & i_d      &\} &\in K.
\end{array}
$$
Given any set $S=\{i_1 < \cdots < i_d\} \in K$,
such a box and decomposition $X=X_1 \sqcup \cdots \sqcup X_d$ exists if and only
if
$$
S \subseteq X \subseteq [\max X]:=\{1,2,\ldots,\max X\},
$$
namely one has the unique decomposition  in which
$$
X_j:= \{i_{j-1}+1,i_{j-1}+2,\ldots,i_{j}-1,i_{j}\} \cap X
$$
with the convention that $i_0:=0$.
Thus for each set $S \in K$ there are
$\binom{|[\max S] \setminus S |}{i}=\binom{\max(S)-d}{i}$
sets $X$ with $S \subset X$, $|X| = i+d$, and $\max(X) = \max(S)$. For each such set $X$,
the finely graded Betti number $\beta_{i,X}(I)$  contributes
$1$ to $\beta_{i,i+d}(I) (=\beta_i(I))$.
This gives the first formula for $\beta_i(I(K))$;  the second follows immediately from the
first.

Similarly, the first formula for $\beta_i(I(F))$ when $F$ is a Ferrers hypergraph comes from
classifying the boxes $X_1 \times \cdots \times X_d$ inside $F$ according to their
maxima $(\max X_1,\ldots,\max X_d)= (i_1,\ldots,i_d) \in F$.  The second formula then
follows from the first.
\end{proof}

\begin{remark} \rm \ \\
It would be desirable to extend Theorem~\ref{cellular-resolutions-theorem}
to deal with the {\it stable ideals}  considered by Eliahou and Kervaire \cite{EK}
and {\it squarefree stable ideals} considered by Aramova, Herzog, and Hibi \cite{AHH}, which are somewhat
less restrictive than their {\it strongly stable} counterparts.

However, in both cases the issue of how one should construct the polarization
$I(F(K))$ from $I(K)$ becomes trickier.  The
following example shows that the construction used in
Theorem~\ref{cellular-resolutions-theorem} does not directly generalize -- new ideas are needed.

\begin{example}
 \label{ex-stable-ideal}
Consider the ideal $I = (x_1 x_2, x_1 x_3, x_2 x_3, x_2 x_4)$.
It is squarefree stable, but not squarefree strongly stable.
If blindly applied, the method of Theorem~\ref{cellular-resolutions-theorem}
would associate to $I$ a 1-dimensional cell complex (its complex of boxes).
However, this complex cannot support a cellular resolution for $I$ (minimal or otherwise),
since $I$ has projective dimension $2$.
\end{example}
\end{remark}

\section{PART III: Instances of Question~\ref{nonbipartite-conjecture} and Conjecture~\ref{bipartite-conjecture}}
\label{conjectures-section}

\subsection{Affirmative answers for Question~\ref{nonbipartite-conjecture}}
\label{nonbipartite-conjecture-section}

The next three propositions are offered
as evidence that many monomial ideals obey the colex lower bound.
Given a $d$-uniform hypergraph $K \subset \binom{\PP}{d}$, let
$C_K$ denote the unique colexsegment $d$-uniform hypergraph having
the same cardinality.

\begin{proposition}
\label{shifted-hypergraphs-satisfy-conjecture}
For any squarefree strongly stable $d$-uniform hypergraph $K \subset \binom{\PP}{d}$
or any strong stable collection $M \subset \binom{\PP+d-1}{d}$, all of the ideals
$I(K),I(F(K)),I(M),I(F(M))$ obey the colex lower bound.
\end{proposition}
\begin{proof}
By the depolarization bijection, one may assume that $M=\depol(K)$.
Then Corollary~\ref{hypergraph-betti-corollary} implies that all of
these ideals have the same Betti numbers $\beta_i(I)$, namely
$\beta_i(I)=\sum_{S \in K} \mu_k(K) \binom{k-d}{i}$.  Since $C_K$ also has $I(C_K)$
squarefree strongly stable, its Betti numbers obey a similar formula.
However, $\mu_k(K) \geq \mu_k(C_K)$ for all $k$ by definition of $\mu_k$ and due to the fact that
the colexicographic ordering on $\binom{\PP}{d}$ is an extension of the preordering by maxima.
\end{proof}

\begin{proposition}
\label{Ferrers-hypergraphs-satisfy-conjecture}
Any Ferrers $d$-partite $d$-uniform hypergraph $F$ obeys the colex lower bound.
\end{proposition}
\begin{proof}
By Corollary~\ref{hypergraph-betti-corollary}, it suffices to show that
$\alpha_k(F) \geq \mu_k(C_F)$ for all $k$.  Note that the map sending
vectors $(i_1,\ldots,i_d) \in \PP^d$ to their partial sums
$(i_1,i_1+i_2,i_1+i_2+i_3,\cdots,i_1+i_2+\cdots+i_d)$
is a bijection $\PP^d \rightarrow \binom{\PP}{d}$
with the property that it sends the distinct elements of $F$ which are counted by
$\alpha_k(F)$ to distinct subsets $S$ in $\binom{\PP}{d}$ having $\max(S)=k$.
Since $C_F$ is an initial segment in a linear ordering on $\binom{\PP}{d}$
that extends the partial ordering by $\max(S)$, this forces
$\alpha_k(F) \geq \mu_k(C_F)$.
\end{proof}

The proof of the following proposition uses an independent, later result (Corollary \ref{lower-bound-valid-corollary}) about Conjecture \ref{bipartite-conjecture}.

\begin{proposition}
\label{nonbipartite-lower-bound-valid}
For any shifted skew diagram $D$ and any ordered subsets $X, Y$,
both the bipartite graph $\Gbip{X}{Y}(D)$ and the nonbipartite
graph $\Gnonbip{X}(D)$ obey the colex lower bound.
\end{proposition}
\begin{proof}
There are several reductions. By Theorem~\ref{specialization-theorem}, one can replace
$\Gnonbip{X}(D)$ by $\Gbip{X}{X}(D)$, and hence it suffices
to prove it only for the bipartite graphs $\Gbip{X}{Y}(D)$.  But then
Corollary~\ref{lower-bound-valid-corollary} implies it suffices to prove it
only for row-nested bipartite graphs.   However row-nested bipartite graphs are
exactly the bipartite graphs isomorphic to Ferrers graphs so it suffices to
prove it for Ferrers graphs.  But these are Ferrers hypergraphs with $d=2$,
and hence the result follows from Proposition~\ref{Ferrers-hypergraphs-satisfy-conjecture}.
\end{proof}

\begin{remark} \rm \ \\
The proofs of Propositions~\ref{shifted-hypergraphs-satisfy-conjecture}
and ~\ref{Ferrers-hypergraphs-satisfy-conjecture}
reveal the important properties of the colexicographic ordering used to define the
colexsegment hypergraph $C_K$:
colex is a linear order with a minimum element, and all intervals finite, that extends
the Gale ordering, and which is weaker than the (total) preordering by maxima on $\binom{\PP}{d}$.
One could replace the colex ordering with any ordering on $\binom{\PP}{d}$
having these properties in defining $C_K$, and the proofs of the previous three
propositions would hold.
\end{remark}

\begin{remark} \rm \ \\
As with Conjecture~\ref{bipartite-graph-bound-conjecture} below,
there is an easy {\it upper} bound that comes from the Taylor resolution of $I(K)$
namely $\beta_i(I(K)) \leq \binom{|K|}{i+1}$.  For $d=2$ (the graph case)
equality is achieved in this upper bound if and only if the graph has
every connected component of $G$ a star, by
Proposition~\ref{prop-star-components} below.
\end{remark}

\begin{remark} \label{rem-counter-ex} \rm \ \\
It is not true that every monomial ideal generated in a single degree obeys the colex lower bound. For example, consider the edge ideal $I$ of a $5$-cycle. It is a Gorenstein ideal with total Betti numbers $(\beta_0, \beta_1, \beta_2) = (5, 5, 1)$. These are smaller than the total Betti numbers $(\beta_0, \beta_1, \beta_2) = (5, 6, 2)$ of the corresponding colexsegment-generated ideal $J$. In fact, it is not too difficult to show that $I$ has the smallest total Betti numbers among all homogeneous (not necessarily monomial) ideals that are minimally generated by 5 quadrics.
\end{remark}

\begin{remark} \rm \ \\
Question~\ref{nonbipartite-conjecture} invites comparison with
Conjecture 4.3 of Aramova, Herzog, and Hibi \cite{AHH}, in which the Betti numbers
of a squarefree monomial ideal are conjectured to be {\it bounded above} by
the unique {\it lexsegment ideal} having the same Hilbert function (rather then
bounded below by the unique colexsegment ideal having the same number of minimal
generators). Note that this conjecture is true if the ground field has characteristic zero by \cite[Theorem 2.9]{AHH-2000}.
\end{remark}

\subsection{Evidence for Conjecture~\ref{bipartite-conjecture} and its refinement}
\label{bipartite-conjecture-section}

Here we present  a more precise version of
Conjecture~\ref{bipartite-conjecture}, incorporating an upper bound
to go with the lower bound, and characterizing when equality occurs
for each.   In the sections following, we are able to prove
\begin{enumerate}
\item[$\bullet$]
the upper bound, which is not hard via the Taylor resolution (Section~\ref{Taylor-section}).
\item[$\bullet$]
the characterization for the case of equality both in the lower and
in the upper bounds (Sections~\ref{Taylor-section}, \ref{two-reductions-section}, \ref{lower-bound-section}).
\item[$\bullet$]
the whole conjecture is valid for graphs of the
form $\Gbip{X}{Y}(D)$ (Section~\ref{bipartite-conjecture-holds-for-skew}).
\end{enumerate}

We begin by defining the four classes of graphs that appear as the extreme cases
in the conjecture:
{\it row-nested} (the lower bound),
{\it nearly-row-nested} (the case of equality in the lower bound),
{\it horizontal} (the upper bound), and
{\it horizontal-vertical} the case of equality in the upper bound).

\begin{definition}
Given a bipartite graph $G$ on bipartite vertex set $X \sqcup Y$ with edge set $E(G)$, we will often refer to its associated {\it diagram}
$$
D:=\{ (x,y): \{x,y\} \in E(G) \} \subset X \times Y
$$
This motivates the following terminology. Define for each
vertex $x \in X$ its {\it row} $R_x$ of $G$ or $D$ as follows:
$$
R_x :=\{y \in Y: \{x,y\} \in E(G) \}.
$$
In other words, these are the neighboring vertices to $x$ in $G$.
Similarly define for vertices $y \in Y$ the {\it column} $C_y$ in $G$ or $D$.

Say that $G$ is {\it row-nested} if the collection of rows $\{R_x\}_{x \in X}$ is
totally ordered by inclusion, that is, if $|R_x| \leq |R_{x'}|$ then $R_x \subseteq R_{x'}$.
In particular, if $|R_x| = |R_{x'}|$ then $R_x = R_{x'}$.

Say $G$ is {\it nearly-row-nested} if
$|R_x| < |R_{x'}|$ implies $R_x \subset R_{x'}$ and for each cardinality $c \geq 0$, one
has
$$
\left| \bigcap_{x: |R_x|=c} R_x \right| \in \{c-1,c\}.
$$
In other words, rows of different cardinalities are nested, while all the rows of
a given cardinality $c$ are either all the same or have a common intersection of cocardinality $1$.

Say $G$ is {\it horizontal} if every square in its associated diagram $D$ is alone within
its column.

Say $G$ is {\it horizontal-vertical} if every square in its associated diagram $D$ is either
alone within its column or alone within its row, or both.

Lastly, define
$$
\beta_{i,X, \bullet}(I(G)):=\sum_{Y' \subseteq Y} \beta_{i,X \sqcup Y}(I(G)).
$$
In other words, these are the finely graded Betti numbers of $I(G)$ with respect to
the specialized multigrading in which all the $Y$ variables have degree $0$.
\end{definition}

It is not hard to see that if $G$ is any bipartite graph on vertices $X \sqcup Y$,
there is up to isomorphism, a unique row-nested bipartite graph $R_G$ on $X \sqcup Y'$
for some $Y'$ with the same row sizes $R_x=(R_G)_x (= \deg_G(x))$ for all $x \in X$.
Similarly, there is up to isomorphism a unique horizontal graph $H_G$ with the same row sizes
as $G$.

Here is the more precise version of Conjecture~\ref{bipartite-conjecture}.

\begin{conjecture}
\label{bipartite-graph-bound-conjecture}
For any bipartite graph $G$ on vertex set $X \sqcup Y$, let $R_G$ be the unique (up to isomorphism)
row-nested graph with the same row sizes/$X$-degrees, and $H_G$ the unique (up to isomorphism) horizontal graph
with with the same row sizes/$X$-degrees.

Then for all $i$ and all $X' \subset X$ one has
\begin{equation}
\label{conjectured-bipartite-inequalities}
\begin{matrix}
\beta_{i,X',\bullet}(I(R_G)) & \leq &
\beta_{i,X',\bullet}(I(G))   & \leq &
\beta_{i,X',\bullet}(I(H_G)) \\
\Vert & & & &\Vert \\
\left\{
\begin{matrix}
\binom{ \mindeg(X') }{i-|X'|+2} &\text{ if } |X'|<i+2 \\
0                               &\text{ otherwise}.\\
\end{matrix}
\right\}
& & & & \binom{ |\deg(X')| }{i+1}
\end{matrix}
\end{equation}
where
$$
\begin{aligned}
\mindeg(X')&:=\min\{\deg_G(x): x \in X'\}, \text{ and }\\
|\deg(X')|&:=\sum_{x \in X'} \deg_G(x) .
\end{aligned}
$$

Furthermore, equality occurs for all $i$ and all $X'$ in the lower (resp. upper) bound, that is, in
the first (resp. second) inequality of \eqref{conjectured-bipartite-inequalities},
if and only if $G$ is nearly-row-nested (resp. horizontal-vertical).
\end{conjecture}

The binomial coefficient expressions for the Betti numbers $\beta_{i,X',\bullet}(I(R_G)),
\beta_{i,X',\bullet}(I(H_G))$ that appear in \eqref{conjectured-bipartite-inequalities}
as lower and upper bounds are easily explained.
For the upper bound, it will be shown in Proposition~\ref{prop-star-components}
below that a bipartite graph $H$ is horizontal-vertical
if and only if the {\it Taylor resolution} for $I(H)$ is minimal, and from this the
given formula for $\beta_{i,X',\bullet}(I(H_G))$ follows immediately.  For the
lower bound, it is easy to see that a graph $R$ is row-nested if and only if it is isomorphic
to the {\it Ferrers bipartite graphs} considered in Section~\ref{Ferrers-example},
and then Proposition~\ref{Ferrers-betti-numbers} gives the formula about the Betti numbers of $R_G$.

\begin{remark} \rm \ \\
The lower bound in Conjecture~\ref{bipartite-graph-bound-conjecture}
can be regarded as an analogue of the {\it Gale-Ryser Theorem} from graph theory:

\begin{theorem}(Gale-Ryser)
A pair of weakly decreasing nonnegative integer sequences $(d^X, d^Y)$
having the same sum are the $X$-degrees and $Y$-degrees of some bipartite graph $G$ on vertex set $X \sqcup Y$
if and only if the conjugate partition $(d^Y)^T$ majorizes $d_X$, that is,
$$
d^X_1 + \cdots + d^X_\ell \leq (d^Y)^T_1 + \cdots + (d^Y)^T_\ell
$$
for all $\ell$.  The equality $(d^Y)^T=d_X$ holds if and only if the associated graph $G$ is row-nested,
that is, a Ferrers graph.
\end{theorem}
\end{remark}

Before proving various parts of this conjecture, we pause to give some useful characterization
of the various classes of bipartite graphs $G$ just defined, in terms of avoidance of certain
{\it vertex-induced subgraphs} $G_{X'\sqcup Y'}$ of $G$, up to isomorphism.
We equivalently phrase them also in terms of the diagram $D$ for $G$ avoid certain subdiagrams
$\Dbip{X}{Y}$, up to relabelling the elements of $X$ and of $Y$.

\begin{proposition}
\label{forbidden-subgraph-characterizations}
Let $G$ be a bipartite graph on vertex set $X \sqcup Y$, with associated diagram $D \subseteq X \times Y$.
\begin{enumerate}
\item[(i)] $G$ is row-nested if and only if $G$ avoids $G_{X' \sqcup Y'}$ isomorphic to two
disjoint edges.  Equivalently, $D$ avoids subdiagrams $\Dbip{X'}{Y'}$ of the form
$$
\begin{matrix}
       &\times\\
\times &
\end{matrix}
$$
\item[(ii)] $G$ is nearly row-nested if and only if $G$ avoids $G_{X' \sqcup Y'}$ isomorphic to
a $6$-cycle or isomorphic to the disjoint union of an edge with a path having two edges and
both endpoints in $Y$. Equivalently, $D$ avoids subdiagrams $\Dbip{X'}{Y'}$ of the form
$$
\begin{matrix}
\times & \times & \\
\times &        & \times \\
       & \times & \times
\end{matrix}
\qquad \text{ or } \qquad
\begin{matrix}
\times &        & \\
       & \times & \times
\end{matrix}
$$

\item[(iii)] $G$ is horizontal if and only if $G$ avoids $G_{X' \sqcup Y'}$ isomorphic to a path with two edges
and both endpoints in $X$.  Equivalently, $D$ avoids subdiagrams $\Dbip{X'}{Y'}$ of the form
$$
\begin{matrix}
\times\\
\times
\end{matrix}
$$
\item[(iv)] $G$ is horizontal-vertical if and only if $G$ avoids $G_{X' \sqcup Y'}$ isomorphic to a path with three edges or a $4$-cycle. Equivalently, $D$ avoids subdiagrams $\Dbip{X'}{Y'}$ of the form
$$
\begin{matrix}
\times&\times\\
\times&
\end{matrix}
\qquad \text{ or } \qquad
\begin{matrix}
\times&\times\\
\times&\times
\end{matrix}
$$
\end{enumerate}
\end{proposition}

\begin{proof}
For each of the four assertions, the forward implication is easy.  It is the
backward implications that require proof, which we give here.

\noindent
{\bf (iii):} Obvious.

\noindent
{\bf (i):}  Assume $G$ is not row-nested.  Then there exist two rows $R_{x_1}, R_{x_2}$
which are not nested, that is, there exists $y_1 \in R_{x_1} \setminus R_{x_2}$
and $y_2 \in  R_{x_2} \setminus R_{x_1}$.  But then $G_{\{x_1,x_2\},\{y_1,y_2\}}$ is the
disjoint union of the two edges $\{x_1,y_1\}, \{x_2,y_2\}$.

\noindent
{\bf (iv):} Assume $G$ is not horizontal-vertical.  Then there exists a cell $(x_1,y_1)$ in
its diagram that is neither alone in its row nor in its column.  Hence there exist
cells of the form $(x_1,y_2), (x_2,y_2)$ in the diagram, and $G_{\{x_1,x_2\},\{y_1,y_2\}}$ will be
a path with two edges or a $4$-cycle, depending upon whether the cell $(x_2,y_2)$ is absent
or present in the diagram.

\noindent
{\bf (ii):}  Assume $G$ is not nearly row-nested.

\noindent
{\sf Case 1.} There exist two non-nested rows $R_{x_1},R_{x_2}$ of unequal sizes, say
$|R_{x_1}| > |R_{x_2}|$.

Then there exist $y_1, y_2 \in R_{x_1} \setminus R_{x_2}$
and $y_3 \in  R_{x_2} \setminus R_{x_1}$.  Hence $G_{\{x_1,x_2\},\{y_1,y_2,y_3\}}$ is the
disjoint union of the edge $\{x_2,y_3\}$ and the path $\{x_1,y_1\},\{x_1,y_2\}$ having
two endpoints in $Y$.

\vskip .1in
\noindent
{\sf Case 2.} There do not exist two non-nested rows $R_{x_1},R_{x_2}$ of unequal size.

Then there must exist a cardinality $c$ for which
$\left| \bigcap_{x: |R_x|=c} R_x \right| \leq c-2$.

\vskip .1in
\noindent
{\sf Subcase 2a.} There is a pair of rows $R_{x_1}, R_{x_2}$ both of cardinality $c$
with $|R_{x_1} \cap R_{x_2}| \leq c-2$.

Then there exist $y_1,y_2,y_3$ having the
same properties as in Case 1 above.

\vskip .1in
\noindent
{\sf Subcase 2b.} Every pair of unequal rows $R_{x_1} \neq R_{x_2}$ both of cardinality $c$
has $|R_{x_1} \cap R_{x_2}| =c-1$.

Start with two unequal rows $R_{x_1} \neq R_{x_2}$ both of cardinality $c$.  Since they
are unequal and of the same cardinality, one can find
$y_1 \in R_2 \setminus R_1, y_2 \in R_1 \setminus R_2$.  Now pick a third row $R_{x_3}$ of
the same cardinality with the property that there exists $y_3 \in R_{x_1} \cap R_{x_2}$ but
$y_3 \not\in R_{x_3}$;  this must exist since $\left| \bigcap_{x: |R_x|=c} R_x \right| \leq c-2$.
We claim that this forces $x_1 \in R_{x_3}$, since $|R_{x_1} \cap R_{x_3}| =c-1$,
and it similarly forces $x_2 \in R_{x_3}$.  But this means $G_{\{x_1,x_2,x_3\},\{y_1,y_2,y_3\}}$
is a $6$-cycle.
\end{proof}

\subsection{Proof of the upper bound and the case of equality}
\label{Taylor-section}

The upper bound in Conjecture~\ref{bipartite-graph-bound-conjecture} follows from the {\it Taylor resolution}
for a monomial ideal $I$, which we recall here; see \cite{Taylor}.

\begin{definition}
Let $I$ be a monomial ideal in the polynomial ring $S$,
and choose an ordered generating set of monomials $(m_1,\ldots,m_p)$, for $I$.
Then the {\it Taylor resolution} $\cT(S/I)$ of $S/I$ as an $S$-module
with respect to this ordered generating set has resolvents
$\cT_k(S/I) \cong S^{\binom{p}{k}}$ for $k=0,1,\ldots,p$, in
which $\cT_k(S/I) \cong S^{\binom{p}{k}}$ is the free $S$-module
on a basis $\{e_A: A \subset [p], |A|=k\}$.  One decrees the
finely graded multidegree of $e_A$ to be
$$
m_A:=\lcm\{ m_a: a \in A\}
$$
and then the differential in $\cT(S/I)$ is defined $S$-linearly by
$$
d(e_A):= \sum_{a \in A} \sgn(A,a)
   \frac{m_A}{m_{A \setminus \{a\}}} e_{A\setminus \{a\}}.
$$
Here $\sgn(A,a)=(-1)^r$ if $a$ is the $r^{th}$ smallest element of $A$.

Alternatively, one can view $\cT(S/I)$ as a {\it cellular resolution} $\cF(\cC)$
(as described in Section~\ref{cellular-linear-resolution-section}) for the cell complex which is a $(p-1)$-simplex in
which the vertices are labelled with $m_1,\ldots,m_p$.
\end{definition}

Since $\cT(S/I)$ is a (not necessarily minimal) free $S$-resolution for $S/I$, the number of basis
elements in $\cT_i(S/I)$ of multidegree $m$ always gives an upper bound on the Betti number
$\beta_{i,m}(S/I) = 0$, and this bound is tight for all $m$ if and only if $\cT(S/I)$ is a {\it minimal}
free resolution.  However for a graph $G$ one can easily characterize when
the Taylor resolution for $I(G)$ is minimal.  Say that a graph is a {\it star} if
it is a tree with at most one vertex of degree larger than $1$.

\begin{proposition}
\label{prop-star-components}
A graph $G$ has the Taylor resolution $\cT(S/I(G))$ minimal if and only if
every connected component of $G$ is a star.  Hence a bipartite graph
$G$ has the Taylor resolution $\cT(S/I(G))$ minimal if and only if
$G$ is horizontal-vertical.
\end{proposition}
\begin{proof}
If every connected component of $G$ is a star then the Taylor
resolution is minimal, as all the least common multiples $m_A$ are distinct for different
subsets $A$:  every generating monomial $m=x_i x_j$ contains a variable $x_i$ or $x_j$
corresponding to a vertex of degree one, which is therefore contained
in no other generating monomial.

Conversely, suppose a graph $G$ has its Taylor resolution minimal.
Then for any subset $X$ of its vertices, the vertex-induced subgraph
$G_X$ must also have its Taylor resolution $\cT(S/I(G_X))$ minimal, as it is a subcomplex
of the Taylor resolution for $I(G)$.  This means that $G$ must avoid
as a vertex-induced subgraph $G_X$ having
\begin{enumerate}
\item[$\bullet$] a $3$-cycle,
\item[$\bullet$] a $4$-cycle, or
\item[$\bullet$] path with $3$ edges,
\end{enumerate}
since one can do a small calculation of the Taylor resolution for
each, and find that none of them are minimal.  We claim that this forces
$G$ to have no cycles -- if not, it would contain some cycle of
minimum length, which would either be of length $3$,
or of length $4$, or of length at least $5$ and hence contain a vertex-induced
path with $3$ edges.
Hence $G$ must be a forest, and its component trees must all have
diameter $2$, in order to avoid the path with $3$ edges.  Thus each component
is a star.
\end{proof}

\begin{corollary}
\label{cor-upper-b}
The upper bound in Conjecture~\ref{bipartite-graph-bound-conjecture}
is valid, as is the assertion there about the case of equality in the upper bound.
\end{corollary}
\begin{proof}
Let $S=k[\xb,\yb]$, and compare the
Taylor resolutions $\cT(S/I(G))$ and $\cT(S/I(H(G)))$.  Note
that when one only looks at the $X$-multidegrees, setting
the $Y$-multidegrees to $0$, the two resolutions have exactly
the same number of basis elements in each $X$-multidegree.  The former is a resolution for $S/I(G)$, and
hence provides an upper bound on its Betti numbers, while the latter is a minimal free
resolution for $S/I(H_G)$ by Proposition~\ref{prop-star-components}.  This proves the
asserted upper bound.

In the case of equality, it must be that a minimal free resolution of $S/I_G$ has
as many terms as its Taylor resolution (namely $2^{|E(G)|}$ terms), and hence the Taylor resolution $\cT(S/I(G))$ is minimal.  Thus by Proposition~\ref{prop-star-components}, $G$ is horizontal-vertical.
\end{proof}

\subsection{Two general reductions in the lower bound}
\label{two-reductions-section}

  Here we give two reductions that may apply to a bipartite graph
when one is attempting to verify the lower bound in Conjecture~\ref{bipartite-graph-bound-conjecture}.
Both will be used in the next section to verify the case of equality conjectured for
the lower bound.

Say that a bipartite graph $G$ on vertex set $X \sqcup Y$, or its diagram $D$,
has the vertex $x \in X$ (resp. $y \in Y$) as a {\it full row (resp. column)}
if $E(G)$ contains all of $\{x\} \times Y$ (resp. $X \times \{y\}$).
Say that $x, x' \in X$ index {\it nested} rows if $R_{x'} \subseteq R_{x}$

The following two results allow one to remove full columns and/or rows,
and remove nested rows, when considering a minimal counterexample
to Conjecture~\ref{bipartite-graph-bound-conjecture}.

\begin{proposition}
\label{full-column-reduction}
Let $G$ be a bipartite graph on vertex set $X \sqcup Y$.
If $G$ has $y \in Y$ as a full column, one has for all $i$ and all $X' \subseteq X$ that
$$
\beta_{i,X',\bullet}(I(G))= \delta_{i,|X'|-1} +
\beta_{i,X',\bullet}(I(G\setminus \{y\}) + \beta_{i-1 ,X',\bullet}(I(G\setminus \{y\}),
$$
where $G\setminus \{y\}$ denotes the vertex-induced subgraph of $G$ on $X \sqcup (Y\setminus \{y\})$.

Consequently, in this situation, $G$ achieves equality in the
lower bound of Conjecture~\ref{bipartite-graph-bound-conjecture} if and only if
$G \setminus \{y\}$ does.
\end{proposition}

\begin{proof}
The idea is to compare the most finely graded Betti numbers $\beta_{i,X' \sqcup Y'}$
for $I(G)$ versus $I(G \setminus \{y\})$.

If $y \not\in Y'$, clearly $G_{X' \sqcup Y'}=(G\setminus\{y\})_{X' \sqcup Y'}$,
so that $\beta_{i,X' \sqcup Y'}(I(G)) = \beta_{i,X' \sqcup Y'}(I(G \setminus \{y\})).$

Furthermore, if $y \not\in Y'$ and $Y'\neq \varnothing$ then
$\Delta(G_{X' \sqcup (Y' \sqcup \{y\})})$ is obtained from
$\Delta(G_{X' \sqcup Y'})$ simply by adding in the vertex $y$ as the apex of a
cone over the base simplex having vertex set $Y'$.  Hence the two complexes
are homotopy equivalent, and
$\beta_{i,X' \sqcup (Y' \sqcup \{y\})}(I(G)) =
\beta_{i-1,X' \sqcup Y'}(I(G \setminus \{y\})).$

Lastly, note that $\beta_{i,X',\{y\}}(I(G))= 1$ for $i=|X'|-1$ and $0$ for all other $i$.

Since  $\beta_{i,X',\bullet}(I(G))=\sum_{Y' \subset Y} \beta_{i,X' \sqcup Y'}(I(G))$,
the formula in the proposition follows.

The second assertion in (ii) is a consequence of the formula, as $R_G$ will
have a full column whenever $G$ does.
\end{proof}

\begin{proposition}
\label{nested-row-reduction}
Let $G$ be a bipartite graph on vertex set $X \sqcup Y$ with two nested rows
$R_{x_2} \subset R_{x_1}$.  Then
$$
\beta_{i,X,\bullet}(I(G))= \beta_{i-1,X\setminus \{x_1\},\bullet}(I(G\setminus\{x_1\}))
$$
for all $i$.

Consequently, if Conjecture~\ref{bipartite-graph-bound-conjecture}
holds for all bipartite graphs with smaller $|X|$, it will also hold
for $G$.
\end{proposition}

\begin{proof}
For the first assertion, introduce the ideal $I(G)+(x_1)$.  Although  this ideal
is no longer quadratic, it is still generated by squarefree monomials, and hence is the Stanley-Reisner
ideal for a simplicial complex on this vertex set $X \sqcup Y$.
Specifically, this complex looks exactly like $\Delta(I(G\setminus\{x_1\}))$, except that it
has $x_1$ (in principle) allowed in its vertex set, although the singleton $\{x_1\}$ does not form a
simplex (!).  Hochster's formula (Proposition~\ref{Hochster-formula}) shows that
$$
\beta_{i-1,X\setminus \{x_1\},\bullet}(I(G\setminus\{x_1\}))=\beta_{i,X,\bullet}(I(G)+(x_1)).
$$
Hence it only remains to show that
\begin{equation}
\label{desired-betti-relation}
\beta_{i,X,\bullet}(I(G)) = \beta_{i,X,\bullet}(I(G)+(x_1)).
\end{equation}
Letting $S:=k[\xb,\yb]$, one can relate the Betti numbers of $I(G)+(x_1)$ and of $I(G)$ via the
long exact sequence in $\Tor^S(-,k)$ associated to the short exact sequence of $S$-modules
$$
0 \rightarrow S/(I(G):x_1) (- \deg x_1) \overset{x_1}{\longrightarrow} S/I(G) \rightarrow S/(I(G)+(x_1)) \rightarrow 0
$$
where
$$
(I(G):x_1):=\{ f \in k[\xb,\yb]: f x_1 \in I(G) \}
$$
denotes the usual {\it ideal quotient} (or {\it colon ideal}).  The ideal
$(I(G):x_1)$ is also not quadratic, but is still generated by squarefree monomials, with
unique minimal monomial generating set given by
\begin{equation}
\label{colon-generators}
\{y_j: \{x_1,y_j\} \in E(G)\} \quad \sqcup \quad
\{x_i y_j:  \{x_i,y_j\} \in E(G) \text{ with } x_i \neq x_1 \text{ and }\{x_1,y_j\} \not\in E(G) \}.
\end{equation}
Note that $x_2 \in X$ does not appear among any of these minimal generators for $(I(G):x_1)$,
because every edge of the form $\{x_2,y_j\} \in E(G)$ also has
$\{x_1,y_j\} \in E(G)$.  Therefore the finely graded component corresponding to $X$ does
not appear in the minimal free resolution for $(I(G):x_1)$, and hence
$$
\beta_{i,X,\bullet}(I(G):x_1)=\beta_{i-1,X,\bullet}(S/I(G):x_1)=0 \text{ for all }i.
$$
Therefore the desired relation \eqref{desired-betti-relation}
follows from the aforementioned long exact sequence.

For the second assertion, note that the hypothesis on $G$ implies that the row-nested graph $R_G$ will
also have rows $(R_G)_{x_2} \subseteq (R_G)_{x_1}$ nested in the same way, and hence the
formula in the proposition holds for $R_G$ also.  The rest is straightforward.
\end{proof}

\subsection{The case of equality in the lower bound}
\label{lower-bound-section}

Although we do not know how to prove the lower bound in
Conjecture~\ref{bipartite-graph-bound-conjecture} in general, we
prove here its assertion about when equality is achieved.

\begin{theorem}
Let $G$ be a bipartite graph on vertex set $X \sqcup Y$, and $R_G$ the unique
(up to isomorphism) row-nested graph having the same row sizes/$X$-degrees.
Then for all $i$ and all $X' \subseteq X$ one has
$$
\beta_{i,X',\bullet}(I(G)) = \beta_{i,X',\bullet}(I(R_G))
\left(
=\left\{
\begin{matrix}
\binom{ \mindeg(X') }{i-|X'|+2} &\text{ if } |X'|<i+2 \\
0                               &\text{ otherwise}.\\
\end{matrix}
\right\}
\right)
$$
if and only if $G$ is nearly row-nested.
\end{theorem}
\begin{proof}
\vskip.1in
\noindent
{\sf The forward implication.}

Note that if $G$ is {\it not} nearly row-nested,
then Proposition~\ref{forbidden-subgraph-characterizations}(ii) shows that
$G$ contains some vertex-induced subgraph $G_{X',Y'}$ isomorphic to either the disjoint of an
edge with a path of two edges having both endpoints in $Y$, or a $6$-cycle.

In the first case, consider the subdiagrams for $R_G$ and $G$ restricted
to the two rows $X'=\{x_1,x_2\}$, which will have sizes $|R_{x_1}|=a, |R_{x_2}|=b$
and, say, $a \geq b$.
In $R_G$, these two rows are nested, while in $G$ they overlap in $c$ columns where
$a-c \geq 2$.  One has $\beta_{1,X',\bullet}(I(R_G))=\binom{b}{1-2+2}=b$
by the formula in the statement, while Corollary~\ref{betti-number-corollary} shows
that $\beta_{1,X',\bullet}(I(G))\geq c+(a-c)(b-c)$:  one will have $\beta_{1,X' \sqcup Y'} (I(G))=1$
for any of the $c$ choices of $Y'$ equal to a single column in the overlap of the
two rows, or for any of the $(a-c)(b-c)$ choices of $Y'$ having two columns of size one,
one each from the nonoverlapping columns.  But then the fact that
$a-c \geq 2$ implies
$$
\beta_{1,X',\bullet}(I(G)) \,\, \geq
c+(a-c)(b-c) \,\, >
\,\,  b =
\,\, \beta_{1,X',\bullet}(I(R_G)).
$$

In the second case, we may assume that the former subgraphs $G_{X',Y'}$ do {\it not} exist
in $G$, but a $6$-cycle $G_{\{x_1,x_2,x_3\},\{y_1,y_2,y_3\}}$ does exist.  The proof of
Proposition~\ref{forbidden-subgraph-characterizations}(ii) showed that in this situation,
the subgraph $G_{\{x_1,x_2,x_3\},Y}$ will have every $y$ of $Y \setminus \{y_1,y_2,y_3\}$
giving a full column.  Hence by Proposition~\ref{full-column-reduction}, it suffices to
remove these full columns and show that a $6$-cycle $G$ does not achieve the
lower bound in the conjecture.  For this, observe that a $6$-cycle $G$ has
$$
\beta_{3,X',\bullet}(I(R_G)) =
\binom{2}{3-3+2} =
1 < 2 =
\dim_k \tilde{H}_1(\Delta(I(G))) =
\beta_{3,X',Y'}(I(G)) \leq \beta_{3,X',\bullet}(I(G)).
$$
Here the calculation $\tilde{H}_1(\Delta(I(G))) \cong k^2$ comes from direct inspection:
if the $6$-cycle $G_{\{x_1,x_2,x_3\},\{y_1,y_2,y_3\}}$ has edges $E(G)=\{x_i,y_j\}_{1 \leq i \neq j \leq 3}$,
then $\Delta(I(G))$ consists of two triangles $\{x_1,x_2,x_3\},\{y_1,y_2,y_3\}$ together with
the three edges $\{x_i,y_i\}_{i=1,2,3}$ connecting them, and is homotopy equivalent to a wedge
of two circles.

\vskip.1in
\noindent
{\sf The reverse implication.}

Assume that $G$ is a nearly row-nested bipartite graph
on $X \sqcup Y$, and one must show that
$\beta_{i,X',\bullet}(I(G)) = \beta_{i,X',\bullet}(I(R_G))$ for all $i$ and
all $X' \subset X$.  Proceed by induction on $|X|+|Y|$.
By the reductions in Propositions~\ref{nested-row-reduction} and \ref{full-column-reduction}
one may assume that $G$ contains no nested pair of rows, and that it contains no full columns.
Because $G$ is nearly row-nested, containing no nested pair of rows implies all the
rows $R_{x_i}$  have the same cardinality $c$.  Containing no full columns then forces $c=1$.
But in this case, both $G$ and $R_G$ are horizontal-vertical, having their Taylor resolutions minimal by
Proposition~\ref{prop-star-components}, and
$$
\beta_{i,X',\bullet}(I(G)) = \beta_{i,X',\bullet}(I(R_G)) =
\begin{cases}
1 & \text{ if }i=|X'|-1 \\
0 & \text{ otherwise.}
\end{cases}
$$
\end{proof}

Note that the above result gives a characterization of nearly row-nested graphs by means of its Betti numbers $\beta_{i,X',\bullet}$. In this sense it is analogous to the characterization of Ferrers graphs as the bipartite graphs whose minimal free resolution is linear (see \cite[Theorem 4.2]{CN1}).

\subsection{Verifying the bipartite conjecture for $\Dbip{X}{Y}$}
\label{bipartite-conjecture-holds-for-skew}

Having already verified the upper bound and the cases of equality for
the upper and lower bounds in Conjecture~\ref{bipartite-graph-bound-conjecture}
generally, we verify here that the lower bound holds for the bipartite
graphs $\Dbip{X}{Y}$, using Corollary~\ref{betti-number-corollary}.
The crux is the following lemma.

\begin{lemma}
\label{lower-bound-crux}
Let $D$ be a shifted skew diagram, and $X, Y$ ordered subsets.

If $\Dbip{X}{Y}$ has no empty rows, then there exists a subset $Y' \subseteq Y$
for which $\Dbip{X}{Y'}$ is spherical and $\rect(\Dbip{X}{Y'})=|Y'|$.

More generally, if $\Dbip{X}{Y}$ has at least $k$ cells in every row, then for
every $j$ in the range $1 \leq j \leq k$, there are at least $\binom{k}{j}$
different choices of subsets $Y' \subseteq Y$ for which $\Dbip{X}{Y'}$ is spherical
and $\rect(\Dbip{X}{Y'})=|Y'|-j+1$.
\end{lemma}
\begin{proof}
For the first assertion, one can give an algorithm to find $Y'$.
Initialize $Y'$ to be empty.  Without loss of generality one can assume that $\Dbip{X}{Y}$
has only one cell in its first row-- simply remove all columns of $Y$ that
intersect the top row except for the longest such column, and one will still
have no empty rows in $\Dbip{X}{Y}$.  Add to $Y'$ the index $y_{n}$ of this
unique column intersecting the top row, which is now the rightmost column of
$\Dbip{X}{Y}$.  Note also that $\Dbip{X}{Y}$ has a rectangular decomposition
that begins with the full rectangle having this single column $y_j$.
As in the algorithm for rectangular decomposition, replace $\Dbip{X}{Y}$ with its restriction to the
rows and columns disjoint from this first full rectangle, which is again
a diagram with no empty rows by construction, and repeat the process until $X$ is
empty.

For the second assertion, one finds $\binom{k}{j}$ different sets $Y'$ by a similar
algorithm.  Initialize $Y'$ to be empty.  Without loss of generality
one can assume that $\Dbip{X}{Y}$ has exactly $k$ cells in its first row-- simply remove all columns of $Y$ that
intersect the top row except for the longest $k$ of them, and this will preserve the
property of every row having at least $k$ cells.  For each $j$-element subset $Y_0$ of
these $k$ columns that intersect the top row, add $Y_0$ to $Y'$, and we will continue the
the algorithm to produce a $Y'$ for which $\Dbip{X}{Y'}$ has
$|Y'|-j+1$ full rectangles in its rectangular decomposition.  As a first step,
remove the other $k-j$ columns that intersect the top row from $Y$, and note that
$\Dbip{X}{Y}$ has a rectangular decomposition that begins with a full rectangle containing
exactly the $j$ columns that intersect the top row.  Furthermore, the hypothesis that
each row has at least $k$ cells insures that, after replacing $\Dbip{X}{Y}$ with
its restriction to the rows and columns disjoint from this first full rectangle,
one will have a diagram with no empty rows.  Thus one can continue the algorithm from
the proof of the first assertion, adding in one more column to $Y'$ each time along with one
more full rectangle in the rectangular decomposition.  Repeating the process until $X$ is
empty, one obtains the desired $Y'$.
\end{proof}

\begin{corollary}
\label{lower-bound-valid-corollary}
For any shifted skew diagram $D$ and ordered subsets $X, Y$,
Conjecture~\ref{bipartite-graph-bound-conjecture} holds for $\Gbip{X}{Y}(D)$.
\end{corollary}
\begin{proof}
Let $X' \subseteq X$ and let $k:=\mindeg(X')$.  Then one must show that
$\beta_{i,X',\bullet}(I(\Gbip{X}{Y}(D))) \geq \binom{k}{i-|X'|+2}$ for
$|X'| < i+2$.

Without loss of generality, $X'=X$.  Let $j:=i-|X'|+2$,
so that $1 \leq j \leq k$, and Lemma~\ref{lower-bound-crux} implies
that there are at least $\binom{k}{j}$ different $Y' \subset Y$
for which $\beta_{|X|+|Y'|-(|Y'|-j)-2,X \sqcup Y'}(I(\Gbip{X}{Y}(D)))=1$.
In other words, $\beta_{|X|+j-2,X,\bullet} \geq \binom{k}{j}$.
Substituting $j:=i-|X'|+2$ gives the desired lower bound.
\end{proof}

\section{EPILOGUE: Further questions}
\label{sec-further}

We conclude with a few questions motivated
by our results. The first ones concern extensions of our results in which we explicitly describe the minimal free resolution of certain ideals.

\begin{question}
\label{skew-questions}
Let $K' \subset K$ be two nested $d$-uniform hypergraphs with both $I(K),I(K')$ squarefree
strongly stable, so that $I(K \setminus K')$ is what we called earlier a {\it skew squarefree
strongly stable} ideal.
\begin{enumerate}
\item[(i)]
Are the multigraded Betti numbers for the ideals $I(K \setminus K')$ and
$I(F(K \setminus K'))$ independent of the field $k$?
\item[(ii)]
Is there a combinatorial recipe like the rectangular
decomposition that allows one to compute them?
\item[(iii)]
Are the Betti numbers for $I(K \setminus K')$ obtained from those of $I(F(K \setminus K'))$
by specialization, as in Theorem~\ref{specialization-theorem}?
\item[(iv)]
Can all these be proven via cellular resolutions, as in Theorem~\ref{cellular-resolutions-theorem}?
\item[(v)]
Can one at least do (iv) in the case of $I(\Gbip{X}{Y}(D))$
with $D$ a shifted skew diagram?  This would mean finding a regular CW-complex whose cells
are indexed by the spherical subdiagrams of a given shifted skew diagram.
\end{enumerate}
\end{question}

Finally, we come back to lower bounds.

\begin{remark}
  Remark \ref{rem-counter-ex} still allows for the possibility that among the monomial ideals generated in one degree and with a fixed number of minimal generators there is one ideal that has the smallest total Betti numbers. Is this at least true for ideals that are generated in degree two? If so, it would be very interesting to describe the ideals that attain the smallest Betti numbers.
\end{remark}

\begin{question}
Can one formulate a reasonable extension of
Conjecture~\ref{bipartite-conjecture} that applies to squarefree
monomial ideals generated in degrees $d \geq 2$, presumably
parametrized by $d$-uniform hypergraphs with some multipartiteness
property?
\end{question}

\section{Appendix}
\label{appendix}

\subsection{On the topological  types of $\Delta(I(G))$}
\label{arbitrary-homotopy-type-subsection}

The goal is to observe that the simplical complexes $\Delta(I(G))$ associated
to edge ideals $I(G)$ of graphs $G$ can
have arbitrary homeomorphism type, and when $G$ is bipartite
they can have the homotopy type of an arbitrary suspension.

The following is simply the well-known observation that the {\it
first barycentric subdivision} of a simplicial complex (see e.g.
\cite[\S 15]{Munkres}) is always  a flag (clique) complex, that is,
of the form $\Delta(I(G))$.

\begin{proposition}
\label{clique-complex-homeomorphism-types} For any finite simplicial
complex $\Delta$ there exists some  graph $G$ with
$\|\Delta(I(G))\|$ homeomorphic to $\|\Delta\|$.
\end{proposition}
\begin{proof}
Let $G$ have vertex set $V$ equal to the collection of all
nonempty simplices of $\Delta$, and an edge between two of them
if the corresponding simplices of $\Delta$ are not included within
one another.  Then $\Delta(I(G))$ is the first barycentric subdivision
of $\Delta$, and hence their geometric relations are homeomorphic.
\end{proof}

Among bipartite graphs $G$, one does not achieve every
homemorphism type for $\Delta(I(G))$, and not even every
homotopy type.  However, it is easy to say exactly what homotopy types one
can achieve, namely all suspensions.  In particular, one
can easily have torsion in the homology of $\Delta(I(G))$,
and hence dependence of the Betti numbers
$\beta_i(I(G))$ on the choice of the field coefficients $k$.

\begin{proposition}
\label{bipartite-complex-homotopy-types} For any finite bipartite
graph $G$, the geometric realization $\|\Delta(I(G))\|$ is homotopy
equivalent to the suspension of the geometric realization
$\|\Delta\|$ of  some finite simplicial complex $\Delta$.

Conversely, for any finite simplicial
complex $\Delta$, one can find a bipartite graph $G$
for which $\|\Delta(I(G))\|$ is homotopy equivalent to
the suspension of $\|\Delta\|$.
\end{proposition}
\begin{proof}
Let $G$ be a bipartite graph on vertex set $X \sqcup Y$. Let $T$ be
the geometric realization of $\|\Delta(I(G))\|$ in $\R^{|X|+|Y|}$
where the vertices corresponding to $X$ (resp. $Y$) are sent to
standard basis vectors in the first $|X|$ (resp. last $|Y|$)
coordinates, and simplices are embedded piecewise-linearly with
these vertices. Define $f:\R^{|X|+|Y|} \rightarrow \R$ to be the
linear map which sums the last $|Y|$ coordinates of the vector, so
that $f(T) \subseteq [0,1]$.

One can write $T = T_X \cup T_Y$ where
$$
\begin{aligned}
 T_X &:= T \cap f^{-1}\left[0,\frac{1}{2}\right] \\
 T_Y &:= T \cap f^{-1}\left[\frac{1}{2},1\right] \\
 T_X \cap T_Y &= T \cap f^{-1}\left(\frac{1}{2}\right)
\end{aligned}
$$
and note that there are straight-line homotopies that deformation
retract $T_X, T_Y$ onto the simplices $X ( =T \cap f^{-1}(0))$ and
$Y( =T \cap f^{-1}(1))$.  Hence $T_X, T_Y$ are contractible, and so
their union $T$ is homotopy equivalent  to the suspension of their
intersection $T \cap f^{-1}(\{\frac{1}{2}\})$ by
Lemma~\ref{wedge-lemma} below. This intersection comes equipped with
a regular CW-decomposition having cells
$$
\left\{ \sigma \cap f^{-1}\left(\frac{1}{2}\right) : \sigma \in \Delta(I(G)) \right\}.
$$
Since every regular CW-complex is homeomorphic to a finite simplicial complex (e.g. its own barycentric
subdivision), the first assertion is proven.

For the converse, start with a finite simplicial complex $\Delta$.
Create a bipartite graph $G$ on vertex $X \sqcup Y$, where $X$ is
the collection of vertices of $\Delta$ and $Y$ is the collection of
facets of $\Delta$, with an edge $\{x,y\}$ if $x$ corresponds to a
vertex of $\Delta$ that does not lie on the facet of $\Delta$
corresponding to $y$.  It is not hard to see that the maximal
simplices of $\Delta(I(G))$ (other than $X, Y$) are exactly the sets
$X' \sqcup Y'$ of the following form:  $Y'=\{F_1,\dots,F_r\}\neq
\varnothing$ indexes a collection of facets whose intersection  $F_1
\cap \cdots \cap F_r$ is $X'$, and $X' \neq \varnothing$.

We claim that $T:=\|\Delta(I(G))\|$ has $T_X \cap T_Y= T \cap
f^{-1}(\{\frac{1}{2}\})$ homotopy equivalent to $\|\Delta\|$, and
hence $T$ is homotopy equivalent to the suspension of $\|\Delta\|$
by the first part of the proof. To see the claim, one can exhibit
{\it good coverings} (that is, ones in which all intersections of
the covering sets are either empty or contractible) of $\|\Delta\|$
and of $T \cap f^{-1}(\{\frac{1}{2}\})$ that  have isomorphic
nerves.

  Cover $\|\Delta\|$ by the simplices which are achieved
as intersections $F_1 \cap \cdots \cap F_r$ of nonempty collections of facets
$F_i$.  Since each facet $F$ is itself such an intersection, this covers
$\|\Delta\|$.  Because intersections of simplices in a simplicial complex
are always empty or other simplices, this is a good covering.

  Cover $T \cap f^{-1}(\{\frac{1}{2}\})$ by the
polyhedral cells $\sigma \cap f^{-1}(\frac{1}{2})$ as $\sigma$ runs
through all the sets $X' \sqcup Y'$ (described above) that give
facets of $\Delta(I(G))$ other than $X, Y$. It is not hard to see
that these cells intersect in the same fashion as their
corresponding simplices $\sigma$ intersect, hence their
intersections are always of the form $\tau \cap
f^{-1}(\frac{1}{2})$ for some  simplex $\tau$. Since this
intersection set  $\tau \cap  f^{-1}(\frac{1}{2})$ is always empty
or contractible, this is a good covering.

  The above analyses of intersections of the covering sets
also shows that the nerves of the two covers are identical. Hence
the two spaces $\|\Delta\|, T \cap f^{-1}(\{\frac{1}{2}\})$ that
they cover are both homotopy equivalent to the nerve of the cover by
the usual Nerve Lemma  \cite[Theorem 10.6]{Bjorner}.
\end{proof}

\subsection{A wedge lemma}
\label{wedge-lemma-subsection}

The goal is here to state and prove the following commonly-used
wedge lemma, whose special cases were used in two proofs above.
It is a very special case of a variation on a lemma of Bj\"orner, Wachs
and Welker, alluded to in one of their remarks;
see \cite[Lemma 7.1, Remark 7.2]{BjornerWachsWelker}.

\begin{lemma}
\label{wedge-lemma}
Let $X,Y$ be two subspaces of a topological space,
and assume that the inclusion maps $X \cap Y \hookrightarrow X,Y$
are both cofibrations, and both homotopic to a constant map.

Then the union $X \cup Y$ is homotopy equivalent to the
wedge $X \vee Y \vee \Susp( X \cap Y )$,
where here $\Susp$ denotes suspension.
\end{lemma}

\noindent
In the situations where we need this lemma,
$X \cap Y, X,$ and $Y$ are all subcomplexes of a $CW$-complex,
and hence the cofibration hypothesis always holds.  Furthermore, we can take
advantage of the fact that if $Z \hookrightarrow Z'$ is an
inclusion of a subcomplex $Z$ in a $CW$-complex $Z'$, it is homotopic to
a constant map if either $Z$ or $Z'$ is contractible, a hypothesis
that holds in the two special cases where we wish to apply the lemma:
\begin{enumerate}
\item[$\bullet$]
If $X$ and $Y$ are contractible,
then $X \cup Y$ homotopy equivalent to $\Susp( X \cap Y )$
(used in the proof of Lemma~\ref{unique-facet-topology-lemma} below).
\item[$\bullet$]
If $X \cap Y$ and $Y$ are contractible,
then  $X \cup Y$ is homotopy equivalent to $X$
(used in the proof of Proposition~\ref{bipartite-complex-homotopy-types}
above).
\end{enumerate}

\begin{proof}(of Lemma~\ref{wedge-lemma}; cf. proof of  \cite[Lemma 7.1]{BjornerWachsWelker})
The inclusions $X \cap Y \hookrightarrow X,Y$ give rise to a
{\it diagram of spaces} ${\mathcal D}$ over the $3$-element poset $Q$ that
has two maximal elements corresponding to $X, Y$ and one
minimum element corresponding to $X \cap Y$.  By \cite[Corollary 2.4]{BjornerWachsWelker},
the union $X \cup Y$ is homotopy equivalent to the {\it homotopy colimit}
of this diagram $\hocolim {\mathcal D}$.

On the other hand, there is another
diagram of spaces ${\mathcal E}$ over the same poset $Q$ in which the
inclusions  $X \cap Y \hookrightarrow X,Y$ are replaced by the constant maps to which they
are homotopic.  Then \cite[Lemma 2.1]{BjornerWachsWelker} implies
that $\hocolim {\mathcal D}$ and $\hocolim {\mathcal E}$ are homotopy equivalent.

Finally, \cite[Lemma 2.2]{BjornerWachsWelker} implies that $\hocolim
{\mathcal E}$ is homotopy equivalent to $X \wedge Y \wedge
\left(\sphere^0 * (X \cap Y)\right)$, where here $\sphere^0$ is a
zero-sphere (that is, two disjoint points) and $\star$ denotes the
topological join.  Since $\sphere^0 *  (X \cap Y) =\Susp (X \cap
Y)$, this completes the proof.
\end{proof}

The following topological lemma was an essential point in the proof
of Theorem~\ref{cellular-resolutions-theorem}.

\begin{lemma}
\label{unique-facet-topology-lemma} Let $\cC$ be a polytopal complex
and $v$ a vertex in $C$ that  lies in a unique facet $P$, and assume
that $P$ has strictly positive dimension.

Then the vertex-induced subcomplex $\cC \setminus \{v\}$,  obtained
by deleting $v$ and all faces that contain it, is homotopy
equivalent to $\cC$.
\end{lemma}
\begin{proof}
Let $P$ be the unique facet of $\cC$ containing $v$, and $P \setminus \{v\}$ the
polytopal complex whose maximal faces are the codimension one faces of $P$ not
containing $v$.  As topological spaces, one has
$$
\begin{array}{cccc}
\cC &= (\cC \setminus \{v\}) & \cup & P \\
P \setminus \{v\} &= (\cC \setminus \{v\}) &\cap & P.
\end{array}
$$
Since $P$ is convex and hence contractible, it then suffices by Lemma~\ref{wedge-lemma}
to show that $P \setminus \{v\}$ is contractible.  In fact,
polytopal complexes of the form $P \setminus \{v\}$ are
even known to be homeomorphic to a ball of dimension $\dim(P)-1$:
one can find a {\it shelling order} on the codimension one
faces of $P$ in which the maximal faces of $P \setminus \{v\}$ appear as an initial segment
\cite[Example 4.7.15 and Proposition 4.7.26(ii)]{OMbook} and \cite[Corollary 8.13]{Ziegler}.
\end{proof}

\subsection{A collapsing lemma}
\label{collapsing-lemma-subsection}

The goal here is Lemma~\ref{topology-lemma} below, which was used in
Section~\ref{sec-homotopy} to show that excess cells in diagrams of shifted skew shapes can be removed
without altering the homotopy type of their associated simplicial complexes.

For this we first recall a central notion from simple homotopy theory \cite{Cohen}.

\begin{definition}
Given two nested simplicial complexes
$\Delta \subset \Delta'$, say that $\Delta$ is obtained from $\Delta'$ by
an {\it elementary collapse} if $\Delta=\Delta'\setminus \{G,F\}$
where $F$ is a facet (maximal face) of $\Delta$
and $G$ is subface of $F$ that lies in no other faces of $\Delta$.  It is not hard
to see that this implies $\Delta$ is a strong deformation retract of $\Delta'$,
and hence that they are homotopy equivalent.  The notion of {\it simple homotopy equivalence}
is the equivalence relation $\sim$ on simplicial complexes generated by the
elementary relations $\Delta \sim \Delta'$ whenever the two complexes are
related by an elementary collapse.
\end{definition}

The following proposition is the straightforward observation that
the operation of canonical Alexander duality $\Delta \mapsto \Delta^\vee$
from Definition~\ref{Alexander-dual} (anti-)commutes with elementary collapses;  this was perhaps observed first by Kahn, Saks and Sturtevant \cite{KSS}.

\begin{proposition}
\label{prop-KSS}
Let $\Delta, \Delta'$ be two simplicial complexes on the same vertex set.
Then $\Delta'$ is obtained from $\Delta$ by an elementary collapse if and
only if $\Delta^\vee$ is obtained from $(\Delta')^\vee$ by an elementary collapse.
\end{proposition}

We also recall here the simplest way in which a simplicial complex can be contractible,
namely when it has a cone vertex.

\begin{definition}
A vertex $v$ in a simplicial complex $\Delta$ is called a {\it cone vertex} $v$ if
every face $F$ in $\Delta$ either contains $v$ or has $v \sqcup F$.
\end{definition}

\begin{lemma}
\label{topology-lemma}
Let $\Delta \subset \Delta'$ be a pair of nested simplicial complexes.
Assume that $\Delta'$ is obtained from $\Delta$ by adding one new
facet $F$ for which the intersection subcomplex $2^F \cap \Delta$ has
a cone vertex.

Then there is a sequence of elementary collapses
from $\Delta'$ down to $\Delta$, so that $\Delta, \Delta'$
have the same (simple) homotopy type.

Furthermore, $\Delta^\vee$ has the same
(simple) homotopy type as $(\Delta')^\vee$.
\end{lemma}
\begin{proof}
Let $v$ be a cone vertex for the subcomplex $2^F \cap \Delta$.
Order the subfaces $F_1,F_2,\ldots,F_s$ of $F$ not lying in $\Delta$ that do not contain $v$, in any
order from largest to smallest that respects the partial ordering by (reverse) inclusion.
Then the following pairs of faces $(G,F)$ give
a sequence of elementary collapses starting from $\Delta'$:
$$
(F_1,F_1 \cup \{v\}),
(F_2,F_2 \cup \{v\}), \ldots, (F_s,F_s \cup \{v\}).
$$
The result at the end of these collapses is $\Delta$.

The assertion for $\Delta^\vee$ and $(\Delta')^\vee$ then
follows from Proposition~\ref{prop-KSS}.
\end{proof}

\subsection{A polarization lemma}
\label{polarization-appendix}

The goal here is a Lemma~\ref{polarization-lemma} about polarizations,
which is well-known (e.g., cf. \cite[Exercise 3.15]{MillerSturmfels}).
However, we have stated it here in the form most convenient for our use, and included a proof for
the sake of completeness.

Let $S$ be a polynomial algebra over a field $k$, and
$I \subset S$ a homogeneous ideal with respect to the standard
$\Z$-grading.  Then $I$ or $S/I$ or any finitely generated $S$-module $M$ has
a minimal free resolution $\cF$ by free $S$-modules $\cF_i$.
Minimality of such a resolution is equivalent to having all entries
in the matrices defining the maps be homogeneous of positive degree.
Recall that the graded Betti number $\beta^S_{ij}(M):=\dim_k \Tor_i^S(M,k)_j$ is
the same as the number of homogeneous basis elements of degree $j$
in $\cF_i$ for any such minimal free resolution $\cF$.

\begin{lemma}
\label{polarization-lemma}
Let $S$ be a polynomial algebra over a field, and $I$ a homogeneous ideal of $S$.
Given $\theta$ in $S$ a non-zero element of degree $1$,
let $\bar{S}:=S/(\theta)$, another polynomial algebra with standard grading.  Set $\bar{I}:= (I + (\theta))/(\theta)$, a
homogeneous ideal of $\bar{S}$.

Then the following are equivalent:
\begin{enumerate}
\item[(i)] Any minimal free resolution $\cF$ for $I$ as an
$S$-module has the property that the specialized complex
$\overline{\cF}:=\bar{S} \otimes_S \cF$ in which one ``mods out $\theta$''
gives a minimal free resolution for $\bar{I}$ as a $\bar{S}$-module.

\item[(ii)] $\beta^{\bar{S}}_{ij}(\bar{I}) = \beta^{S}_{ij}(I)$ for all $i,j$.

\item[(iii)] $\Hilb(\bar{S}/\bar{I},t)=(1-t)\Hilb(S/I,t).$

\item[(iv)]  $\theta$ acts as a non-zero divisor on $S/I$.
\end{enumerate}
\end{lemma}
\begin{proof}
The implication (i) implies (ii) is clear.
For  (ii) implies (iii), recall that taking the Euler characteristic in each graded component
of the minimal resolution $\cF \rightarrow S \rightarrow S/I \rightarrow 0$ gives
$$
\Hilb(S/I,t):=\Hilb(S,t) \sum_{i} (-1)^i \beta^S_{ij}(S/I) t^j.
$$
Since $\Hilb(\bar{S},t)=(1-t)\Hilb(S,t)$, the assertion (iii) follows.

For (iii) implies (iv), recall that for any $S$-module $M$, the exact sequence
$$
0 \rightarrow \Ann_M(\theta) (-1) \rightarrow M(-1)
\overset{\cdot \theta}{\rightarrow} M \rightarrow M/\theta M \rightarrow 0
$$
shows that
$$
\Hilb(M/\theta M,t)=(1-t)\Hilb(M,t) + t \Hilb(\Ann_M(\theta),t).
$$
Hence $\theta$ is a non-zero-divisor on $M$ exactly when $\Hilb(M/\theta M,t)=(1-t)\Hilb(M,t)$.
Apply this to $M=S/I$.

For (iv) implies (i), first argue the vanishing $\Tor_i^S(S/I,S/(\theta))=0$ for $i > 0$
as follows.  Since $\theta \neq 0$ it is a non-zero-divisor on $S$, and one has the $S$-resolution
$$
0 \rightarrow S(-1) \overset{\cdot \theta} \rightarrow S \rightarrow S/(\theta) \rightarrow 0
$$
for $S/(\theta)$, which one can tensor over $S$ with $S/I$ to obtain
the complex
$$
0 \rightarrow S/I(-1) \overset{\cdot \theta} \rightarrow S/I \rightarrow S/(I+(\theta)) \rightarrow 0.
$$
Taking homology (with the $S/(I+(\theta))$ term omitted) computes
the relevant $\Tor$, and the vanishing follows because  the
assumption of (iv) implies this complex is exact.

Now note that since $\cF \rightarrow S \rightarrow S/I \rightarrow 0$
is a resolution of $S/I$, when one tensors over $S$ with $\bar{S}$ to obtain
$$
\overline{\cF} \rightarrow \bar{S} \rightarrow S/(I+(\theta)) \rightarrow 0,
$$
the homology of this complex (with the $S/(I+(\theta))$ term omitted)
will compute the same $\Tor$.  The vanishing result for $\Tor$ then
implies this complex is exact.  Hence $\overline{\cF}$ resolves $\bar{I}$.  In fact, it will
be a {\it minimal} resolution because tensoring over $S$ with $\bar{S}$ preserves the property
that all matrix entries in the maps are of positive degree.
\end{proof}


\end{document}